\documentclass[twoside,a4paper,reqno,11pt]{amsart} 

\usepackage{amsfonts, amsbsy, amsmath, amssymb, latexsym}
\usepackage{mathrsfs,array}
\usepackage[top=30mm,right=30mm,bottom=30mm,left=30mm]{geometry}
\usepackage{stmaryrd}
\usepackage{bm}

\usepackage{hyperref}
\hypersetup{citecolor=red, linkcolor=blue, colorlinks=true}

\def\a{\alpha} 
\def\b{\beta}

\def\e{\epsilon} 
\def\O{\Omega} 
 
\def\l{\lambda}
\def\M{{\mathcal{M}}} 
\def\DD{{\mathcal{D}}} 
\def\PPP{{\mathcal{P}}}

\def\soc{{\rm soc}} 
\def\Inc{{\rm Inc}}

\def\Sym{{\rm Sym}} 
\def\F{\mathbb{F}}

\def\D{\Delta}

\def\G{\Gamma} 

\def\no{\noindent}

\def\Aut{{\rm Aut}} 
\def\e{\epsilon} 
\def\la{\langle} 
\def\ra{\rangle} 

\def\a{\alpha} 
\def\b{\beta}

\def\e{\epsilon} 
\def\O{\Omega} 
 
\def\l{\lambda}
\def\M{{\cal M}}

\def\soc{{\rm soc}} 
\def\F{\mathbb{F}}

\def\D{\Delta}

\def\G{\Gamma} 
\def\M{\mathcal M}
 
\def\no{\noindent}

\newtheorem{thm}{Theorem}[section] 
\newtheorem{prop}[thm]{Proposition}
\newtheorem{definition}[thm]{Definition}
\newtheorem{hyp}[thm]{Hypothesis}
\newtheorem{lem}[thm]{Lemma} 
\newtheorem{lemma}[thm]{Lemma} 
\newtheorem{coroll}[thm]{Corollary} 
\newtheorem{remark}[thm]{Remark} 
\newtheorem{construction}[thm]{Construction} 
 
\newtheorem{prob}[thm]{Problem} 
\numberwithin{table}{section}
\numberwithin{equation}{section}

\newcommand\GL{\operatorname{GL}}
\newcommand\AGL{\operatorname{AGL}}
\newcommand\ASL{\operatorname{ASL}}
\newcommand\PSL{\operatorname{PSL}}
\newcommand\AGaL{\operatorname{A\Gamma L}}
\newcommand\PGL{\operatorname{PGL}}
\newcommand\PG{\operatorname{PG}}
\newcommand\PGaL{\operatorname{P\Gamma L}}
\newcommand\GaL{\operatorname{\Gamma L}}

\newcommand\SU{\operatorname{SU}}

\newcommand\SO{\operatorname{SO}}

\newcommand\POm{\operatorname{P\Omega}}
\newcommand\PSO{\operatorname{PSO}}
\newcommand\PSU{\operatorname{PSU}}
\newcommand\PGU{\operatorname{PGU}}

\newcommand\Sp{\operatorname{Sp}}
\newcommand\PSp{\operatorname{PSp}}
\newcommand\SL{\operatorname{SL}}
\newcommand\Sz{\operatorname{Sz}}

\newcommand\Soc{\operatorname{soc}}
\newcommand\Alt{\operatorname{Alt}}

\newcommand\Out{\operatorname{Out}}

\newcommand\hatB{\widehat{B}}
\newcommand\BB{\mathcal{B}}

\newcommand\Or{\operatorname{O}}
\newcommand\PGSp{\operatorname{PGSp}}

\begin{document}
	
	\title{Maps, simple groups, and arc-transitive graphs}
	
	\author{Martin W. Liebeck}
	\address{M.W. Liebeck, Department of Mathematics,
		Imperial College, London SW7 2BZ, UK}
	\email{m.liebeck@imperial.ac.uk}
	
	\author{Cheryl E. Praeger}
	\address{C. E. Praeger, Department of Mathematics and Statistics,
		The University of Western Australia, Perth, WA 6009, Australia}
	\email{cheryl.praeger@uwa.edu.au}

	\no \footnote {
		\emph{Funding:} This work was supported by the Australian Research Council [Grants DP230101268; DP160102323].}
	
	\begin{abstract} We determine all factorisations $X=AB$, where $X$ is a finite almost simple group and $A,B$ are core-free subgroups such that $A\cap B$ is cyclic or dihedral. As a main application, we classify the graphs $\G$ admitting an almost simple arc-transitive group $X$ of automorphisms, such that $\G$ has a 2-cell embedding as a map on a closed surface admitting a core-free arc-transitive subgroup $G$ of $X$. We prove that apart from the case where $X$ and $G$ have socles $A_n$ and $A_{n-1}$ respectively, the only such graphs are the complete graphs $K_n$ with $n$ a prime power, the Johnson graphs $J(n,2)$ with $n-1$ a prime power, and 14 further graphs. In the exceptional case, we construct infinitely many graph embeddings.
	\end{abstract} 
	
	\date{}
	\maketitle

	\begin{center}
		{\small {\it Dedicated to the memory of our friend and colleague Tony Gardiner}}
	\end{center}
	
	\vspace{4mm}
	
	\section{Introduction}\label{s:intro}
	
	A \emph{(simple) graph} $\Gamma=(\Omega,E)$ consists of a set $\Omega$ of vertices, and a set $E$ of unordered pairs of vertices, which we call edges. The graphs we consider are all connected and have automorphism group $\Aut(\G)$ transitive on \emph{arcs}, that is ordered pairs of adjacent vertices. A \emph{map} $\M = (\O,E,F)$ is a $2$-cell embedding of a simple graph $\G=(\Omega,E)$ in a closed surface with face set $F$; and the automorphism group $\Aut(\M)$ of $\M$ is the subgroup of $\Aut(\G)$ consisting of those graph automorphisms which extend to self-homeomorphisms of the supporting surface (see \cite[Section 6]{J15}).  Avoiding degenerate situations we assume that each edge of $\G$ is adjacent with exactly two faces, in which case the only element of $\Aut(\M)$ fixing a \emph{flag} (a mutually incident vertex-edge-face triple) is the identity. A consequence is that, if $G\leq \Aut(\M)$ and $G$ is arc-transitive on $\G$, then a vertex-stabiliser $G_{\a}$ is cyclic or dihedral, and $G_{\a\b}\leq C_2$ for an arc $(\a,\b)$. More details are given in Section~\ref{prelim}. 
	
	A group theoretic approach to studying graph embeddings was pioneered by Tony Gardiner, to whom this paper is dedicated. In 1990 Tony and his colleagues \cite{GNSS} showed that the problem of embedding a graph $\G$ as a \emph{flag-regular map} $\M$ (that is, where $\Aut(\M)$ is regular on flags) could be posed as an equivalent problem concerning subgroups of $\Aut(\G)$.   We address here a more general problem for the case where $\Aut(\M)$ is assumed to be transitive on arcs.  
	Such embeddings have been studied extensively for some particular families of graphs. For example, edge-transitive embeddings of the complete graphs $K_n$ have been classified completely by Jones \cite[Theorem 1.8]{J21} building on constructions and classifications  in \cite{B71, Ja83, JJ85}. In particular, $n$ must be either 6 or a prime-power, and in both cases arc-transitive embeddings exist. Similarly flag-regular embeddings of complete bipartite graphs $K_{n,n}$, $n$-dimensional cubes $Q_n$, and merged Johnson graphs have  also been classified (see \cite{CCDKNW, J05, J10} and references therein), and graphs from several of these families arise in our work. 
	
	Our focus in this paper is on embeddings of connected simple graphs $\Gamma$ that have an arc-transitive {\it almost simple} group of automorphisms $X$, with simple socle $X_0$. There are many such graphs of interest, and of course the concept is very general: it just amounts to saying that $\G$ is a connected self-paired orbital graph of $X$.
	Let $\M$ be an embedding of  $\Gamma$, and assume that $G \le \Aut(\M) \cap X$ is also arc-transitive. Then either $G$ contains the socle $X_0$ or it does not. In the first case there is a great deal of information about the structure and action of $G$ and $X$ (see for example Proposition \ref{cri}). We shall focus on the second case, in which $G$ does not contain the socle $X_0$. We prove that apart from the case where the socles of $X$ and $G$ are $A_n$ and $A_{n-1}$ respectively, the possibilities for the graph $\G$ are very restricted: namely, $\G$ must be a complete graph $K_n$ with $n$ a prime power, a Johnson graph $J(n,2)$ with $n-1$ a prime power, or one of 14 further graphs. In the exceptional case $(A_n,A_{n-1})$, we construct infinitely many graph embeddings.
	
	The strategy of our proof is based on the theory of group factorisations. In Proposition \ref{cri}, we show that the above situation pertains (namely $X \le \Aut(\G)$, $G \le \Aut(\M) \cap X$ both arc-transitive), precisely when the following conditions hold for $\G$ and the groups $X$ and $G$, where $(\a,\b)$ is an arc:
	\begin{itemize}
		\item[{\rm (i)}] $X = GX_\a$ and $X_\a = G_\a X_{\a\b}$;
		\item[{\rm (ii)}] $G_\a$ is cyclic or dihedral, and $G_{\a\b} = 1$ or $C_2$;
		\item[{\rm (iii)}] there is an involution $g \in N_G(X_{\a\b})$ such that $\la X_\a,g \ra = X$.
	\end{itemize}
	Writing $A=G$, $B = X_\a$, parts (i) and (ii) lead to a factorisation $X=AB$ with $A\cap B$ cyclic or dihedral; we call such an expression a {\it cyclic/dihedral} factorisation of $X$. Motivated by this application (and potentially others), in this paper we classify all such factorisations of almost simple groups $X$. Using this classification, we then find all examples satisfying properties (i)-(iii), and hence classify all the corresponding arc-transitive embeddings. When $X_0$ is an alternating group $A_n$, there are infinitely many arc-transitive graphs with such embeddings (see Theorem \ref{t:maps-an}), but when $X_0$ is a simple group of Lie type, there are precisely five graphs, and when $X_0$ is a sporadic simple group there are precisely four more (see Theorems \ref{t:classical-maps} and \ref{spormaps}).
	
	We now state our main results: first the cyclic/dihedral factorisations, and then the consequences for map embeddings. The factorisation results refer to tables, all of which can be found in Section \ref{tables} at the end of the paper.

	\subsection{Factorisations of $A_n$ and $S_n$}
	
	Let $X=A_n$ or $S_n$, where $n\geq5$, and suppose that $X=AB$ with $A, B$ core-free subgroups of $X$ such that $A\cap B$ is cyclic or dihedral. Let $\Omega=\{1,\dots,n\}$ with $X$ acting naturally. 
	All proper core-free subgroups $A, B$  such that $X=AB$ are given in \cite{WW} (or see \cite[Theorem D]{LPS90}). Apart from some exceptional cases with $n\leq 10$, all examples, possibly with $A$ and $B$ interchanged, satisfy
	\begin{equation}\label{e:gencase}
		A_{n-k}\unlhd A\leq S_{n-k}\times S_k, \ \mbox{with $B$ $k$-homogeneous on $\Omega$, for some $k\in\{1,\dots,5\}$.}
	\end{equation}
	We refer to this as the `general case'.
	
	\begin{thm}\label{t:Anfactns}
		Let $X=A_n$ or $S_n$, or $A_6\leq X\leq \Aut(A_6)$ if $n=6$, and suppose that $X=AB$ with $A,B$ core-free and $A\cap B$ cyclic or dihedral.
		
		(a) In the general case where \eqref{e:gencase} is valid for $A$, one of the following holds: 
		\begin{enumerate}
			\item[(a.1)] $k=1$ and,  for some $x\in\Omega$, either
			\begin{enumerate}
				\item[(i)] $(X,A)=(S_n,S_{n-1})$ or $(A_n,A_{n-1})$: here $A=X_x$ and  $B$ is a transitive subgroup of $X$ with $B_x$ cyclic or dihedral; or
				\item[(ii)] $(X,A) = (S_n, A_{n-1})$: here $A_n=A(B\cap A_n)$, $B\cap A_n$ is a transitive subgroup of $A_n$, $B$ contains an odd permutation, and $(B\cap A_n)_x$ is cyclic or dihedral.
			\end{enumerate}
			\item[(a.2)] $k\in\{2,3,4,5\}$, $B^\Omega$ is $k$-homogeneous, and $(X,A,B)$ satisfy one of the Lines of Table~\ref{tab:Anfactns}, where $x, y, z,w,v$ are pairwise distinct points of $\Omega$.
		\end{enumerate}
		
		(b) In the exceptional cases, with $n\leq 10$, either $(X,A^\tau,B^\tau)$ are as in case $(a)$, for some automorphism $\tau$ of $A_n$, or $(n,A,B, A\cap B)$ are as in the table below.
	\end{thm}
	
	\[
	\begin{array}{l|l|l|l}
		\hline
		n & A & B & A\cap B \\
		\hline
		6 & 3^3.[4a_4] & 5.[b_8] & 1,\,2 \\
		8 & \AGL_3(2) & 15.[a_4] & 1,\,2,\,4 \\
		10 & (S_5 \wr S_2) \cap A_{10} & \PSL_2(8) & 2^2 \\
		\hline 
	\end{array}
	\]
	
	In the table, the notation $a_r,b_r$ means a divisor of $r$, and $[m]$ just denotes a group of order $m$.
	
	We remark that there are additional restrictions in some cases on the groups $A, B$ in Theorem~\ref{t:Anfactns} -- see Remark \ref{r:Anfactns}. Remark \ref{r:Anfactns}(c,d) shows that there exist examples of groups $B$ satisfying the conditions of parts (a.1)(i) and (a.1)(ii) of the theorem, except possibly for (a.1)(ii) when $n \equiv 1 \hbox{ mod }4$.

	\subsection{Factorisations of groups of Lie type}
	
	Our main result here classifies all cyclic/di\-hedral factorisations of almost simple groups of Lie type.
	
	\begin{thm}\label{classgen}
		Let $X$ be an almost simple group with socle $X_0$, a simple group of Lie type, and suppose that $X = AB$ with $A,B$ core-free and $A \cap B$ cyclic or dihedral. Then $X$, $A$, $B$, and $A \cap B$ are as in Tables 
		$\ref{families}$ and $\ref{excep-psl}-\ref{excep-sp}$. Conversely, for each line in the tables there exists such a factorisation for
		some almost simple group with socle $X_0$.
	\end{thm}
	
	\begin{remark} \label{classrem} 
		{\rm There are some small differences in the way we record the factorisations $X=AB$ in the tables: in Table \ref{families} we just give the possibilities for $(X_0,A\cap X_0,B\cap X_0)$; whereas in Tables \ref{excep-psl}-\ref{excep-sp}, we give
			the possibilities $(X,A,B)$ where $X$ is an almost simple group which is minimal subject to having a factorisation $X=AB$ with cyclic or dihedral intersection. In all cases the tables list possibilities up to interchanging $A$ and $B$, and up to applying automorphisms of $X$. }
	\end{remark}

	\subsection{Factorisations of sporadic groups}
	
	Here we classify the cyclic/dihedral factorisations of the sporadic groups.
	
	\begin{thm}\label{sporadic}
		Let $X$ be almost simple with socle a sporadic group $X_0$, and suppose that $X = AB$ with $A,B$ core-free and $A \cap B$ cyclic or dihedral. Then $X$, $A$, $B$, and $A \cap B$ are as in Table 
		$\ref{spor}$, and conversely, for each line in the table there exists such a factorisation.
	\end{thm}
	
	\begin{remark} \label{sporrem} 
		{\rm In Table \ref{spor}, if the socle factorises as $X_0 = A_0B_0$, we do not include factorisations of $X = X_0.2$ of the form $X=AB$, where $A\cap X_0 = A_0, B\cap X_0 = B_0$.}
	\end{remark}
	
	\subsection{Consequences for map embeddings}\label{s:1.4}
	
	We shall use our factorisation theorems to classify embeddings of graphs $\Gamma$ that have an arc-transitive {\it almost simple} group of automorphisms $X$, with simple socle $X_0$. 
	Let $\M$ be an embedding of  $\Gamma$, and assume that $G = \Aut(\M)\cap X$ is also arc-transitive and does not contain the socle $X_0$. This situation holds if and only if conditions (i)-(iii), given above, pertain for an arc $(\a,\b)$ of $\G$, which we repeat here in an equivalent but slightly altered form, noting that the involution in (iii) must interchange $\a$ and $\b$:
	
	\begin{hyp}\label{hypo} $X$ is an almost simple group with socle $X_0$, acting faithfully and transitively on a set $\O$, $G$ is a core-free subgroup of $X$, and $g \in G$ is an involution such that the following hold:
		\begin{itemize}
			\item[{\rm (i)}] for $\a \in \O$ we have $X = GX_\a$, and $G_\a$ is cyclic or dihedral;
			\item[{\rm (ii)}] $X_\a = G_\a\,X_{\a\b}$, and $|G_{\a\b}| \le 2$, where $\b = \a^g$;
			\item[{\rm (iii)}] $\la X_\a,g \ra = X$.
		\end{itemize}
	\end{hyp}
	
	In the results below, we classify all pairs $(X,G)$ satisfying Hypothesis \ref{hypo}, and as a consequence all the arc-transitive graphs with embeddings as above; the latter are just the orbital graphs on $\O$ with edge-set $\{\a,\b\}^X$, where $\b = \a^g$. It turns out that there are many such embeddings when $X_0$ is an alternating group, but only a handful when $X_0$ is classical or sporadic.
	
	We begin with our result for alternating groups. In this result $K_n$ denotes the complete graph on $n$ vertices, while $K_{n,n}\setminus n.K_2$ is the complete bipartite graph $K_{n,n}$ minus a complete matching; $J(n,2)$ is the \emph{Johnson graph}, for which the vertices are the unordered pairs from an $n$-set, two pairs being adjacent if they contain a common point.
	
	\begin{thm}\label{t:maps-an} 
		Suppose that Hypothesis $\ref{hypo}$ holds for $X, G$ and involution $g\in G$ with $\soc(X)=A_n$ for some $n\geq5$, and let $\G$ be the orbital graph on $\O$ with edge-set $\{\a,\b\}^X$, where $\b = \a^g$. Then either 
		\begin{itemize}
			\item[{\rm (i)}]  one of the lines of Table~$\ref{t:anmaps}$ or  $\ref{6711tbl}$ holds; or
			\item[{\rm (ii)}] ${\rm soc}(G) = A_{n-1}$, and there are examples for infinitely many values of $n$.
		\end{itemize}
	\end{thm}
	
	\begin{remark}\label{altrem}
		
		{\rm We construct an infinite family of examples satisfying Theorem~\ref{t:maps-an}(ii) in Section \ref{bk1prop}. In these examples we have $n = (p-1)!/2$ for certain primes $p$, and $(X,G) = (A_n,A_{n-1})$ with $X_\a \cong A_p$, $G_\a \cong C_p$. The graphs have $n!/p!$ vertices and valency $p$.}
	\end{remark}
	
	
	\begin{table}[h!]
		\caption{Infinite families of examples from $A_n$; here $p$ is a prime}\label{t:anmaps}
		\[
		\begin{array}{c|c|c|c|c|c|c|l|l}
			\hline
			n&X & X_\a & X_{\a\b} & G & G_\a & G_{\a\b} & \G & \hbox{ref} \\
			\hline
			p^f& S_n & S_{n-1} & S_{n-2} & \AGL_1(n) & \GL_1(n) & 1 & K_n &\S~\ref{ak1} \\
			2^f& A_n & A_{n-1} & A_{n-2} & \AGL_1(n) & \GL_1(n) & 1 & K_n &\S~\ref{ak1} \\
			p^f+1&S_n, A_n & (S_{n-2}\times S_2)\cap X& S_{n-3}\cap X & \PGL_2(n-1) & D_{2(n-2)}&1&J(n,2)& \S~\ref{1st}\\
			\hline
		\end{array}
		\]
	\end{table}

	
	\begin{table}[h!]
		\caption{Exceptional examples from $A_n$}\label{6711tbl}
		\[
		\begin{array}{c|c|c|c|c|c|l|l}
			\hline
			X & X_\a & X_{\a\b} & G & G_\a & G_{\a\b} & \G & \hbox{ref} \\
			\hline
			S_6, A_6 & S_5, A_5 & S_4, A_4 & \PSL_2(5) & D_{10} & 2 & K_6 &\S~\ref{ak1} \\
			S_6 & A_5 & A_4 & \PGL_2(5) & D_{10} & 2 & K_{6,6}\setminus 6.K_2&\S~\ref{ak1}\\
			S_6.2 & S_5 & S_4 & (S_3\wr S_2).2 & D_{12} & 2 & K_{6,6}&\S~\ref{cn6}\\
			S_7 & \PSL_3(2) & 7.3 & S_5 \times 2 & D_8 & 1 & \hbox{ antiflag graph of $\PG_3(2)$} & \S~\ref{xas} \\
			S_8 & \AGL_3(2) & \GL_3(2) & S_5 \times 2 & D_8 & 1 & \hbox{ antiflag graph of $\PG_3(2)$} & \S~\ref{bk3} \\
			S_{11} & \PSL_2(11) & 11.5 & S_9\times 2 & D_{12} & 1 & |V\G|=12\cdot 7!, \hbox{valency $12$}& \S~\ref{xas} \\
			\hline
		\end{array}
		\]
	\end{table}
	
	In contrast to Theorem \ref{t:maps-an}, for groups of Lie type there are just five embeddings:
	
	\begin{thm}\label{t:classical-maps}
		Suppose that Hypothesis $\ref{hypo}$ holds for $X, G$ and involution $g\in G$, and that $\soc(X)$ is a simple group of Lie type that is not isomorphic to an alternating group. Let $\G$ be the orbital graph on $\O$ with edge-set $\{\a,\b\}^X$, where $\b = \a^g$. Then $(X,G,X_\a)$ and $\G$ are as in Table $\ref{class-maps}$. Conversely, each entry in the table gives an arc-transitive embedding.
	\end{thm}
	
	\begin{table}[h!]
		\caption{Arc-transitive embeddings from groups of Lie type}\label{class-maps}
		\[
		\begin{array}{c|c|c|c|c|c|l}
			\hline
			X & X_\a & X_{\a\b} & G & G_\a & G_{\a\b} & \G \\
			\hline
			\PSL_3(2).2 & P_1 & D_8 & 7.6 & 3 & 1 & \hbox{incidence graph of ${\rm PG}_2(2)$} \\
			\PSL_3(8).6 & P_1 & \hbox{Borel} & 73.18 & 9 & 1 & \hbox{incidence graph of ${\rm PG}_2(8)$} \\
			\PSL_2(11).2 & A_5 & A_4 & 11.10 & 5 & 1 & \hbox{incidence graph of 11 point biplane} \\
			\SO_5(3) & 3^3.S_4 & 3^{1+2} & 2^4.S_5 & D_{24} & 1 & \hbox{see Subsection \ref{ex:u42}} \\
			& 3^3.A_4 & 3^{1+2} & 2^4.S_5 & D_{12} & 1 & \hbox{see Subsection \ref{ex:u42}} \\
			\hline
		\end{array}
		\]
	\end{table}
	
	Finally, for sporadic groups there are four embeddings:
	
	\begin{thm}\label{spormaps}  Suppose that Hypothesis $\ref{hypo}$ holds for $X, G$ and involution $g\in G$, and that $\soc(X)$ is a sporadic group. Let $\G$ be the orbital graph on $\O$ with edge-set $\{\a,\b\}^X$, where $\b = \a^g$. Then $(X,G,X_\a)$ and $\G$ are as in Table $\ref{class-maps}$. Conversely, each entry in the table gives an arc-transitive embedding; in the table we include the number of vertices $n$, and the valency $k$ of the graph.
	\end{thm}

	\begin{table}[h!]\label{spormaptab}
		\caption{Arc-transitive embeddings from sporadic groups}
		\[
		\begin{array}{l|l|l|l|l|l|l|l}
			\hline
			X & X_\a & X_{\a\b} & G & G_\a & G_{\a\b} & n & k\\
			\hline
			M_{24} & \PSL_2(7) & S_4 & M_{23} & 7 & 1 & 14572980 & 7 \\
			J_2.2 & (A_5\times D_{10}).2 & 5^2.4 & \PSU_3(3).2 & 12 & 1 & 1008 & 12 \\
			HS.2 & (5\times A_5).4 & 5^2.4 & M_{22}.2 & 12 & 1 & 73920 & 12 \\
			He.2 & 7^2.\SL_2(7).2 & 7^{1+2}.6 & \Sp_4(4).4 & 16 & 1 & 244800 & 16 \\
			\hline
		\end{array}
		\]
	\end{table}
	
	Note that $X_\a$ is maximal in $X$ in lines 1, 2 and 4 of Table~\ref{spormaptab}, but in line 3,  $X_\a$ has index $2$ in a maximal subgroup $5.4 \times S_5$ of $X=HS.2$. More generally, although we make no such assumption, it is very often the case that the almost simple group $X$ is vertex-primitive on $\G$, or $\G$ is bipartite and $X$ is vertex-biprimitive (that is, the index $2$ subgroup fixing the two parts of the bipartition is primitive on each part).
	
	\begin{coroll}\label{cor:vpr}
		Suppose that Hypothesis $\ref{hypo}$ holds for $X, G, g$, and that $\G$ is the orbital graph on $\O$ with edge-set $\{\a,\b\}^X$, where $\b = \a^g$. Then one of the following holds.
		\begin{itemize}
			\item[{\rm (i)}] $\G$ is a complete graph $K_n$ with $n$ a prime power, a Johnson graph $J(n,2)$ with $n-1$ a prime power, or one of the $14$ graphs in Tables $\ref{6711tbl}$, $\ref{class-maps}$, $\ref{spormaptab}$;
			\item[(ii)] $\soc(X)=A_n$ and $\soc(G)=A_{n-1}$, as in Theorem~$\ref{t:maps-an}(ii)$.
		\end{itemize}
		Moreover, in case (i), either $X$ is vertex-primitive on $\G$, or $\G$ is bipartite and $X$ is vertex-biprimitive, or $X=\SO_5(3)$  as in the last Lines of Table~$\ref{class-maps}$.
	\end{coroll}
	
	Theorem~\ref{t:maps-an}(ii) is the one case where our classification of map embeddings is not completely explicit. In this case Construction~\ref{c:k1eg} demonstrates that there are infinitely many examples $\G(p)$, and by Theorem~\ref{k1eg},  each of these graphs $\G(p)$  is non-bipartite and $X=A_n$ (where $n=(p-1)!/2$) is imprimitive on vertices. Thus  Corollary~\ref{cor:vpr}(iii) provides infinitely many non-bipartite examples where vertex-primitivity fails. 
	Moreover, we know that these are not the only examples for Theorem~\ref{t:maps-an}(ii)  and in Subsection~\ref{bk1} we discuss several others  arising from $X$-arc-transitive graphs with $X=S_8$ or $S_9$.   It would be very interesting to understand this case better, and a start would be to have more examples.
	
	\begin{prob}\label{prob1}
		Find further infinite families of examples of arc-transitive embeddings $\M$ of graphs admitting an arc-transitive subgroup of automorphisms $A_n$ or $S_n$ such that $\Aut(\M)$ is $A_{n-1}$ or $S_{n-1}$.   
	\end{prob}
	
	Another observation we draw from our results is that, although we only ask for arc-transitive embeddings, some of our examples are in fact flag-regular. We note that a $G$-arc-transitive embedding is flag-regular if and only if $G_{\a\b}=C_2$, and we can read off this property from Tables~\ref{t:anmaps}--\ref{spormaptab}. 
	
	

	\begin{coroll}\label{cor:fr}
		Suppose that Hypothesis $\ref{hypo}$ holds for $X, G, g$ and that $\G$ is the orbital graph on $\O$ with edge-set $\{\a,\b\}^X$ (where $\b = \a^g$). Suppose also that $X, G$ do not satisfy  Theorem~$\ref{t:maps-an}(ii)$.  Then either the $G$-arc-transitive embedding is arc-regular, or  one of Lines $1,2,3$ of Table~$\ref{6711tbl}$ holds and the embedding is flag-regular.
	\end{coroll}
	
	We note that each of the three exceptional flag-regular embeddings in Corollary~\ref{cor:fr} is related to the group $A_6$.
	
	\vspace{4mm}
	\subsection{Layout} We offer a brief outline of the layout of the paper. Section \ref{prelim} contains some preliminary material on graphs and maps, together with a discussion of our strategy for proving our map classification theorems \ref{t:maps-an} and \ref{t:classical-maps}. It also has a short subsection \ref{maggie} on our use of Magma computations in the proofs. In Section \ref{arc-tr} we describe some of the most interesting graphs that occur in our map classification theorems, and Section \ref{bk1prop} comprises our construction of an infinite family of graphs that admit an arc-transitive group of automorphisms $A_n$, and have an embedding as a map with arc-transitive group $A_{n-1}$. Sections \ref{ansnfactn} and \ref{classfactn} contains the proofs of our cyclic/dihedral factorisation theorems \ref{t:Anfactns} and \ref{classgen} for alternating groups and for groups of Lie type, respectively. In Sections \ref{s:anmaps} and \ref{s:liemaps} we prove the map classification theorems \ref{t:maps-an} and \ref{t:classical-maps}, and Section \ref{s:sporadic} contains our treatment of the sporadic groups. Finally, Section \ref{tables} gives the tables of cyclic/dihedral factorisations of alternating groups and groups of Lie type that occur in our factorisation theorems.
	
	\subsection{Acknowledgements}
	As mentioned, we make considerable use of Magma computations in our proofs, and we are very grateful to Eamonn O'Brien for his generous assistance with this. We also wish to thank Cai Heng Li and Binzhou Xia for making available pre-publication versions of their paper \cite{LWX}, which allowed us to greatly reduce the length of our exposition in the proof of Theorem~\ref{classgen}. 
	
	\section{Preliminaries on graphs and arc-transitive embeddings}\label{prelim}
	
	We study arc-transitive $2$-cell embeddings of connected simple graphs in closed surfaces, which we introduce as follows. For more details about graph embeddings see for example \cite{GW}, \cite{J15}, or \cite{STW}. 
	
	\medskip\noindent
	\emph{Graphs.}\quad A \emph{(simple) graph} $\Gamma=(\Omega,E)$ consists of a set $\Omega$ of vertices, and a set $E$ of unordered pairs of vertices, which we call edges, and $\Gamma$ is \emph{connected} if, for any vertices $\a, \b$, there exists a finite vertex sequence $\a_1,\dots,\a_r$ such that $\a_1=\a$, $\a_r=\b$, and each $\{\a_i,\a_{i+1}\}$ is an edge. The \emph{automorphism group} $\Aut(\G)$ of a graph $\G=(\O,E)$ consists of all permutations of the vertex set $\O$ which leave the edge set $E$ invariant. An \emph{arc} of $\G$ is an ordered pair $(\a,\b)$ of vertices such that $\{\a,\b\}\in E$, and for $G\leq \Aut(\G)$, $\G$ is said to be \emph{$G$-arc-transitive} if $G$ is transitive on the set of arcs. Note that if $\G$ is $G$-arc-transitive, then $G$ is also vertex-transitive and edge-transitive on $\G$; in particular, $\G$ is \emph{regular} in the sense that the set $\G(\a)$ of vertices adjacent to  $\a$ has the same cardinality for each vertex $\a$, called the \emph{valency} of $\G$. 
	
	\medskip\noindent
	\emph{Maps.}\quad 
	A \emph{map} $\M = (\O,E,F)$ is a $2$-cell embedding of a simple graph $\G=(\Omega,E)$ in a closed surface with face set $F$; and the automorphism group $\Aut(\M)$ of $\M$ is the subgroup of $\Aut(\G)$ consisting of those graph automorphisms which extend to self-homeomorphisms of the supporting surface (see \cite[Section 6]{J15}), sometimes also defined as the subgroup of automorphisms of $\G$ that preserve the set of \emph{flags} (mutually incident vertex-edge-face triples) \cite[Section 2]{S01}. We assume that $\G$ is connected, that $\G$ is not a simple cycle (equivalently $\G$ has valency at least $3$), and that  each edge is incident with exactly two faces.  
	The only map automorphism that fixes a flag is the identity, that is to say, $\Aut(\M)$ is semiregular on the set of flags, and we call $\M$ \emph{flag-regular} if $\Aut(\M)$ is transitive (hence regular) on flags.  Also we call $\M$ \emph{arc-transitive} if $\Aut(\M)$ is transitive on the arcs of $\G$.  Since  the stabiliser in $\Aut(\M)$ of an arc $(\a,\b)$ must preserve the two faces of $\M$ incident with the edge $\{\a,\b\}$, the arc-stabiliser $\Aut(\M)_{\a,\b}$ has order at most two, and it has order $2$ if and only if $\M$ is flag-regular. In fact, for a flag-regular map $\M$, the stabiliser $\Aut(\M)_\a$ of a vertex $\a$ cyclically permutes the $k:=|\G(\a)|$ edges incident with $\a$ and hence induces a circular ordering on $\G(\a)$; the stabiliser of an edge $\{\a,\b\}$ in this action is the arc-stabiliser $\Aut(\M)_{\a,\b}$, and this interchanges the two faces incident with $\{\a,\b\}$, and hence interchanges the vertices before and after $\b$  in the circular ordering induced on $\G(\a)$. Thus for a flag-regular map,  $\Aut(\M)_\a$ acts faithfully on $\G(\a)$ as the dihedral group $D_{2k}$ of order $2k$, and the edge-stabiliser $\Aut(\M)_{\{\a,\b\}}\cong C_2\times C_2$.  An arc-transitive map $\M$ which is not flag-regular acts regularly on the arc set, and the stabiliser $\Aut(\M)_\a$ is regular on $\G(\a)$; it is either a cyclic group $C_k$, or $k$ is even and it is a dihedral group $D_k$ (see \cite{GW}, \cite[Table I]{STW} or the summary in \cite[Table before Theorem 1.3]{LiPS24}). If we wish to emphasise the arc-transitive subgroup $G$ of $\Aut(\M)$, then we refer to the embedding as a $G$-arc-transitive embedding.
	
	We use the following necessary and sufficient conditions for connected graphs other than cycles to have an arc-transitive embedding. Proof of sufficiency relies on constructions in \cite{LiPS21, LiPS24} which are effective for not necessarily simple arc-transitive graphs. An analogue of this result for flag-regular embeddings was proved in \cite{GNSS}.
	
	\begin{prop}\label{cri}
		Let $\G = (\O,E)$ be a finite connected $X$-arc-transitive graph of valency at least $3$, where $X\leq \Aut(\G)$, and let $(\a,\b)$ be an arc of $\G$ and $G\leq X$. Then $\G$ has a $G$-arc-transitive embedding $\M$ if and only if  the following three conditions hold:
		\begin{itemize}
			\item[{\rm (i)}] $X = GX_\a$ and $X_\a = G_\a X_{\a\b}$;
			\item[{\rm (ii)}] $G_\a$ is cyclic or dihedral and  $G_{\a\b} = 1$ or $C_2$;
			\item[{\rm (iii)}] there is an involution $g \in N_G(X_{\a\b})$ such that $\la X_\a,g \ra = X$.
		\end{itemize}
		Furthermore,  if these conditions hold and $X_{\a\b}$ contains a nontrivial normal subgroup $N$ of $X_\a$, then $N^g\ne N$ and $N^g\unlhd X_{\a\b}$. 
	\end{prop}
	
	\begin{proof} 
		Suppose that $\M$ is a $G$-arc-transitive embedding of $\G$, so $G\leq \Aut(\M)$. Then both $G$ and $X$ are arc-transitive on $\G$, and this yields part (i). Part (ii) follows from the discussion before the proposition.  For part (iii), let $g\in G$ be an element that maps the arc $(\a,\b)$ to the reverse arc $(\b,\a)$ (which exists since $G$ is arc-transitive). Then $g\in G_{\{\a,\b\}}$, and hence $g^2=1$ since  
		$G_{\{\a,\b\}}\leq C_2\times C_2$, again by the discussion above. In particular $g$  normalises $X_{\a\b}$, and since $\G$ is connected, $\langle X_{\a}, g\rangle=X$ (see, for example, \cite[Theorems 2.2(a) and 2.11]{LiPS21}). 
		
		Conversely, suppose that conditions (i)--(iii) all hold for an arc-transitive subgroup $X\leq \Aut(\G)$ and a subgroup $G\leq X$. 
		Then condition (i) implies that $G$ is arc-transitive on $\G$. The equality $\langle X_\a,g\rangle=X$ in condition (iii) implies that $g\not\in X_\a$ so $\b=\a^g\ne\a$, and the inclusion $g\in N_G(X_{\a\b})$  in condition (iii) implies that $g\in N_G(G_{\a\b})$, and hence $\langle G_\a, g\rangle=G$ as $\G$ is connected.   Let $k:=|\G(\a)|$, the valency of $\G$, so $k\geq3$. Then condition (ii) implies that, for the arc-transitive group $G$, the vertex-stabiliser $G_\a$ is one of $D_{2k}$ or $C_k$, or $k$ is even and $G_\a=D_k$. Suppose first that $G_\a=D_{2k}$. Then by conditions (ii) and (iii), $G_{\a\b}=\langle x\rangle\cong C_2$, $G_{\{\a,\b\}}$ (of order $4$) is equal to $\langle x,g\rangle\cong C_2\times C_2$, and there is a third involution $y$ such that $G_\a=\langle x,y\rangle$. Then the involution triple $(x,y,g)$ is a flag-regular triple as in \cite[Definition 4.8]{LiPS21}, and by \cite[Proposition 4.11]{LiPS21}, the construction \cite[Construction 4.9]{LiPS21} produces a flag-regular map with automorphism group $G$ (or see \cite[Section 3]{GNSS}), embedding a graph $\Sigma$ with vertex set $V\G$. Since $G$ acts arc-transitively on $\Sigma$ and $(\a,\b)$ is an arc of $\Sigma$ in this construction, it follows that $\Sigma=\G$, and we have a flag-regular embedding of $\G$.  Next suppose that $G_\a=\langle a\rangle = C_k$. Then the pair $(a,g)$ is a rotary pair as in \cite[Definition 1.6]{LiPS21}, and by \cite[Propositions 4.2 and 4.5]{LiPS21}, the constructions \cite[Constructions 4.2 and 4.4]{LiPS21} produce arc-transitive maps with vertex set $V\G$ and  automorphism group containing $G$ as an arc-transitive subgroup. The same argument as in the flag-regular case shows that the embedded graph is $\G$. Moreover  by \cite[Theorem 1.9]{LiPS21}, these constructions provide the only arc-transitive embeddings of $\G$ up to map isomorphism. 
		Now let $G_\a=D_k$ with $k$ even. Then $G_\a=\langle x,y\rangle$ for some involutions $x,y$, and by condition (iii), $G=\langle x,y,g \rangle$, 
		so $(x,y,z)$ is an arc-transitive triple as in \cite[Definition 1.1]{LiPS24}. Hence by   \cite[Propositions 3.4 and 3.7]{LiPS24}, the 
		constructions \cite[Constructions 3.2 and 3.5]{LiPS24} produce  maps with vertex set $V\G$ and with automorphism groups containing the arc-regular subgroup $G$. Arguing as before, the embedded graph in each case is $\G$ and so we have $G$-arc-transitive embeddings of $\G$. By \cite[Theorem 1.3]{LiPS24}, these constructions are the only  arc-transitive embeddings of $\G$ up to map isomorphism. 
		
		Finally suppose that conditions (i)--(iii) hold and that    $X_{\a\b}$ contains a nontrivial normal subgroup $N$ of $X_\a$. Then by part (iii), $N^g \unlhd (X_{\a\b})^g=X_{\a\b}$. If $N^g= N$  then $N$ is normalised by $\langle X_\a,g\rangle$, which equals $X$ by (iii). Thus $N\unlhd X$ and so $X$ acts unfaithfully on $V\G$, which is a contradiction. Thus $N^g\ne N$.
	\end{proof}

	\subsection{The map theorem: strategy and set-up}\label{s:mapsetup} 
	
	Our main application of the  factorisation results for almost simple groups in Theorems~\ref{t:Anfactns}, \ref{classgen} and~\ref{sporadic}, relates to arc-transitive embeddings of graphs. We are   concerned with a graph $\G$ admitting an arc-transitive, almost simple subgroup $X$ of automorphisms, such that $\G$ has a $G$-arc-transitive embedding for some subgroup $G$ of $X$.  We analyse the case where  $G$ does not contain $\soc(X)$, and obtain an almost complete classification.  
	
	Given the graph $\G$ and groups $X$ and $G$, existence of a $G$-arc-transitive embedding is equivalent to conditions (i)--(iii) of Proposition~\ref{cri}. Moreover, 
	as explained in the proof of Proposition~\ref{cri}, for permutation groups $G<X$ such that conditions (i)--(iii) of Proposition~\ref{cri} hold, these three conditions may be used to construct both the $X$-arc-transitive graph and the embedding. This observation underpins our strategy. 
	
	Conditions (i)--(iii) of Proposition~\ref{cri} together with the condition that $G$ does not contain $\soc(X)$ are captured in Hypothesis~\ref{hypo} (on taking $\b=\a^g$ in Hypothesis~\ref{hypo}). Using this hypothesis we examine each of the cyclic/dihedral factorisations identified in 
	Theorems~\ref{t:Anfactns}--\ref{sporadic}; we determine whether or not there exists an involution $g\in G$ with the properties required by Hypothesis~\ref{hypo}, and if found then we construct the graph and note the embedding. We obtain explicit identifications of the graphs with the exception of the cases of Theorem~\ref{t:maps-an}(ii).
	
	To help with our analysis, we note that the last sentence of Proposition~\ref{cri} provides additional information about the subgroup structure which follows from conditions (i)--(iii), and this extra condition proves very useful. Thus we make the following hypothesis which (by Proposition \ref{cri}) is equivalent to Hypothesis~\ref{hypo}.

	\begin{hyp}\label{h:map}
		$X$ is an almost simple group with socle $X_0$, acting faithfully and transitively on a set $\O$, $G$ is a core-free subgroup of $X$, and $g \in G$ is an involution such that the following hold for $\a\in\O$ and $\b=\a^g$:
		\begin{enumerate}
			\item[(a)] $X=G\, X_\a$,\  $X_\a = G_\a\, X_{\a\b}$, $G_\a$ is cyclic or dihedral, and $|G_{\a\b}|\leq 2$;
			\item[(b)] $g \in N_G(X_{\a\b})\setminus N_G(X_{\a})$, and $\la X_\a,g \ra = X$;
			\item[(c)] if $1\ne N\unlhd X_\a$ and $N\leq X_{\a\b}$, then $N^g\ne N$ and $N^g\unlhd X_{\a\b}$.
		\end{enumerate}
		Note that these conditions imply that $\G=(\O,\{\a,\b\}^X)$ is a connected simple graph of valency at least $3$ with arc-transitive automorphism subgroups $X, G$, and that $g$ reverses the arc $(\a,\b)$.
	\end{hyp}
	
	We shall analyse the cyclic/dihedral factorisation for the alternating groups, Lie type groups, and sporadic groups in Sections \ref{s:anmaps}, \ref{s:liemaps} and \ref{s:sporadic}, respectively, to determine which of them satisfy Hypothesis \ref{h:map} and hence correspond to arc-transitive embeddings, thus proving Theorems \ref{t:maps-an}, \ref{t:classical-maps} and \ref{spormaps}.
	
	\subsection{Our use of computation} \label{maggie}
	
	In our proofs of the factorisation theorems \ref{t:Anfactns}, \ref{classgen}, \ref{sporadic}, and also of the map classification theorems \ref{t:maps-an}, \ref{t:classical-maps}, \ref{spormaps}, we make considerable use of computations in Magma \cite{magma} for various almost simple groups. Here is a brief description of these computations.
	
	First, there is a list of small almost simple groups $X$ -- namely, those with socle $A_n$ for $n\le 10$, and those in the list (\ref{listsm}) -- for which conjugacy class representatives of all subgroups can be computed. For these, we can verify the conclusions of Theorems \ref{t:Anfactns}, \ref{classgen}, \ref{t:maps-an} and \ref{t:classical-maps} in a very naive fashion, as follows. To classify the cyclic/dihedral factorisations, compute all pairs $(A,B)$ of subgroups such that $|X| = \frac{|A|\,|B|}{|A\cap B|}$ and $A\cap B$ is cyclic or dihedral. For such a pair $(A,B)$, let $\O = \{ Ax\mid x\in X\}$, and write $X_\a = A$, $G = B$ to conform with the notation of Hypothesis \ref{h:map}. Compute all double coset representatives $h$ for $X_\a / X\backslash X_\a$, and for each $h$, let $\b = \a^h$ (so that $X_{\a\b} = X_\a \cap X_\a^h$). Now check whether the conditions in (a) of Hypothesis \ref{h:map} hold, and if they do, search for an involution $g$ satisfying (b) of Hypothesis \ref{h:map}.
	
	We also need to do computations with some larger almost simple groups, for which it is not practical to compute all subgroups -- namely, the classical groups $X$ appearing in Tables \ref{excep-psl} - \ref{excep-sp}, and the sporadic groups in Table \ref{spor}. For these classical groups, we have proved the factorisation theorem \ref{classgen} without using computation, and only use Magma for the map theorem \ref{t:classical-maps}: the latter proceeds as described in the previous paragraph -- namely, for each pair $(A,B)$ in the tables, we compute double coset representatives, check whether the factorisation conditions in Hypothesis \ref{h:map}(a) hold, and if they do, search for an involution $g$ satisfying Hypothesis \ref{h:map}(b). Finally, for the sporadic groups, we know all the factorisations by \cite{Giudici}, so again we know exactly which subgroups to compute with, which we do in the same way.

	\section{Some arc-transitive graphs admitting arc-transitive embeddings}\label{arc-tr}
	
	In this section we introduce some of the graphs which occur in our classifications in Theorems \ref{t:maps-an} and  \ref{t:classical-maps}.
	It is sometimes convenient to denote  the vertex set of a graph $\G$ by $V\G$.
	
	\subsection{Some infinite families of graphs associated with symmetric groups}
	The following well-known graphs arise in our work:
	\begin{itemize}
		\item[(i)]  For each $n\geq2$ the \emph{complete graph} $K_n$ is the graph on $n$ vertices with all pairs of distinct vertices forming an edge; $\Aut(K_n)=S_n$.
		\item[(ii)]  For each $n\geq2$ the \emph{complete bipartite graph} $K_{n,n}$ is the graph on $2n$ vertices with vertex set $V^+\cup V^-$, where $V=\{1,\dots,n\}$, and edges all pairs $\{v^+, u^-\}$ with $v,u\in V$; $\Aut(K_{n,n})=S_n\wr S_2$.
		\item[(iii)]  For each $n\geq2$ the \emph{complete bipartite graph minus a matching} $\G=K_{n,n}\setminus  n.K_2$ is the graph on $2n$ vertices with vertex set $V\G=V^+\cup V^-$, where $V=\{1,\dots,n\}$, and edges all pairs $\{v^+, u^-\}$ with $v,u\in V$ and $u\ne v$; 
		$\Aut(\G)=S_n\times S_2$.
		\item[(iv)]  For each $n\geq 3$ the \emph{Johnson graph} $J(n,2)$ is the graph on $\binom{n}{2}$ vertices with vertices being the unordered pairs from an $n$-set, and with distinct vertices $\a,\b$ forming an edge if and only if $\a\cap\b\ne\emptyset$. If $n\ne 4$ then $\Aut(J(n,2))=S_n$, while $\Aut(J(4,2))=S_4\times S_2$.
		
	\end{itemize}
	All of these graphs occur in Theorem~\ref{t:maps-an}.
	
	\subsection{Arc-transitive graphs from point-block incidence structures}
	
	A point-block \emph{incidence structure} $\DD=(\PPP,\BB)$ consists of a set $\PPP$ of points and a set $\BB$ of blocks together with an incidence relation $I\subseteq \PPP\times\BB$. Often the elements of $\BB$ are subsets of $\PPP$ and the relation $I$ is inclusion. 
	Some of our examples arise as the \emph{incidence graph} $\Inc(\DD)$ of a point-block structure $\DD$; this is the graph with vertex set $\PPP\cup\BB$ and edges the pairs $\{\a,\b\}$ such that $\a\in\PPP,\b\in\BB$ and $\a,\b$ incident. Such a pair is called a \emph{flag} of $\DD$. 
	
	The incidence structures arising in our examples are $2$-designs, that is, each unordered pair of points is incident the same number $\lambda$ of blocks. 
	Several examples of this nature come from \emph{projective planes} $\PG_2(q)$ where $\PPP$ and $\BB$ are the sets of $1$-spaces and $2$-spaces, respectively, of a vector space $\mathbb{F}_q^3$; these are $2$-designs with $\lambda=1$. Another example comes from the $11$-point \emph{biplane} with  blocks of size $5$ (which is a $2$-design with $\lambda=2$). Lastly there is an example we call the \emph{antiflag graph} of the projective geometry $\PG_3(2)$ where $\PPP$ and $\BB$ are the sets of $1$-spaces and $3$-spaces, respectively, of the vector space $\mathbb{F}_2^4$; and the edges of the antiflag graph are the point-block pairs which are non-incident.

	\subsection{Other general constructions}\label{s:other}
	
	Several of our examples of arc-transitive graphs arise as orbital graphs of a transitive permutation group. Given a transitive permutation group $G\leq \Sym(\O)$ we use a \emph{nontrivial self-paired orbital}, that is to say a $G$-orbit $\Delta\subseteq \O\times \O$ consisting of pairs of distinct points such that $(\a,\b)\in\D$ if and only if  $(\b,\a)\in\D$. The associated \emph{orbital graph} has vertex set $\O$ and edge set $\{ \{\a,\b\}\mid (\a,\b)\in\D\}$; the arc set is $\D$ and by construction $G$ is admitted as an arc-transitive subgroup of automorphisms.
	
	Graphs described as orbital graphs include the infinite family constructed from finite alternating groups discussed in Section~\ref{bk1prop}.
	
	There are several graphs arising from a small orthogonal geometry which we describe in Subsection~\ref{s:ortho}. That discussion uses the following concepts:
	\begin{itemize}
		\item[(i)]  The \emph{lexicographic product} $\Sigma[2.K_1]$ of a graph $\Sigma$ and the graph $2.K_1$ (consisting of two isolated vertices) has vertex set $V\Sigma\times \mathbb{Z}_2$ with $(\sigma,i)$ adjacent to $(\sigma',j)$ if and only if $\{\sigma,\sigma'\}$ is an edge of $\Sigma$,  see \cite[p.17]{GR}.  The automorphism group $\Aut(\Sigma[2.K_1])$ contains $C_2\wr \Aut(\Sigma)$ (sometimes as a proper subgroup), and for a subgroup $A\leq \Aut(\Sigma)$, the group $C_2\wr A$ is arc-transitive on $\Sigma[2.K_1]$ if and only if $A$ is arc-transitive on $\Sigma$.
		
		\item[(ii)] Let  $G\leq \Sym(\O)$ be a transitive permutation group with an orbit $\D\subseteq \O\times\O$ consisting of pairs of distinct points which is not self-paired, so that $\D':=\{(\b,\a)\mid (\a,\b)\in\D\}$ is a $G$-orbit different from $\D$. We say that $\D$ and $\D'$ are \emph{paired orbitals}. We define  a (simple, undirected, bipartite) graph $\widehat{\D}$  which is a type of `doubling' of the orbital graph construction given above.  
		The vertex set $V\widehat{\D}$ is a disjoint union of two copies of $\O$ which we write as $V\widehat{\D} = \{v^+\mid v\in \O\}\cup \{v^-\mid v\in \O\}$. The edges of $\widehat{\D}$ are the pairs $\{u^+, w^-\}$ for which $(u,w)\in\Delta$.  The graph $\widehat{\D}'$ is defined similarly with $\Delta'$ replacing $\Delta$, and the map
		\begin{equation}\label{e:rho}
			\rho: u^+\ \longleftrightarrow\ u^-\quad \text{for}\ u\in \O
		\end{equation}
		induces a graph isomorphism interchanging $\widehat{\D}$ and $\widehat{\D}'$. The group $G$ is admitted as an edge-transitive group of automorphisms of both $\widehat{\D}$ and $\widehat{\D}'$, and $G$ has two vertex-orbits. Also $\langle G,\rho\rangle = G.2$ is transitive on $V\widehat{\D}=V\widehat{\D}'$.
	\end{itemize}
	
	\subsection{Arc-transitive graphs from orthogonal geometry}\label{s:ortho}\label{ex:u42}


	Let $X=\SO_5(3)=\POm_5(3).2$ with natural module $V=\F_3^5$ equipped with a non-degenerate quadratic form, and let $X_0=\Omega_5(3)$ so $X=X_0.2$, We introduce several arc-transitive graphs admitting $X$, all related to the vertex-primitive, rank $3$, strongly regular graph $\Sigma$ we define in part (a) below. These graphs arise in the analysis for Lie type groups, in Table~\ref{class-maps}.
	\begin{enumerate}
		\item[(a)] \emph{Orthogonality graph $\Sigma$ on singular $1$-spaces:}\quad Let $\Sigma$ be the graph with  $V\Sigma$ the set of $40$ singular $1$-spaces in $V$, such that two vertices $\sigma,\sigma'$ are adjacent if and only if they are orthogonal. Then  $\Sigma$ is strongly regular with valency $12$, and  $\Aut(\Sigma)=X$ is vertex primitive with rank $3$.
		
		\item[(b)] \emph{Orthogonality graph $\G$ on singular vectors:}\quad Geometrically this is a natural `doubling' of the graph $\Sigma$. The vertex set of $\G$ is the set $\Omega$ of $80$ non-zero singular vectors in $V$, with edges the pairs $\{v,u\}$ from $\O$ such that $v, u$ are orthogonal and $v\ne \pm u$. Then $\G$ has valency $24$ and we will prove that $\G\cong \Sigma[2.K_1]$. Moreover:
		\begin{enumerate}
			\item[(i)]  $\Aut(\G)$ contains $X$ as an arc-transitive subgroup  and, for $v\in\O$, $X_v=3^3.S_4$, $X_v$ has orbits $\{v\}$, $\{-v\}$, $\G(v)=\{w\mid (v,x)=0\}$, $\{w\mid (v,x)=1\}$ and $\{w\mid (v,x)=-1\}$, of sizes $1,1,24, 27,27$.
			\item[(ii)] The socle $X_0$ of $X$ is vertex-transitive on $\G$. The (unique) index $2$ subgroup $X_0\cap X_v=3^3.A_4$ of $X_v$ is transitive on the two $X_v$-orbits (of size $27$), and has two orbits $\Delta(v), \Delta'(v)$ of length $12$ in $\G(v)$. There is an involution $g\in X_v\setminus (X_0\cap X_v)$ such that $g$ interchanges  $\Delta(v)$ and $\Delta'(v)$. The $X_0$-orbitals $\Delta$ and $\Delta'$ are paired with each other.
		\end{enumerate}    
		
		\item[(c)]  \emph{The `doubled' graph $\widehat{\G}$:}\quad  We define $\widehat{\G}$ as the doubling $\widehat{\D}$ of the $X_0$-orbital $\D$ in (b)(ii) - see  Subsection~\ref{s:other}(ii). 
		Thus $\widehat{\G}$ has $160$ vertices and valency $12$. The graph $\widehat{\G}'$ defined similarly with the paired orbital $\Delta'$ is isomorphic to $\widehat{\G}'$ under the map $\rho$ in \eqref{e:rho}.
	\end{enumerate}

	\begin{lemma}\label{l:u42}
		Let $X, X_0$, $\G$, $\Sigma$, and $\widehat{\G}$ be as above. Then the following hold.
		\begin{enumerate}
			\item[(i)] The claims of  (b) are valid, and moreover the map $\tau: u\to -u$ $(u\in\O)$ lies in $C_{\Aut(\G)}(X)\setminus X$, the graph $\G\cong \Sigma[2.K_1]$, and $\Aut(\G)$ contains $C_2\wr X=2^{40}\cdot X$.
			
			\item[(ii)] The claims of  (c) are valid and moreover, with $\rho$ as in \eqref{e:rho}, and letting each $x\in X$ and $\tau$ act on $V\widehat{\G}$ by $x:u^\epsilon\to (u^x)^\epsilon$ and $\tau: u^\epsilon\to (-u)^\epsilon$ $(u\in\O, \epsilon=\pm)$, we have 
			\[
			\Aut(\widehat{\G})\cap \left(\langle \rho\rangle\times \langle \tau\rangle\times X\right)=\langle X_0, \rho g\rangle\times \langle \tau\rangle \cong X\times C_2
			\]
			and $\langle X_0, \rho g\rangle\cong X$  is arc-transitive on $\widehat{\G}$. 
		\end{enumerate}
		
	\end{lemma}
	
	\begin{proof}
		(i) The claims in part (b) can be checked by Magma computations. The map $\tau=-I$ lies in $\Aut(\G)$ by inspection, $\tau$ centralises $X=\SO_5(3)$, and $\tau\not\in X$. It follows that $\G(v)^\tau=\G(-v)$ is equal to the unique $X_v$-orbit $\G(v)$ of size $24$ in $\O$. Thus for the partition $\Pi$ with parts $\{u,-u\}$ (for $u\in\O$), two parts $\sigma, \sigma'$ are adjacent if and only if each vertex of $\sigma$ is adjacent to each vertex of $\sigma'$. Hence the quotient of $\G$ modulo $\Pi$ is $\Sigma$ and $\G\cong \Sigma[2.K_1]$, so $\Aut(\G)$ contains $C_2\wr X=2^{40}\cdot X$.  
		
		(ii) As noted in Subsection~\ref{s:other}, the map $\rho$ induces a graph isomorphism interchanging $\widehat{\G}$ and $\widehat{\G}'$. 
		Thus the claims of  part (c) are valid. It follows from the definition of $\widehat{\G}$ that this graph is bipartite and that $X_0\leq \Aut(\widehat{\G})$ with two orbits $\O^+:= \{v^+\mid v\in \O\}$ and $\O^-:=\{v^-\mid v\in \O\}$; and the $X_0$-actions on $\O,\O^+,\O^-$ are all equivalent. In particular, $(X_0)_v$ is transitive on $\widehat{\G}(v^+) = \G(v)^-$. 
		
		The map $\rho$ in \eqref{e:rho} centralises $X$ and $\tau$. Consider the group $Y:=\langle \rho\rangle\times \langle \tau\rangle\times X$. By the previous paragraph $\Aut(\widehat{\G})\cap Y$ contains $X_0$, and it is straightforward to see that $\tau\in \Aut(\widehat{\G})\cap Y$. By part (b)(ii),  there is an involution $g\in X_v\setminus (X_0\cap X_v)$ such that $g$ interchanges  $\Delta(v)$ and $\Delta'(v)$ in its action on $\O$. 
		Thus, acting on $\O^+\cup\O^-$, $g$ fixes each of $\O^+$ and $\O^-$ setwise and sends each edge  
		$\{u^+,w^-\}$ of $\widehat{\G}$ (that is, $(u,w)\in\Delta$) to an edge  $\{w^+,u^-\}$  of 
		$\widehat{\G}'$ (since $(w,u)\in\Delta'$). 
		That is to say, both $g$ and $\rho$ interchange the graphs  $\widehat{\G}$ and $\widehat{\G}'$, and so their product $\rho g$ lies in $\Aut(\widehat{\G})\cap Y$. 
		Therefore $\Aut(\widehat{\G})\cap Y$ contains $\langle\tau, X_0, \rho g\rangle\cong \langle\tau\rangle\times\langle X_0, \rho g\rangle \cong C_2\times X$, and since this subgroup has index $2$ in $Y$, it follows that equality holds. Since $\rho g$ interchanges $\O^+$ and $\O^-$, the group $\langle X_0, \rho g\rangle$ is vertex-transitive, and the stabiliser of $v^+$ in $\langle X_0, \rho g\rangle$ is  $(X_0)_v$, which is transitive on $\widehat{\G}(v^+)$. Thus $\langle X_0, \rho g\rangle$ is arc-transitive on $\widehat{\G}$.
	\end{proof} 
	
	
	\section{A new infinite family of arc-transitive maps from alternating groups}\label{bk1prop}
	
	In Construction~\ref{c:k1eg} we define an infinite family of graphs which have arc-transitive embeddings: each of these graphs has prime valency $p$ and admits an arc-transitive action of an alternating group $A_n$ and also an arc-transitive action of one of its subgroups $A_{n-1}$, for certain primes $p$, where $n=(p-1)!/2$. We use the following set of primes. 
	
	\begin{definition}\label{d:pi}
		{\rm   Define $\Pi$ to be the set of all primes $p$ such that 
			\begin{enumerate}
				\item [(a)] $p\equiv 3\pmod{4}$ and $p>3$; and
				\item[(b)] the only  insoluble transitive permutation groups of degree $p$ are $A_p$ and $S_p$.
			\end{enumerate}
		}
	\end{definition}
	
	\begin{remark}\label{r:infprime}
		{\rm It is well-known (see for example \cite[Theorem 1]{G83}) that $\Pi$ contains all primes satisfying part (a) except $11, 23$, and those of the form $\frac{q^d-1}{q-1}$, where $q$ is a prime-power. We claim that for a large positive integer $x$, 
			\begin{equation}\label{pino}
				|\{ p \in \Pi : p \le x\}| \sim \frac{x}{2\log x},
			\end{equation}
			(so in particular, $\Pi$ is infinite). To see this, write $X = \mathbb{N} \cap [1,x]$, and note that the number of integers in $X$ of the form $\frac{q^d-1}{q-1}$ for some $d\ge 3$ is at most $x^{1/2}+x^{1/3}+\cdots + x^{1/(d-1)}$. Also the number of primes in $X$ of the form $q+1$, with $q$ a prime power, is at most $\log_2x$. It follows that 
			\[
			|\Pi \cap X| = \pi_1(x) + O(x^{1/2}),
			\]
			where $\pi_1(x)$ is the number of primes in $X$ congruent to 3 modulo 4. By Dirichlet's theorem $\pi_1(x) \sim \frac{x}{2\log x}$, and (\ref{pino}) follows.
		}
	\end{remark}
	
	\begin{construction}\label{c:k1eg}
		{\rm Let $p\in \Pi$ and $n:=(p-1)!/2$. 
			\begin{itemize}
				\item[(i)]  Let $H=A_p$ with subgroups $C=C_p$ and $K=A_{p-1}$ so that $H=CK$, and $H$ has a transitive coset action on $I:=[H:C]=\{ Cx\mid x\in K\}$, with $|I|=n$.
				\item[(ii)] Let $X:=\Alt(I)$ and $\nu:=C\in I$, so $G:=X_\nu=A_{n-1}$. Then $X$ has a transitive coset action on $\O:=[X:H]=\{ Hx\mid x\in X\}$ and, for $\a:=H\in\O$, $X_\a=H$.
				\item[(iii)] Let $g_0\in\Aut(K)$ correspond to a transposition in $\Aut(K)\cong S_{p-1}$, and  define an associated map on $I = [H:C]$ as follows:
				\[
				g : Cx \mapsto Cx^{g_0} \;\hbox{ for }x \in K.
				\]
				We show in Claim 1 below that $g\in G$, that $|g|=2$, and that $\b:=\a^g\in\O\setminus\{\a\}$. 
				\item[(iv)] Define $\G(p)$ to be the orbital graph on $\O$ with edge set $\{\a,\b\}^X$. So we have 
				\[
				X=A_n, \,X_\a=A_p,\, X_{\a\b}=A_{p-1},\, G=A_{n-1},\, G_\a=C_p,\, G_{\a\b}=1.
				\]
			\end{itemize}
		}
	\end{construction}
	
	\begin{thm}\label{k1eg}
		Let $p \in \Pi$ (as in Definition~$\ref{d:pi}$), and let $n = (p-1)!/2$ and $\G=\G(p), X=A_n, G=A_{n-1},$ etc. be as in Construction~$\ref{c:k1eg}$. Then the following hold.
		\begin{itemize}
			\item[(i)] $\G$ is a connected, non-bipartite, $X$-arc-transitive graph of valency $p$ with $n!/p!$ vertices, and $X$ is imprimitive on vertices;
			\item[(ii)] the vertex-stabiliser $X_\a \cong A_p$, and arc-stabiliser  $X_{\a\b} \cong A_{p-1}$ (with $X_{\a\b}$ acting regularly in the natural action on $n$ points);
			\item[(iii)] $\G$ admits a $G$-arc-transitive map embedding with $G$ acting regularly on arcs ($G_\a \cong C_p$ and $G_{\a\b} = 1$).
		\end{itemize}
	\end{thm}
	
	We note that the map $\M$ constructed in Theorem~\ref{k1eg} (as in the proof of Proposition~\ref{cri}) has automorphism group  $\Aut(\M) \cong A_{n-1}$ or $S_{n-1}$.

	\begin{proof}
		Adopt the notation of Construction~\ref{c:k1eg}, and note that $G_\a = H\cap G = C$. 
		First we make some obervations about the $X$-action on $\O$. The simple group $X=A_n$ has no index $2$ subgroup, so the $X$-action on $\O$ does not preserve a block system consisting of just two blocks (so the graph $\G$ will be non-bipartite). 
		Also the transitive action of $H = A_p$ on the $n$-set $I$ is its action on the cosets of $C=C_p$, and as $C<N_H(C)=C.\frac{p-1}{2} <H$, this action is imprimitive, and $H$ is contained in an imprimitive wreath product subgroup $A_k \wr A_{n/k}$ of $X$ (with $k=\frac{p-1}{2}$). Hence $H=X_\a$ is not maximal in $X$, so the action of $X$ on $\O$ is imprimitive.

		Next we establish the claims made in Construction~\ref{c:k1eg}(iii) about the element $g$.
		
		\medskip\noindent
		\emph{Claim 1: $g \in \Alt(I)$, $g\in G$, $|g|=2$, and $|{\rm fix}_I(g)| = |C_K(g_0)|=(p-3)!$.}
		
		\smallskip\noindent\emph{Proof of Claim 1.}\quad 
		By definition $g \in \Sym(I)$ (since $K$ is regular on $I$), $|g|=2$ (since $|g_0|=2$), and $|{\rm fix}_I(g)| = |C_K(g_0)|$. The latter implies that $g$ is a product of $\frac{1}{2}(|K|-|C_K(g)|)$ transpositions. Since $p>7$ (as $p \in \Pi$), all involutions in $A_{p-1}$ have centralisers of order divisible by $4$, and so $g \in \Alt(I)=X$. Moreover, by the definition of $g$ we see that $g$ fixes the point $\nu=C\in I$ and hence $g\in X_\nu=G$. Finally, as $g_0$ corresponds to a transposition in $S_{p-1}$, 
		$C_K(g_0) \cong (S_2\times S_{p-3})\cap A_{p-1}\cong S_{p-3}$, proving the  claim.
		
		\medskip
		Our strategy for proving the theorem is to verify that all  conditions (i)-(iii) of Proposition \ref{cri} hold. Firstly since $H$ is transitive on $I$, we have $X=GH=GX_\a$\, and also $G_\a=G\cap H=C$ is cyclic. It is not immediately obvious that the points $\a,\b$ are distinct, so we address this next.
		
		

		\medskip\noindent
		\emph{Claim 2: $g \in N_X(K)\setminus N_X(H)$; in particular $\b\ne\a$ and $g$ interchanges $\a$ and $\b$.}
		
		\smallskip\noindent\emph{Proof of Claim 2.}\quad 
		By the definition of $g$ we have $g \in N_X(K)$.  Suppose for a contradiction that $H^g=H$. Then $Y:=\langle H,g\rangle$ contains $H=A_p$ as a subgroup of index $2$ and $g$ induces a non-identity automorphism of $H$ (as $g$ does not centralise $K$). Hence $Y\cong S_p$. Also $Y\cap G$ contains $H\cap G=C$, and as $g\in (Y\cap G)\setminus H$ it follows that  $D:=Y\cap G =(H\cap G).2= C.2$. Now $C$ is self-centralising in $Y\cong S_p$, so $D=Y\cap G\cong D_{2p}$. For the transitive group $Y$ on $I$ with point stabiliser $D$, counting pairs $(x,\mu)$ such that $x\in g^Y, \mu\in I$ and $\mu^x=\mu$, we have $|g^Y|\cdot |{\rm fix}_I(g)|=|I|\cdot |g^Y\cap D|$. By Claim 1, $|{\rm fix}_I(g)| = (p-3)!$, and we have 
		$|g^Y\cap D|=p$ and $|I|=n=\frac{1}{2}(p-1)!$. It follows that $|g^Y|=\frac{p!}{2\cdot (p-3)!}$, and hence  $|C_Y(g)|=2.\cdot (p-3)!$. However no involution in $S_p$ has a centraliser of this order. Thus  $g \not\in N_X(H)$ and the first assertion is proved. In the coset action of $X$ on $\O$, $g$ maps the point $\a=H$ to $\b=Hg$, and as $g\not\in N_X(H)$ it follows that $\b\ne\a$, and as $|g|=2$ by Claim 1, the involution $g$ interchanges $\a$ and $\b$.
		
		\medskip
		By Claims 1 and 2, $\G$ is an orbital graph for a nontrivial self-paired $X$-orbital, so $\G$ is $X$-arc-transitive (see Subsection~\ref{s:other}). 
		It also follows from Claim 2 that $X_{\a\b} = X_\a \cap X_\a^g =H\cap H^g= K$ since $g$ normalises the maximal subgroup $K$ of $H$. Thus $g\in N_G(X_{\a\b})$ by Claim 2, and also  $\G$ has valency $p=|X_\a:X_{\a\b}|$ and $|X:X_\a|=n!/p!$ vertices. Also $G_{\a\b}=G\cap X_{\a\b}=G\cap K=C\cap K=1$ and it follows that $G$ acts arc-transitively (in fact arc-regularly) on $\G$ and $X_\a = G_\a X_{\a\b}$.  Thus to  show that  conditions (i)-(iii) of Proposition \ref{cri} hold, it remains to prove that  $Y:=\la X_\a, g \ra = X$ (this property is equivalent to proving that $\G$ is connected). We now assume to the contrary that $Y:=\la X_\a, g \ra \ne X$ and prove in several steps that this leads to a contradiction. 
		
		\medskip\noindent
		\emph{Claim 3: $Y:=\la X_\a, g \ra$ is primitive on the set $I$.}
		
		\smallskip\noindent\emph{Proof of Claim 3.}\quad 
		Suppose to the contrary that $\BB$ is a nontrivial block system for $Y^I$. Since $g$ fixes the point $\nu=C\in I$, $g$ fixes setwise the block $B \in \BB$ containing $\nu$. As $H$ is transitive on $I$ with point-stabilizer $C$, we have $C < H_B < H$. Since $p \in \Pi$, the only insoluble transitive groups of degree $p$ are $A_p$ and $S_p$, and hence $C \triangleleft H_B$, and so $H_B = C\cdot \la z\ra \le \AGL_1(p)$, where $z$ is an element of order $r$ dividing $p-1$. We may take $z$ to lie in $K = H_\nu\cong A_{p-1}$, and so $K_B = \la z\ra$. Since $g$ fixes $B$ setwise and normalises $K$ (by Claim 2), it follows that $g$ normalises $K_B = \la z\ra$. However by definition $g$ induces the automorphism $g_0$ of $K$ which is conjugation by a transposition in $S_{p-1}$ (see Construction~\ref{c:k1eg}(iii)), and so $\langle K, g\rangle \cong S_{p-1}$. Now  $\langle z \rangle  \le \AGL_1(p)$ acts semiregularly on the set $\{1,\dots,p-1\}$ and is therefore normalised by $g$ if and only if $r=|z|=2$.  However this means that the permutation $z$ induces on $\{1,\dots,p-1\}$ is a product of $(p-1)/2$ cycles of length $2$, and as $p\in\Pi$, $(p-1)/2$ is odd, so $z\not\in A_{p-1}$, which is a contradiction, proving Claim 3. 
		
		\medskip
		Since $Y = \la H, x \ra < X=A_n$, we have $Y\leq M<X$ for some maximal subgroup $M$ of $X$, and by Claim 3, $M$ is primitive 
		on $I$.  We consider what O'Nan-Scott type of primitive group $M$ can be. Recall that $|I| = n = (p-1)!/2$. Using Bertrand's Postulate, we see that $n$ is not a prime, or a proper power of an integer. Hence $M$ cannot be of affine type, of product type, of twisted wreath type, or of simple diagonal type of degree $|T|^k$ with $k\ge 2$, where $T$ is a simple group. It follows that either 
		\begin{itemize}
			\item[(i)] $M$ is of simple diagonal type of degree $|T|$, for a non-abelian simple group $T$, or 
			\item[(ii)] $M$ is almost simple.
		\end{itemize}
		In case (i), since $A_{p-1}$ is the unique simple group of order $n$ (see \cite{KLST}), we have $T = A_{p-1}$. As $M$ is of simple diagonal type,  $M \le T^2.(\Out\, T \times S_2)$. However $M$ contains $H = A_p$, which is a contradiction.
		Hence (ii) holds, and $M$ is almost simple. Now $M< X=A_n$, and $M$ is primitive of degree $n$ and contains a regular subgroup $K = A_{p-1}$. It follows by \cite[Thm. 1.4 and Table 2]{LPS10} that there is a prime $r$ such that $p-1 = r^2-2$, but this implies that $p=r^2-1$, which is impossible. Thus we conclude that $Y = \la H, \tau \ra = X$, as required. This completes the proof of Theorem~\ref{k1eg}.
	\end{proof}

	\section{Factorisations of $A_n,\,S_n$: Proof of Theorem \ref{t:Anfactns}}\label{ansnfactn}
	
	We now give the proof of Theorem~\ref{t:Anfactns}. It relies on Lemma~\ref{lem:k2-B2hom} which is proved below.
	
	\vspace{4mm}
	\no {\bf Proof of Theorem \ref{t:Anfactns}}
	
	Let $X=A_n$ or $S_n$  ($n\geq5$), and suppose that $X=AB$ with $A, B$ proper subgroups of $X$ not containing $\soc(X)$, and with  $A\cap B$ cyclic or dihedral. Let $\Omega=\{1,\ldots,n\}$, and let $x,y,z,w,v$ be pairwise distinct points of $\Omega$. As noted in the preamble to Theorem \ref{t:Anfactns}, apart from some exceptional cases with $n\le 10$, we can assume the following holds for some $k$:
	\begin{equation}\label{e:gencaserep}
		A_{n-k}\unlhd A\leq S_{n-k}\times S_k, \ \mbox{with $B$ $k$-homogeneous on $\Omega$, for some $k\in\{1,\dots,5\}$.}
	\end{equation}
	We shall now assume this, and leave the exceptional cases to the end of the proof.
	
	\vspace{2mm}
	(a) Suppose first that $k=1$. If $(X,A)=(S_n,S_{n-1})$ or $(A_n,A_{n-1})$, then $A = X_x$ for some $x\in\Omega$, so $X=AB$ if and only if $B^\Omega$ is transitive, and we have $A\cap B$ cyclic or dihedral as in case a.1(i) of Theorem \ref{t:Anfactns}. The only other possibility is $(X,A)=(S_n,A_{n-1})$. Here $A$ is the stabiliser in $A_n$ of some point $x\in\Omega$ and $|X_{x}:A|=2$. Also $X=AB$ implies that $B^\Omega$ contains an odd permutation, and that $X=X_xB$, whence $B^\Omega$ is transitive. Thus $B_x = X_x\cap B\geq A\cap B$, and we have  $|B_x:A\cap B|=|B:A\cap B|/|B:B_x|=|AB:A|/|B:A\cap B|= 2n/n=2$. This implies that $A_n=A(B\cap A_n)$ so $B\cap A_n$ is transitive on $\Omega$.  Thus conclusion a.1(ii) of Theorem \ref{t:Anfactns} holds since $(B\cap A_n)_x=A\cap B$.
	
	(b) Now assume that $k=2$.  If $A=(S_{n-2}\times S_2)\cap X$ or $(S_{n-2}\times 1)\cap X$, then $X=AB$ implies that $B^\Omega$ is $2$-homogeneous or $2$-transitive, and hence that $A\cap B=B_{\{x,y\}}$ or $B_{x,y}$, respectively. Thus $B$ is given by Lemma~\ref{lem:k2-B2hom} and Table~\ref{tab:B2hom}, and Line 1 or 2 of Table~\ref{tab:Anfactns} holds. These are the only possibilities if $X=A_n$, so we may assume that $X=S_n$, and that $A$ is one of (i) $A=(S_{n-2}\times S_2)\cap A_n$, or  (ii) $A=A_{n-2}\times 1$, or (iii) $A=A_{n-2}\times S_2$. Consider first cases (i) and (ii). Here $A<A_n$, and $A$ is the stabiliser in $A_n$ of on unordered pair $\{x,y\}$ or an ordered pair $(x,y)$ from $\Omega$, respectively. The condition $X=AB$ implies that $A_n=A(B\cap A_n)$, and hence $B\cap A_n$ is $2$-homogeneous or $2$-transitive on $\Omega$, and so $B\cap A_n$  is given by Lemma~\ref{lem:k2-B2hom} and Table~\ref{tab:B2hom}, and Line 3 or 4 of Table~\ref{tab:Anfactns} holds. This leaves case (iii). Here $A\cap A_n = A_{n-2}\times 1$ and $A\not\leq A_n$. If $B\leq A_n$ then $X=AB$ implies that $A_n=(A\cap A_n)B=(A_{n-2}\times 1)B$, and we have just shown that this yields that $(B\cap A_n)^\Omega=B^\Omega$ is $2$-transitive with $A\cap B= (B\cap A_n)_{xy}$ as in Table~\ref{tab:B2hom}. On the other hand, if $B\not\leq A_n$, then we first note that $(A\cap A_n)\cap (B\cap A_n)=(A\cap B)\cap A_n$ has index at most $2$ in $A\cap B$. Hence
	\[
	\frac{|A\cap A_n|\cdot |B\cap A_n|}{|(A\cap A_n)\cap (B\cap A_n)|} = \frac{(|A|/2)\cdot(|B|/2)}{|A\cap B\cap A_n|}
	\geq \frac{|A|\cdot |B|}{2|A\cap B|}=\frac{|AB|}{2} = \frac{|X|}{2} = |A_n|
	\]
	and we deduce that $A_n=(A\cap A_n)(B\cap A_n)$ so again $(B\cap A_n)^\Omega$ is $2$-transitive with $A\cap B\cap A_n= (B\cap A_n)_{xy}$ as in Table~\ref{tab:B2hom}. Thus Line 5  of Table~\ref{tab:Anfactns} holds.
	
	\vspace{2mm}
	(c) Next assume that $k\ge 3$. Now $X=AB$ implies that $X= ((S_{n-k}\times S_k)\cap X)B$, so $B^\Omega$ is $k$-homogeneous. By \cite[Theorem 9.4(B)]{DM}, $k\leq 5$ and either $B^\Omega$ is $k$-transitive or $B^\Omega$ satisfies  \cite[Theorem 9.4(B)(ii)--(iv)]{DM}. We treat the various values of $k$ separately.
	
	Suppose first that $k=3$, so by \cite[Theorem 9.4(B)]{DM} either $B^\Omega$ is $3$-transitive, or $n=q+1$, where $q=p^f$ with $p$ prime, and $\PSL_2(q)\unlhd B\leq \widehat{B}:=\PGaL_2(q)$. In the latter case $\widehat{B}_{xyz}=C_f$, and we have examples for all groups $B$, as in Line 1 of $k=3$ in Table~\ref{tab:Anfactns}. Note that there are some extra restrictions on $A,B$ in this case -- details are given in Remark~\ref{r:Anfactns}(a). 
	For the remaining cases, $B$ is 3-transitive with socle not $\PSL_2(q)$, and $B$ together with the stabiliser $B_{xyz}$ satisfies one of the columns of Table~\ref{tab:k3}. Since $A\cap B\geq B_{xyz}$ and $A\cap B$ is cyclic or dihedral, it follows that $B$ is one of $M_{11}$ of degree $12$, $\AGL_3(2)$, $\AGL_1(8), \AGaL_1(8)$, or $\AGaL_1(32)$.  Each of these yields examples in Table~\ref{tab:Anfactns}.
	\begin{table}
		\caption{$B$ and $B_{xyz}$ with $k=3$ and $\Soc(B)\ne \PSL_2(q)$ }\label{tab:k3}
		\begin{tabular}{l|ccccccc}
			\hline
			$B$         & $M_{11}$ & $M_{11}$  & $M_n$& $2^4.A_7$&$\AGL_d(2)$&$\AGL_1(8)$ or &$\AGaL_1(32)$ \\
			& $n=11$& $n=12$& $n=12,22,23,24$& &&$\AGaL_1(8)$& \\
			\hline
			$B_{xyz}$ & $Q_8$ & $S_3$ & $\geq Q_8$ or  $\geq 2^4$& $A_4$ & $\geq 2^{d-1}$ & $1$ or $C_3$ & $1$\\ 
			\hline
		\end{tabular}
	\end{table}
	
	Suppose next that $k=4$, so by  \cite[Theorem 9.4(B)]{DM} either $B^\Omega$ is $4$-transitive, or $B$ is $\PGL_2(8), \PGaL_2(8)$ or $\PGaL_2(32)$. Thus $B$ and its stabiliser subgroups $B_{xyzw}$ and $B_{\{x,y,z,w\}}$ satisfy  one of the columns of Table~\ref{tab:k4}. Since $A\cap B\geq B_{xyzw}$ and $A\cap B$ is cyclic or dihedral, it follows that $B$ is one of $M_{11}$, $\PGL_2(8)$, $\PGaL_2(8)$, or $\PGaL_2(32)$.  Each of these yields examples as in Table~\ref{tab:Anfactns}. More details for the case $\PGaL_2(8)$ are given in Remark~\ref{r:Anfactns}(b). 
	\begin{table}
		\caption{$B$, $B_{xyzw}$  and $B_{\{x,y,z,w\}}$ with $k=4$ }\label{tab:k4}
		\begin{tabular}{l|ccccc}
			\hline
			$B$         & $M_{11}$  & $M_n$&$\PGL_2(8)$&$\PGaL_2(8)$ &$\PGaL_2(32)$ \\
			& $n=11$& $n=12,23,24$& && \\
			\hline
			$B_{xyzw}$ & $1$ & not cyclic/dihedral & $1$ & $1$ & $1$ \\ 
			$B_{\{x,y,z,w\}}$ & $S_4$ & -- & $V_4$ & $A_4$ & $V_4$ \\ 
			\hline
		\end{tabular}
	\end{table}
	
	Finally we consider $k=5$ where, by  \cite[Theorem 9.4(B)]{DM}, $B=M_{12}$ or $M_{24}$. In the latter case $B_{xyzwv}=[48]$ is not cyclic or dihedral so yields no examples. Thus $B=M_{12}$ and $B_{xyzwv}=1$ with $B_{\{x,y,z,w,v\}}=S_5$, and we obtain the examples in Table~\ref{tab:Anfactns}.
	
	\vspace{2mm}
	(d) It remains to handle the case where (\ref{e:gencaserep}) does not hold. These are exceptional factorisations with $n\le 10$, and we compute with Magma that conclusion (b) of Theorem \ref{t:Anfactns} holds.
	
	\vspace{4mm}
	To complete the proof of Theorem~\ref{t:Anfactns} we need the following lemma.
	
	\begin{lemma}\label{lem:k2-B2hom}
		Suppose that $B\leq \Sym(\Omega)$ with $n=|\Omega|$, that $B$ is $2$-homogeneous and $B\not\geq A_n$, and for distinct points $x,y\in\Omega$, that $B_{\{x,y\}}$ or $B_{xy}$ is cyclic or dihedral. Then one of the Lines of Table~\ref{tab:B2hom} holds; in the table, apart from Line $6$, there is a $2$-homogeneous subgroup $B^\circ \triangleleft B$ as in column $2$ of the table. Moreover column $5$ of Table~\ref{tab:B2hom} contains a $\checkmark$ if and only if some $B$ (with the given $B^\circ$)  contains an odd permutation. 
		
	\end{lemma}
	
	\begin{proof}
		Suppose first that $B$ is an almost simple $2$-homogeneous group on $\Omega$ with $\soc(B)\ne A_n$. By \cite[Theorem 9.4B]{DM}, $B$ is $2$-transitive, and in all cases $B\cap A_n$ is $2$-transitive. First we show that the examples of $2$-transitive groups with $B_{\{x,y\}}$ or $B_{xy}$  cyclic or dihedral are precisely those in Table~\ref{tab:B2hom} with $B^\circ$ simple. If $B$ is a Lie type group of rank 1, with socle $B^\circ=\,^2\!B_2(q)$, $^2\!G_2(q)$, $\PSU_3(q)$ or $\PSL_2(q)$, then we obtain the examples in Table~\ref{tab:B2hom} Lines 1--4.
		If $B$ has socle $\PSL_d(q)$ with $n=(q^d-1)/(q-1)$ and $d\geq3$, then the only example with the cyclic/dihedral property is $B=\PSL_3(2)\cong \PSL_2(7)$ with $B_{\{x,y\}}=D_{8}$ and $B_{xy}=2^2$ as in  Table~\ref{tab:B2hom} Line 7. 
		If $B= \Sp_{2d}(2)$ acting on the set of quadratic forms on $\mathbb{F}_2^{2d}$ of type $\varepsilon=\pm$, then $B_{xy}$ involves $\Sp_{2d-2}(2)$ so  the cyclic/dihedral property implies that $d=2$, and as $B\not\geq A_n$ we have $n=10$ and $\soc(B)\cong \PSL_2(9)$ as in Table~\ref{tab:B2hom} Line 4. 
		The remaining almost simple $2$-transitive groups do not lie in an infinite family and are: $M_n$ with $n\in\{11,12,22,23,24\}$, $M_{11}$ with $n=12$, $A_7$ with $n=15$, $\PSL_2(11)$ with $n=11$, $HS$ with $n=176$, and $Co_3$ with $n=276$.  It is straightforward to check each of these actions and the only example is $B=\PSL_2(11)$ with $B_{\{x,y\}}=D_{12}$ and $B_{xy}=S_3$ as in Table~\ref{tab:B2hom} Line 8.  
		
		We now determine the entries $\checkmark$ or $\times$ in column 5 of Table~\ref{tab:B2hom} for these groups.
		Observe that, in all cases  $B_{\{x,y\}}=D_{2r}$ for some $r$. 
		We make the following observation about existence of odd permutations in $B$ (still under the assumption that $B$ is almost simple).
		
		\medskip
		\noindent
		\emph{Claim:} The following are equivalent:
		\begin{itemize}
			\item [(i)] $B$ contains an odd permutation;
			\item[(ii)] $B_{xy}$ contains an odd permutation;
			\item[(iii)] $B_{\{x,y\}}=D_{2r}$ with $r$ even, and $B_{\{x,y\}}$ contains an odd permutation of order $2$.
		\end{itemize}
		
		\smallskip\noindent\emph{Proof of Claim.} The equivalence of (i) and (ii) holds since $B\cap A_n$ is $2$-transitive on $\Omega$. 
		Obviously (iii) implies (i), so to complete the proof we show that (ii) implies (iii). 
		So assume that $B_{xy}$ contains an odd permutation. We have $B_{\{x,y\}}= \langle g\rangle.\langle h\rangle=C_r.2=D_{2r}$ for some $r$. If $r$ were odd then $B_{xy}$ would be the unique index $2$ subgroup $B_{xy}=\langle g\rangle$ of odd order $r$, which implies that  $B_{xy}<A_n$, a contradiction. Thus  $r$ is even. If $B_{xy}\cong D_r$  then without loss of generality  $B_{xy}=\langle g^2,h\rangle$, and since $g^2\in A_n$ any odd permutation in $B_{xy}$ is an involution. Thus we may assume that  $B_{xy}=\langle g\rangle$, so $B_{xy}\cap A_n=\langle g^2\rangle$ and without loss of generality $B_{\{x,y\}}\cap A_n= \langle g^2,h\rangle\cong D_{r}$ (since  $\langle g\rangle$ contains odd permutations). In this final case the involution $gh\in B_{\{x,y\}}\setminus A_n$, proving the Claim.
		
		\smallskip
		We now verify the entries in column 5 of Table~\ref{tab:B2hom} for these groups. For $B^\circ=\, ^2\!B_2(q)$ or $^2\!G_2(q)$, $|B:B^\circ|$ is odd and hence $B\leq A_n$. If $B^\circ= \PSU_3(q)$, 
		then $\PGU_3(q)\leq A_n$ and an involutory odd permutation $\sigma$ in $B_{\{x,y\}}$ would be a field automorphism;  however $\sigma$ fixes $q+1$ points 
		of $\Omega$ and so has $(q^3-q)/2=q(q^2-1)/2$ cycles of length $2$, which is even. Hence $B\leq A_n$.  Next consider $B^\circ= \PSL_2(q)$. If $q$ is odd, then $\PGL_2(q)_{xy}$ contains a $(q-1)$-cycle which is odd, hence the $\checkmark$ in Table~\ref{tab:B2hom}. Now assume $q$ is even, and suppose $t \in B$ is an odd involution. Since $\PGL_2(q)\leq A_{q+1}$, $t$ is a field automorphism, so $t$ fixes $q_0+1$ points and has $(q-q_0)/2=q_0(q_0-1)/2$ cycles of length $2$. But this implies that $t$ is even (noting that $q_0\ne2$ since $B\not\geq A_n$), a contradiction.
		Finally, in Lines 7 and 8 of Table \ref{tab:B2hom}, we have $B=B^\circ\leq A_n$. Thus the lemma is proved in the case where $B^\circ$ is almost simple. 
		
		\vspace{2mm}
		Assume now that $B$ is an affine $2$-homogeneous group. By \cite[Theorem 9.4B]{DM} either $B$ is $2$-transitive or $\ASL_1(q)\leq B\leq \AGaL_1(q)$. In both cases we have  $n=q^d\geq 5$ and $B_x \leq\Gamma L_d(q)$. 
		
		If $d=1$ then $B\leq \AGaL_1(q)$ and we have the examples in Line 6 of Table~\ref{tab:B2hom}. If $q$ is odd, a Singer cycle in $\GL_1(q)$ acts as a $(q-1)$-cycle, hence is an odd permutation. And if $q = 2^m$, then $\AGaL_1(q) \le \AGL_m(2)$, which is contained in $\Alt(2^m)$ provided $m>2$ (which is the case, as $n=q>4$). This justifies the entry in column 5 of Table \ref{tab:B2hom} for this case.

		Suppose next that $d \ge 3$ (and $B_x \leq\Gamma L_d(q)$). Then by \cite[Appendix]{lieb}, one of the following holds:
		\begin{itemize}
			\item[(i)] $B_x\unrhd \SL_d(q),\, \Sp_d(q)$ or $G_2(q)\,(d=6)$; 
			\item[(ii)] $n=3^4$ and $R:=2^{1+4}\unlhd B_x$;
			\item[(iii)] $n=3^6$ and $\SL_2(13) \unlhd B_x$;
			\item[(iv)] $n=2^4$ and $B_x = A_6$ or $A_7$.
		\end{itemize}
		Since $B_{x,y}$ is cyclic or dihedral, the only possibilities to give examples are (ii) and (iii), and these are recorded in Table~\ref{tab:B2hom}.
		
		Assume finally that $d=2$. In this case the only infinite family has $B^\circ=\ASL_2(p)$ with $B^\circ_{x,y}=Z_p$, $B^\circ_{\{x,y\}}=D_{2p}$. Further, if we take $B$ to be generated by $B^\circ$ together with the matrix $t:=\hbox{diag}(1,-1)$, then $B_{xy} \cong D_{2p}$, $B_{\{x,y\}} \cong D_{4p}$, and $t$ induces an odd permutation of $\F_p^2$ if and only if $p \equiv 3 \hbox{ mod }4$, as claimed in Table~\ref{tab:B2hom}.
		In the remaining cases, $p^d$ is $5^2,\ldots ,59^2$, and we compute with Magma a minimal example $B^\circ$ as in 
		Table~\ref{tab:B2hom}, together with a computation deciding whether there is a larger example $B$ containing an odd permutation.
	\end{proof}

	\begin{remark}\label{r:Anfactns}
		{\rm
			(a) 
			There are additional restrictions in some cases on the groups $A, B$ in Theorem~\ref{t:Anfactns}. The first of these is in case (a.2) with $k=3$ and $\Soc(B)=\PSL_2(q)$ (with $n=q+1$ where $q=p^f$ with $p$ prime and $f\geq1$). Here $B\leq \widehat{B}:=\PGaL_2(q)$, and $\widehat{B}_{xyz}=C_f$, so we have examples with $A=(S_{n-3}\times 1)\cap X$ for all $B$ satisfying $\PSL_2(q)\unlhd B\leq \widehat{B}$, with $A\cap B$ cyclic of order dividing $f$. Since $\widehat{B}_{\{x,y,z\}}=S_3\times C_f$, we also obtain dihedral intersections when $A=(S_{n-3}\times S_3)\cap X$, namely $A\cap B=D_6$ if $B=\PGL_2(q)$, and  we can also get $A\cap B=D_{12}$ if $f$ is even. The largest cyclic groups $A\cap B$ arising are $C_{3f}$ if $\gcd(f,3)=1$ and   $C_{2f}$  if $\gcd(f,2)=1$. Note that $B\ne \Soc(B)$ if  $q\equiv 1\pmod{4}$ since $B^\Omega$ must be $3$-homogeneous.
			
			(b) In case (a.2) with $k=4$ and $B=\PGaL_2(8)$, we have $B\leq A_9$ and $B_{\{x,y,z,w\}}=A_4$. Thus $(S_5\times S_4)\cap B=A_4$, and to obtain examples with a cyclic or dihedral intersection, we take $A=(S_5\times S)\cap X$ for certain proper subgroups $S<S_4$ such that $S_4=SA_4$. The possibilities are $(S, A\cap B)=(S_3,C_3), (D_8,2^2), (C_4,C_2), (\langle (x,y)(z,w)\rangle, C_2)$ and $(\langle (x,y)\rangle, 1)$.  
			
			(c) In case (a.1)(i) we can demonstrate the existence of examples as follows: if $X=S_n$ we may take $B=\langle(1,2,\dots,n)\rangle$; and if $X=A_n$, then we may take $B=Z_n$ if $n$ is odd; while if $n=2^am$ with $a\geq2$ and $m$ odd, then we may take $\Omega = \F_2^a\times Z_m$ and $B=B_0\times Z_m$ with $B_0$ the group of translations of $\F_2^a$. If $n=2m$ with $m\geq 3$ odd, then we take $\Omega=\{1,1',\dots,m,m'\}$ and $B=\langle g,h\rangle$ with $g=(1,2,\dots,m)(1',2',\dots,m')$, and if $m\equiv 3\pmod{4}$ then $h=(m,m')\prod_{i=1}^{(m-1)/2}(2i-1, (2i-1)',2i,(2i)')$, while if $m\equiv 1\pmod{4}$ then $h=(m,m')(1,1')(2,2') \prod_{i=2}^{(m-1)/2}(2i-1, (2i-1)',2i,(2i)')$; note that $B\leq X=A_n$, $B$ is transitive, and $B\cap A_{n-1}\cong C_2$. 
			
			(d) In case (a.1)(ii) with $n\not \equiv 1 \pmod{4}$, there exist examples, as follows. Here $X=S_n$ and $A=A_{n-1}$, and 
			we take $B=\langle g,h\rangle=D_{2n}$, where $g=(1,2,\dots,n)$ and $h=\prod_{i=1}^{n/2} (i,n-i+1)$ if $n$ is even, or $h=\prod_{i=1}^{(n-1)/2} (i,n-i)$ if $n\equiv 3\pmod{4}$. Then $B\cap A_n=\langle g^2,h\rangle$, $\langle g^2,gh\rangle$  or $\langle g\rangle$, $B\cap A_n$ is transitive and $(B\cap A_n)_x=1$. 
			
			However, in case (a.1)(ii) there are values of $n \equiv 1 \pmod{4}$, for which examples do not exist. To see this, take $n= p\equiv 1 \pmod{4}$, chosen such that the only insoluble transitive permutation groups of degree $p$ are $A_p$ and $S_p$. If $B$ is a transitive subgroup of $S_p$ and $B \ne A_p,S_p$, then $B \le \AGL_1(p)$, and since $\AGL_1(p) \le A_p$ when $p\equiv 1 \pmod{4}$, it follows that $B$ contains no odd permutation -- hence cannot satisfy (a.2)(ii) of Theorem \ref{t:Anfactns}.
		}
	\end{remark}
	
	\section{Factorisations of groups of Lie type: Proof of Theorem \ref{classgen}}\label{classfactn}
	
	In this section we prove Theorem \ref{classgen}. Let $X$ be an almost simple group with socle $X_0$, a simple group of Lie type, and suppose that $X = AB$ with $A,B$ core-free and $A \cap B$ cyclic or dihedral. We aim to show that $X$, $A$, $B$, and $A \cap B$ are as in Tables $\ref{families}$ and $\ref{excep-psl}-\ref{excep-sp}$.
	
	We begin with the case where $X_0$ is an exceptional group of Lie type. In this case the factorisations of $X$ are completely determined in \cite{HLS}, where the intersections $A\cap B$ can also be found, and the only cases where this intersection is cyclic or dihedral are those occuring for $X_0 = G_2(q)$ on the last two lines of Table \ref{families}. 
	
	From now on, we assume that $X_0$ is a classical group that is not isomorphic to an alternating group. We make extensive use of the results on factorisations of classical groups in \cite{LX19, LX22, LWX}. These papers divide the consideration of factorisations $X=AB$ into various sub-cases: first, where one factor has two or more non-abelian composition factors; second, where both factors have a unique non-abelian composition factor; and third, where one (or both) factor is soluble. We shall divide our proof into these sub-cases as well.
	
	\subsection{Case where $A$ or $B$ has two or more non-abelian composition factors}\label{twocomp}
	
	As above, assume that $X=AB$ with $A\cap B$ cyclic or dihedral, and $\soc(X) = X_0$ a classical group. In the case where one of the factors $A,B$ has two or more non-abelian composition factors, the factorisations $X = AB$ are given by \cite{LX19}. We record the result in the next lemma, also listing a subgroup $K$ of $A\cap B$.
	
	\begin{lem} Assume that $X = AB$ and that $A$ has two or more non-abelian composition factors. Then $X_0,A,B$ are as in Table $\ref{2ormoresp}$, along with a subgroup $K$ of $A\cap B$.
	\end{lem}

	\begin{table}[h!] \label{2ormoresp}
		\caption{Factorisations $X=AB$ with $X$ classical and $A$ having two nonabelian cf's}
		\[
		\begin{array}{l|l|l|l|l}
			\hline
			\hbox{no.} & X_0 & A\; \triangleright & B\; \triangleright & K \\
			\hline
			1& \Sp_{4l}(q),\,q=2^f & (\Sp_{2a}(q^b))^2,ab=l & \O_{4l}^-(q) & \O_{2a}^-(q^b)\times  \O_{2a}^+(q^b) \\
			2& \Sp_{12l}(q),\,q=2^f & (G_2(q^l))^2 & \O_{12l}^-(q) & \SU_3(q^l) \times \SL_3(q^l) \\
			3& \Sp_{4l}(4),\,l\ge 2 & \Sp_2(4) \times \Sp_{4l-2}(4) & \Sp_{2l}(16) & \Sp_2(4) \times \Sp_{2l-2}(16) \\
			4&                     & \Sp_2(4) \otimes \Sp_{2l}(4) & \O_{4l}^-(4) & \Sp_{2l-2}(4) \\
			5& \Sp_{8l}(2),l\ge 2 & \Sp_2(4) \otimes \Sp_{2l}(4) & \O_{8l}^-(2) & \Sp_{2l-2}(4) \\
			6& \Sp_4(q),\,q=2^f & \O_4^+(q) & ^2\!B_2(q) & q-1 \\
			7& \Sp_6(q),\,q=2^f \ge 4 & U\times \Sp_4(q)\,(U\le \Sp_2(q)) & G_2(q) & \SL_2(q) \\
			8& \Sp_{12}(4) & \Sp_2(4)\times \Sp_{10}(4) & G_2(16) & \SL_2(16) \\
			9&             & \Sp_2(4) \times G_2(4) & \O_{12}^-(4) & \SL_2(4) \\
			10& \Sp_{24}(2) & \Sp_2(4) \times G_2(4) & \O_{24}^-(2) & \SL_2(4) \\
			\hline
			11 & P \O_{4l}^+(q),\,l\ge 2,q\ge 4 & U \times \PSp_{2l}(q)\,(U\le \PSp_2(q)) & \O_{4l-1}(q) & \Sp_{2l-2}(q) \\
			12 & \O_{8l}^+(2) & \Sp_2(4) \times \Sp_{2l}(4) & \Sp_{8l-2}(2) & \Sp_{2l-2}(4) \\
			13 &  \O_{8l}^+(4) & \Sp_2(4^c) \times \Sp_{2l}(16)\,(c\le 2) & \Sp_{8l-2}(4) & \Sp_{2l-2}(16) \\
			14 & \O_{12}(q), q=2^f\ge 4 & U \times G_2(q) ,(U\le \Sp_2(q)) & \Sp_{10}(q)  & \SL_2(q) \\
			15 & \O_{24}^+(2) & \Sp_2(4) \times G_2(4) & \Sp_{22}(2) & \SL_2(4) \\
			16 & \O_{24}^+(4) & \Sp_2(4^c)\times G_2(16) \,(c\le 2) & \Sp_{22}(4) & \SL_2(16) \\
			\hline
		\end{array}
		\]
	\end{table}
	
	\begin{prop}\label{2cf} Theorem $\ref{classgen}$ holds if $A$ has two or more non-abelian composition factors.
	\end{prop}
	
	\begin{proof}
		From Table \ref{2ormoresp}, we see that the only cases where $K$ can be cyclic or dihedral are 
		\begin{itemize}
			\item[(i)] Case 1 with $a=1$,
			\item[(ii)] Case 6,
			\item[(iii)] Case 12 with $l=1$,
			\item[(iv)] Case 13 with $l=1$.
		\end{itemize}
		
		Consider (i): $X_0 = \Sp_{4l}(q)$, $A \triangleright \Sp_2(q^l) \times \Sp_2(q^l)$, $B \triangleright \O_{4l}^-(q)$. If $B$ contains $\Or_{4l}^-(q)$, then $A\cap B \ge \Or_2^+(q^l) \times \Or_2^-(q^l)$, which is not cyclic or dihedral. Hence $B \cap X_0 = \O_{4l}^-(q)$. When $l$ is even, we have $\Or_2^+(q^l) \le \O_{2l}^+(q)$ by \cite[4.3.14]{KL}, so that 
		$A\cap B \ge \Or_2^+(q^l) \times \O_2^-(q^l)$, which is not cyclic or dihedral. Therefore $l$ is odd. At this point we have the example in Table \ref{families}: $l$ odd, $A = \Sp_2(q^l) \wr S_2$, $B = \O_{4l}^-(q)$ with 
		$A\cap B = (\O_2^+(q^l) \times \O_2^-(q^l)).2 \cong D_{2(q^{2l}-1)}$.
		
		Now consider (ii). By \cite[p.96-97]{LPS90}, there is a factorisation $\Sp_4(q) = AB$ with  $A = \Or_4^+(q)$, 
		$B=\,^2\!B_2(q)$ and $A\cap B = D_{2(q-1)}$. This is an example in Table \ref{families}. Finally, 
		this $D_{2(q-1)}$ is contained in $\O_4^+(q)$, so there is no such factorisation $\O_4^+(q)\,^2\!B_2(q)$.
		
		In case (iii), we have $A \triangleright \O_4^+(4)$ and $B = N_1$, and $A\cap B$ contains $N_1(A) = \O_3(4)$, which is not cyclic or dihedral. Similarly, case (iv) does not occur. 
	\end{proof}
	
	\subsection{Case where both $A$ and $B$ have exactly one non-abelian composition factor}\label{onecomp}
	
	In this section we consider factorisations $X = AB$ of classical groups satisfying the hypothesis:
	\begin{equation}\label{bothone}
		\hbox{Both $A$ and $B$ have exactly one non-abelian composition factor.}
	\end{equation}
	Such factorisations are classified in \cite{LWX}. We use this to prove Theorem \ref{classgen} in this case, family by family.
	
	\begin{prop}\label{bothonelinear}
		Theorem $\ref{classgen}$ holds for $X_0 = \PSL_n(q)$ under hypothesis $(\ref{bothone})$.
	\end{prop}
	
	\begin{proof}
		By our initial assumption we exclude $X_0 = \SL_4(2) \cong A_8$ from consideration. 
		
		For $X_0 = \PSL_n(q)$, the factorisations under hypothesis (\ref{bothone}) are given in \cite[Table 3.1]{LWX}, where a subgroup of $A\cap B$ is also exhibited. The only cases where this subgroup is cyclic or dihedral are listed in Table \ref{1cflin}, where the first column lists the case number in  \cite[Table 3.1]{LWX}. In the table we list the minimal examples $(X,A,B)$, as given in \cite{LWX}. All of these are in Table \ref{excep-psl} in the conclusion of Theorem \ref{classgen}.
	\end{proof}

	\begin{table}[h!]
		\caption{$X_0 = \PSL_n(q)$: Possible $A,B$, both with one non-abelian cf} \label{1cflin}
		\[
		\begin{array}{c|l|l|l|l}
			\hline
			\hbox{no. in \cite[Tb.3.1]{LWX}} & X & A & B  & A\cap B \\
			\hline
			5 & \PSL_4(4).2 & \PSL_2(16).4 & \SL_3(4).2 & 1 \\
			6 & \PSL_6(2) & G_2(2) & \SL_5(2) & D_6 \\
			11 & \PSL_4(3) & S_5,\,4.A_5  &  3^3.\SL_3(3) & 3,\,D_6 \\
			12 & \PSL_6(3) & \PSL_2(13) & 3^5.\SL_5(3) & 3 \\
			14 & \PSL_3(4).2 & M_{10} & \SL_3(2).2 & D_6 \\
			\hline
		\end{array}
		\]
	\end{table}
	
	\begin{prop}\label{bothoneunitary}
		Theorem $\ref{classgen}$ holds for $X_0 = \PSU_n(q)\,(n \ge 3)$ under hypothesis $(\ref{bothone})$.
	\end{prop}
	
	\begin{proof}
		As in the previous proof, we refer to \cite[Tables 4.1,4.2]{LWX}, where the factorisations under (\ref{bothone}) are listed, together with the intersections of the factors. The only cases where $A\cap B$ could possibly be cyclic or dihedral are listed in Table \ref{1cfuni}, where the first column refers to the case numbers in  \cite[Tables 4.1,4.2]{LWX}. (Note that case 6 of \cite[Table 4.1]{LWX} gives the factorisation $\SU_6(2) = G_2(2)\,SU_5(2)$ with intersection $\SL_2(2) \cong D_6$, but this descends to a factorisation of $\PSU_6(2)$ with an intersection of order 18 that is not dihedral, hence is not included in Table \ref{1cfuni}.)
		
		All the possibilities in Table \ref{1cfuni}, except that in the first row, appear in Table \ref{excep-psu} in the conclusion of Theorem \ref{classgen}. So it remains to rule out the first row of  Table \ref{1cfuni}. In this case we see from \cite[Lemma 4.6]{LWX} that $c \ge 4b$, and hence $A\cap B$ contains a group $C$ of order $q^{2b+1}$. This subgroup $C$ lies in $A = q^c.\SL_2(q^{2b})$, where the normal subgroup $q^c$ is elementary abelian, and hence $C$ has exponent at most $p^2$. So $C$ cannot be cyclic or dihedral unless $b=1$ and $q=2$. But then $B = \SU_3(2)$ is soluble, a contradiction.
	\end{proof}
	
	\begin{table}[h!] 
		\caption{$X_0 = \PSU_n(q)$: Possible $A,B$, both with one non-abelian cf} \label{1cfuni}
		\[
		\begin{array}{l|l|l|l|l}
			\hline
			\hbox{no. in \cite[Tb.4.1]{LWX}} & X & A & B  & A\cap B \\
			\hline
			1\,(a=2) & \PSU_{4b}(q) & q^c.\SL_2(q^{2b}) & \SU_{4b-1}(q) & [q^{c-2b+1}] \\
			5 & \PSU_4(4).4 & \PSL_2(16).4 & \SU_3(4).4 & 1 \\
			4\,(\hbox{Tb.4.2}) & \PSU_4(3).[4] & 3^4.A_6.[4] & \PSL_2(7).[4] & D_6 \\
			\hline
		\end{array}
		\]
	\end{table}
	
	\begin{prop}\label{bothoneorthog1}
		Theorem $\ref{classgen}$ holds under hypothesis $(\ref{bothone})$ for $X_0 = \O_{2m+1}(q)\,(m \ge 3,\,q \hbox{ odd})$ or $P\Omega_{2m}^-(q)\,(m\ge 4)$.
	\end{prop}
	
	\begin{proof}
		Here we refer to \cite[Tables 5.1, 6.1]{LWX}. The only cases where $A\cap B$ can be cyclic or dihedral occur with $X_0 = \O_{2m+1}(q)$, and are listed in Table \ref{1cforth1}. The first three rows are in Table \ref{families}, and the last row is in Table \ref{excep-orthog}.
	\end{proof}
	
	\begin{table}[h!] 
		\caption{$X_0 = \O_{2m+1}(q)$: Possible $A,B$, both with one non-abelian cf} \label{1cforth1}
		\[
		\begin{array}{c|l|l|l|l}
			\hline
			\hbox{no. in \cite[Tb.5.1]{LWX}} & X & A & B  & A\cap B \\
			\hline
			2 & \O_7(q)\,(q=3^f) & \SL_3(q) & \O_6^-(q) & q^2-1 \\
			3 & \O_7(q)\,(q=3^f) & \SU_3(q) & \O_6^+(q) & q^2-1 \\
			3 & \O_7(q)\,(q=3^{2f+1}) & ^2\!G_2(q) & \O_6^+(q) & q-1 \\
			8 & \O_7(3) & 3^3.\SL_3(3) & A_9 & D_6 \\
			\hline
		\end{array}
		\]
	\end{table}
	
	\begin{prop}\label{bothoneorthog2}
		Theorem $\ref{classgen}$ holds under hypothesis $(\ref{bothone})$ for $X_0 = P\O_{2m}^+(q)\,(m \ge 4)$.
	\end{prop}
	
	\begin{proof}
		In this case, for convenience we handle the case $X_0 = \O_8^+(2)$ using Magma, and exclude it from now on.
		Now we use \cite[Tables 7.1,7.2]{LWX}, and argue from these that the minimal examples $(X,A,B)$ with $A\cap B$ cyclic or dihedral are those given in Table \ref{1cforth2}. These are all in Table \ref{excep-orthog}, so this will complete the proof of the proposition.
		
		First consider \cite[Table 7.1]{LWX}. The fact that $A\cap B$ is cyclic or dihedral, together with the exclusion of $X_0 = \O_8^+(2)$ implies that only rows 1,2,3,5,7 and 23 of Table 7.1 can occur. We handle these one by one. 
		
		In row 1 of \cite[Table 7.1]{LWX}, we have $X_0 = P\O_{2ab}^+(q)$ and $A\cap B = [q^{c-b+1}].\SL_{a-1}(q^b)$ or 
		$[q^{c-b+1}].\Sp_{a-2}(q^b)$, where $a\ge 2$ and $c$ is given by \cite[Prop. 7.26]{LWX}. The latter proposition also gives that either $c\ge a^2b$ or $b=1$. In the former case $c\ge 4b$ and $A\cap B$ contains $[q^{3b+1}]$, which cannot be cyclic or dihedral. And if $b=1$, then $a\ge 4$ and the cyclic or dihedral condition implies that $a=4$, $q=2$, hence 
		$X_0 = \O_8^+(2)$ , which we have excluded.
		
		Row 2 of \cite[Table 7.1]{LWX} is handled in entirely similar fashion: again, the only possibility is that $X_0 = \O_8^+(2)$ , which we have excluded.
		
		Now consider row 3 of \cite[Table 7.1]{LWX}. Here $X = \O_{2ab}^+(4).2$ and $A\cap B = [4^{c-b+2}].\SL_{a-1}(4^b)$ or 
		$[4^{c-b+2}].\Sp_{a-2}(4^b)$, where $a\ge 2$ and $b$ is even.  Moreover, \cite[Prop. 7.27]{LWX} implies that either $c \ge a^2b$ or $b=2$. In the former case, $A\cap B$ contains $[4^{3b+2}]$, which cannot be cyclic or dihedral. Hence $b=2$, and we must have $a=2$ and $c\le 1$. This gives the possibility in row 1 of Table \ref{1cforth2}.
		
		Rows 5,7 and 23 of \cite[Table 7.1]{LWX} lead to rows 4,2 and 3 of Table \ref{1cforth2}, respectively.
		
		Finally, the only possibility in  \cite[Table 7.2]{LWX} with  $A\cap B$ cyclic or dihedral is in row 9 of that table, which is included in row 1 of Table \ref{1cforth2}. 
	\end{proof}
	
	\begin{table}[h!] 
		\caption{$X_0 = P\O_{2m}^+(q)$: Possible $A,B$, both with one non-abelian cf} \label{1cforth2}
		\[
		\begin{array}{l|l|l|l}
			\hline
			X & A & B  & A\cap B \\
			\hline
			\O_8^+(4).2 & 2^{2c}.\SL_2(16).4\,(c\le 1) & \Sp_6(4).2 & 2^{2c} \\
			& \SL_2(16).8 & \Sp_6(4.2 & 2 \\     
			& \O_8^-(2).2 & \O_6^-(4).4 & D_{12} \\
			\O_{12}^+(2) & G_2(2) & \Sp_{10}(2) & D_6 \\
			\hline
		\end{array}
		\]
	\end{table}
	
	\begin{prop}\label{bothonesp}
		Theorem $\ref{classgen}$ holds under hypothesis $(\ref{bothone})$ for $X_0 = \PSp_{2m}(q)\,(m \ge 2)$.
	\end{prop}
	
	\begin{proof}
		For convenience we handle the cases $X_0 = \Sp_4(2),\,\Sp_6(2),\,\Sp_8(2)$ using Magma, and exclude these from now on.
		Now we use \cite[Tables 8.1,8.2]{LWX}, and argue from these that the minimal examples $(X,A,B)$ with $A\cap B$ cyclic or dihedral are those given in Table \ref{1cfsp}.
		
		First consider \cite[Table 8.1]{LWX}. The fact that $A\cap B$ is cyclic or dihedral, together with the exclusions of $X_0 = \Sp_{2m}(2)$ for $m=2,3,4$, implies that only rows 1-3, 7-14, 17 and 20 of Table 8.1 can occur. We handle these one by one. 
		
		In row 1 of \cite[Table 8.1]{LWX}, we have $X_0 = \PSp_{2ab}(q)$, $A = \PSp_{2a}(q^b)$ and $A\cap B = [q^d].\Sp_{2a-2}(q^b)$, where $d = (2a-1)b$ (see \cite[8.14]{LWX}). The subgroup $[q^d]$ of $A\cap B$ has exponent $p$, so we must have $q=2$ and $d=2$. But this forces $a=1,b=2$ and $X_0 = \Sp_4(2)$, which is not the case.
		
		Row 2 of \cite[Table 8.1]{LWX} has $X_0 = \Sp_{2ab}(q)\,(q=2^f)$, $A = \Sp_{2a}(q^b)$, $B = \O_{2ab}^+(q).(2,b)$ and $A\cap B = \O_{2a}^+(q^b).(2,b)$. The cyclic/dihedral condition forces $a=1$, and this case is included in Table \ref{1cfsp}. 
		
		Row 3 of \cite[Table 8.1]{LWX} coincides with row 4 of our Table \ref{1cfsp}.
		
		Now consider rows 7 and 8 of \cite[Table 8.1]{LWX}. Here $X_0 = \Sp_{4b}(q)$ with $q$ even, $A = q^c.\Sp_2(q^b).r$ ($r = 1$ or $b_2$), and $A\cap B$ contains a subgroup $[q^{c-b}]$, where $c \ge  3b$ (see \cite[8.47,8.48]{LWX}). Moreover, this subgroup $[q^{c-b}]$ has exponent at most 4. The cyclic/dihedral condition then forces $b=1$, $q=2$, so $X_0 = \Sp_4(2)$, a contradiction.
		
		In rows 9 and 10 of  \cite[Table 8.1]{LWX}, we have $X_0 = \Sp_{4b}(2)$ and $A\cap B$ contains a subgroup $[2^{c-b+1}]$ of exponent at most 4, where $c \ge 2b$ (see \cite[8.49,8.50]{LWX}). Hence the cyclic/dihedral condition forces $b\le 2$, which is not the case. Row 11 is similar: here $X_0 = \Sp_{4b}(4)$ with $b$ even, and $A\cap B$ contains a subgroup $[4^{c-b+1}]$ of exponent at most 4, where $c \ge 2b$. Hence $A\cap B$ cannot be cyclic or dihedral.
		
		In row 12 of \cite[Table 8.1]{LWX}, we have $X_0 = \Sp_{2ab}(q)$ ($q,b$ even), $A = \Sp_{2a}(q^b).b_2$, $B = \O_{2ab}^-(q)$ and $A\cap B = \Or_{2a}^-(q^b).(b_2/2)$. Hence $a=1$ and $b_2=2$, which is row 3 of Table \ref{1cfsp}. Similarly, row 14 of \cite[Table 8.1]{LWX} leads to row 1 or 2 of Table \ref{1cfsp}. 
		
		Row 13 of \cite[Table 8.1]{LWX} is row 5 of Table \ref{1cfsp}.
		
		Finally, consider row 17 of \cite[Table 8.1]{LWX}. Here $X_0 = \Gamma \Sp_{2m}(4) = \Sp_{2m}(4).2$, and the cyclic/dihedral condition on $A\cap B$ forces $m=2$ and one of the possibilities in rows 7,8 of Table \ref{1cfsp} to occur. Row 20 of 
		\cite[Table 8.1]{LWX} leads to the same examples in Table \ref{1cfsp}, where a graph automorphism of $X_0$ is applied to both factors $A,B$.
		
		This completes the analysis of the entries of \cite[Table 8.1]{LWX}. 
		
		Now we consider \cite[Table 8.2]{LWX}. In this table, the possibilities for $X, A, B$ and $A\cap B$ are listed explicitly, and we can read off that the only cases with $A\cap B$ cyclic or dihedral are those in Table \ref{1cfsp} with $X_0 = \PSp_4(p)$, $\Sp_4(4)$, 
		$\PSp_6(3)$, $\Sp_6(4)$ or $\Sp_{12}(2)$.
		
		Note finally that the first five rows of Table \ref{1cfsp} are included in Table \ref{families}, and the rest are in Table \ref{excep-sp}.
	\end{proof}

	\begin{table}[h!] 
		\caption{$X_0 = \PSp_{2m}(q)$: Possible $A,B$, both with one non-abelian cf} \label{1cfsp}
		\[
		\begin{array}{l|l|l|l}
			\hline
			X & A & B  & A\cap B \\
			\hline
			\Sp_{2b}(q),\,q=2^f & \Sp_2(q^b) & \Or_{2b}^\e(q) & D_{2(q^b-\e)} \\
			&       & \O_{2b}^\e(q)\,(b \hbox{ odd}) & q^b-\e \\
			& \Sp_2(q^b).2\,(b_2=2) & \O_{2b}^-(q) & D_{2(q^b+1)} \\
			\Sp_{4b}(q),\,q=2^f,\,bf\hbox{ odd} & ^2\!B_2(q^b) & \Or_{4b}^+(q) & D_{2(q^b-1)} \\
			\Sp_{8b}(q),\,q=2^f,\,bf\hbox{ odd} & \Sp_4(q^b).2 & \O_{8b}^-(q) & D_{2(q^b-1)} \\
			\PSp_4(p),\,p=11,29,59 & \PSp_2(p^2) & p^{1+2}.2A_5 & p \\
			\Sp_4(4).2 & A_5.a\,(a=2,4) & \Or_4^-(4).2 & a/2 \\
			& A_6.a\,(a=2,4) & \Or_4^-(4).2 & D_{3a} \\
			\PSp_6(3) & \PSL_2(13) & P_1 & 3 \\
			& 3^{1+4}.2^{1+4}.A_5 & \PSL_2(27).3 & 3 \\
			\Sp_6(4).2 & \PSU_3(3).2 & \Or_6^-(4).2 & D_6 \\
			& \SL_2(16).4 & G_2(4).2 & 1 \\
			\Sp_{12}(2) & \PSU_3(3).2 & \Or_{12}^-(2) & D_6 \\
			\hline
		\end{array}
		\]
	\end{table}

	\subsection{Case where $A$ or $B$ is soluble}\label{solcase}
	
	In this section we consider factorisations $X = AB$ of classical groups, where one or both of the factors $A,B$ is soluble. Our starting point is \cite[Thm. 1.1]{LX22}, which states that for such a factorisation, one of the following holds:
	\begin{itemize}
		\item[(I)] we have $A\cap X_0 \le H$ and $B \cap X_0 \triangleright K$, where $(X_0,H,K)$ are listed in \cite[Table 1.1]{LX22}, and reproduced in Table \ref{9inf} below;
		\item[(II)] $(X,A,B)$ is one of 28 exceptional cases listed in \cite[Table 1.2]{LX22};
		\item[(III)] both $A$ and $B$ are soluble, and $(X,A,B)$ is given by \cite[Prop. 4.1]{LX22} (one infinite family with $X_0 = \PSL_2(q)$ and 11 exceptional cases).
	\end{itemize}
	
	In Table \ref{9inf} we use the standard notation $P_i$ for the parabolic subgroup obtained by deleting node $i$ of the Dynkin diagram.
	
	We shall first deal with cases (II) and (III), and then handle each of the nine families under (I).
	
	\begin{table}[h!] 
		\caption{Factorisations $X=AB$ with $A$ soluble} \label{9inf}
		\[
		\begin{array}{l|l|l|l|l}
			\hline
			\hbox{Case} &  X_0 & H & K  & \hbox{Conditions} \\
			\hline
			1 & \PSL_n(q) & \frac{q^n-1}{(q-1)d}.n & q^{n-1}.\SL_{n-1}(q) & d = (n,q-1) \\
			2 & \PSL_4(q) & q^3.\frac{q^3-1}{d}.3 < P_1 \hbox{ or }P_3 & \PSp_4(q) & d=(4,q-1) \\
			3 & \PSp_{2m}(q) & q^{\frac{1}{2}m(m+1)}.(q^m-1).m < P_m & \O_{2m}^-(q) & m\ge 2, q \hbox{ even} \\
			4 & \PSp_4(q) & q^3.(q^2-1).2 < P_1 & \Sp_2(q^2) & q \hbox{ even} \\
			5 & \PSp_4(q) & q^{1+2}.\frac{q^2-1}{2}.2 < P_1 & \PSp_2(q^2) & q \hbox{ odd} \\
			6 & \PSU_{2m}(q) & q^{m^2}.\frac{q^{2m}-1}{(q+1)d}.m < P_m & \SU_{2m-1}(q) & m\ge 2, d = (2m,q+1) \\
			7 & \O_{2m+1}(q) & [q^{\frac{1}{2}m(m+1)}].\frac{q^m-1}{2}.m < P_m & \O_{2m}^-(q) & m\ge 3,q \hbox{ odd} \\
			8 & P\O_{2m}^+(q) & q^{\frac{1}{2}m(m-1)}.\frac{q^m-1}{d}.m < P_{m-1},P_m & \O_{2m-1}(q) & m\ge 5,
			d=(4,q^m-1) \\
			9 & P\O_8^+(q) & q^6.\frac{q^4-1}{d}.4< P_1,P_3,P_4 & \O_7(q) & d=(4,q^4-1) \\
			\hline
		\end{array}
		\]
	\end{table}
	
	\begin{prop}\label{IImag} Theorem $\ref{classgen}$ holds for $(X,A,B)$ as in $(II)$.
	\end{prop}
	
	\begin{proof} 
		In this case, $(X,A,B)$ is as in  \cite[Table 1.2]{LX22}: the factor $A\cap X_0$ is contained in the soluble subgroup listed in the third column of the table, and the insoluble factor $B \cap X_0$ is given in the fourth column. For all the 28 rows of the table, we compute using Magma the possibilities for $(X,A,B)$ with $A\cap B$ cyclic or dihedral, and these can all be found in Tables \ref{excep-psl}-\ref{excep-sp} of Theorem \ref{classgen}. 
	\end{proof}
	
	\begin{prop}\label{IIIcase} Theorem $\ref{classgen}$ holds for $(X,A,B)$ as in $(III)$.
	\end{prop}
	
	\begin{proof} 
		In this case both $A$ and $B$ are soluble, and $(X,A,B)$ is given by \cite[Prop. 4.1]{LX22}: either $(X,A,B)$ is as in \cite[Table 4.1]{LX22} (which is a subset of our Tables \ref{excep-psl}-\ref{excep-sp}), or the following holds:
		\begin{equation}\label{psl2sol}
			X_0 = \PSL_2(q),\,A\cap X_0 \le D_{2(q+1)/d} \hbox{ and } [q] \triangleleft B\cap X_0 \le [q].((q-1)/d),
		\end{equation}
		where $d = (2,q-1)$. But here $(X,A,B)$ is as in line 1 of Table \ref{families} in Theorem \ref{classgen}.
	\end{proof}
	
	It remains to deal with the possibilities under case (I), where $A\cap X_0 \le H$ and $B \cap X_0 \triangleright K$, where 
	$(X_0,H,K)$ are as in Table \ref{9inf}.
	
	\begin{prop}\label{caseI1} Theorem $\ref{classgen}$ holds for $(X_0,H,K)$ as in line $1$ of Table $\ref{9inf}$.
	\end{prop}
	
	\begin{proof} 
		In this case, $X_0 = \PSL_n(q)$, $A\cap X_0 \le \frac{q^n-1}{(q-1)d}.n$ and $B\cap X_0 \triangleright q^{n-1}.\SL_{n-1}(q)$. Moreover we may assume we are not in case (III), so $B$ is insoluble, and so in particular $n\ge 3$ and $(n,q) \ne (3,2)$, 
		$(3,3)$. Thus $(X,A,B)$ is as in line 1 of Table \ref{families} in Theorem \ref{classgen}.
		
		Note that, as claimed in Theorem \ref{classgen}, there is an example satisfying line 1 of Table \ref{families}, as follows: 
		\[
		X = \PGL_n(q),\; A = P_1,\; B= \frac{q^n-1}{q-1}.r,
		\]
		where $r|n$. In this case $X = AB$ and $A\cap B$ is a cyclic group of order $r$.
	\end{proof}
	
	\begin{prop}\label{caseI2} Theorem $\ref{classgen}$ holds for $(X_0,H,K)$ as in line $2$ of Table $\ref{9inf}$.
	\end{prop}
	
	\begin{proof} 
		Here $X_0 = \PSL_4(q)$, $A\cap X_0 \le q^3.\frac{q^3-1}{d}.3<P_i$ ($i=1$ or 3), and $B \triangleright \PSp_4(q)$. Also, by \cite[Prop. 2.2]{BL}, $A$ has order divisible by $q^3(q^3-1)/(2,q-1)$, and contains the unipotent radical $Q = [q^3]$ of $P_i$. Hence $A\cap B$ contains a subgroup of $Q$ of order $q$, and so the cyclic/dihedral condition implies that either $q=p$ or $q=4$. The case $q=4$ gives an example in Table \ref{excep-psl}.
		
		Now assume that $q=p$. If $p \equiv 1 \hbox{ mod }4$, then $B\cap X_0 = PSp_4(p)$ (see \cite[4.8.3]{KL}) and $|A\cap X_0|$  divides $3q^3(q^3-1)/4$, so $X_0$ does not factorise as $(A\cap X_0)(B\cap X_0)$. On the other hand, \cite[Prop.2.2(ii)]{BL} gives a factorisation $X=AB$, where
		\[
		X = \left\{ \begin{array}{l} X_0.2, \hbox{ if }p \equiv 1 \hbox{ mod }4 \\ X_0, \hbox{otherwise} \end{array} \right. \hbox{ and }
		A = \PSp_4(p).2,\,B = p^3.(p^3-1)/2,
		\]
		and $|A\cap B|=p$. Replacing $B$ by $B.3$ gives another factorisation with $A\cap B = p\times 3$.
		These possibilities are in Table \ref{families}.
	\end{proof}
	
	\begin{prop}\label{caseI3} Theorem $\ref{classgen}$ holds for $(X_0,H,K)$ as in line $3$ or $4$ of Table $\ref{9inf}$.
	\end{prop}
	
	\begin{proof} 
		Observe that line 4 of Table \ref{9inf} is equivalent to line 3 with $m=2$ and a graph automorphism of $X_0 = \Sp_4(q)$ applied to $H$ and $K$, so we need only deal with line 3.
		
		In this case, $X_0 = \PSp_{2m}(q)$ with $m\ge 2$ and $q$ even, $A\cap X_0 \le q^{\frac{1}{2}m(m+1)}.(q^m-1).m < P_m$, and $B \triangleleft \O_{2m}^-(q)$. 
		
		We next justify the existence of the factorisations claimed in line 3 of Table \ref{families}. We start with the factorisation $X_0 = HK$, where
		\begin{equation}\label{beg}
			H = \Or_{2m}^-(q),\;K = \Sp_2(q^m).m,\; H\cap K = \Or_2^-(q^m).m \cong (q^m+1).2m
		\end{equation}
		(see Table \ref{1cfsp} above). Pick a subgroup $K_1 =q^m.(q^m -1).m < K$, and observe that $K = (H\cap K)\,K_1$ with intersection $H \cap K_1 = 2m$, a cyclic group. Thus we have a factorisation 
		\begin{equation}\label{beg1}
			X_0 = AB,\;A =  \Or_{2m}^-(q),\; B = K_1 =q^m.(q^m -1).m,\; A\cap B = C_{2m}.
		\end{equation}
		We can replace $m$ by any of its divisors $m'$ and get an example with intersection $C_{2m'}$, as in line 3 of Table \ref{families}.
		
		Now replace $H$ by $\O_{2m}^-(q)$, and observe that we still have a factorisation $X_0 = HK$ with $K = \Sp_2(q^m).m$, where now $H\cap K = \O_2^-(q^m).m \cong (q^m+1).m$ (this follows from the proof of \cite[4.3.16]{KL}, which shows that 
		$\Or_2^-(q^m).m$ is not contained in $\O_{2m}^-(q)$). Now if we choose $K_1$ as above, we have $H\cap K_1 = m$, and so we have a factorisation 
		\begin{equation}\label{beg2}
			X_0 = AB,\;A =  \O_{2m}^-(q),\; B = K_1 =q^m.(q^m -1).m,\; A\cap B = C_{m}.
		\end{equation}
		Again we can replace $m$ by any of its divisors $m'$, provided $(m')_2 = m_2$, giving an example with intersection $C_{m'}$, as in line 3 of Table \ref{families}. (We suspect that there is no such example in this case with $(m')_2 < m_2$, but have not proved this.)
		
		It remains to prove that all factorisations corresponding to line 3 of Table \ref{9inf} satisfy the conditions of Table \ref{families}. 
		We have $A\cap X_0 \le Q.(q^m-1).m < P_m$, where $Q$ is the unipotent radical of $P_m$ and $P_m = QL$ with $L \cong \GL_m(q)$. Here $Q$ is elementary abelian of order $q^{\frac{1}{2}m(m+1)}$, and $L$ acts on $Q$ as $\GL_m(q)$ acts on $S^2(V)$, where $V = \F_q^m$ is the natural module. Let $M$ be a cyclic subgroup $\GL_1(q^m)$ of $L$, and consider the action of $M$ on $Q$. Write $M = \langle x \rangle$, and let $\overline{V}_m=V_m\otimes \overline{\mathbb{F}}_q$, where $\overline{\mathbb{F}_q}$ is an algebraic closure of $\mathbb{F}_q$. Then $x$ acts on $\overline{V}_m$ as $\hbox{diag}(\lambda, \lambda^q,\dots,\lambda^{q^{m-1}})$, where $\l \in  \overline{\mathbb{F}}_q$ and $|\lambda|=q^m-1$. Hence $x$ acts on $S^2 \overline{V}_m$ as a diagonal matrix with eigenvalues $\l^{2q^i}$ and $\l^{q^i+q^j}$ for $1\le i<j\le m$. It follows that if $m$ is odd then the $\F_q$-composition factors of $Q\downarrow M$ all have dimension $m$, while if $m$ is even then there is one composition factor of dimension $m/2$ (with eigenvalues $\l^{q^i+q^{i+\frac{m}{2}}}$ for $1\le i\le m/2$), and the rest have dimension $m$.
		
		We know that $|A|$ is divisible by $q^m(q^m-1)$. Assume that the primitive prime divisors $q_m$ and also $q_{m/2}$ (if $m$ is even) exist. These divide $|A|$, which therefore has $\F_q$-composition factors on $Q$ of dimensions $m$, $m/2$ as above.
		Hence $A\cap Q$ either has order $q^{m/2}$ or $q^m$, or it has order $\ge q^{3m/2}$. 
		If $|A\cap Q| = q^m$ then $A \cap X_0 \le q^m.(q^m-1).m$, and $(X,A,B)$ are as in Table \ref{families}, as required.
		If $|A\cap Q| = q^{m/2}$, then we must have $q^{m/2}$ dividing $|X/X_0|$, hence dividing $\log_2q$, which is impossible. 
		And if  $|A\cap Q| \ge  q^{3m/2}$, then we have $|A\cap B\cap Q| \ge q^{m/2}$, and so the cyclic/dihedral condition 
		forces $(m,q) = (2,4)$ or $(4,2)$; at this point we compute using Magma that this gives an example for $X_0 = \Sp_4(4)$ (line 3 of Table \ref{excep-sp}), but no example for $X_0 = \Sp_8(2)$.
		
		Finally, if either $q_m$ or $q_{m/2}$ does not exist, then $(m,q) = (6,2)$ or $(12,2)$. But in these cases $A$ contains a full Singer cycle $2^m-1$ of $L = \GL_m(2)$, so has composition factors on $Q$ of dimensions $m$ and $m/2$, and the above argument goes through. 
	\end{proof}
	
	\begin{prop}\label{caseI5} Theorem $\ref{classgen}$ holds for $(X_0,H,K)$ as in line $5$ of Table $\ref{9inf}$.
	\end{prop}
	
	\begin{proof} 
		Here $X_0 = \PSp_4(q)$ with $q$ odd, $A\cap X_0 \le H:=Q.\frac{q^2-1}{2}.2 < P_1$ with $Q = q^{1+2}$, and $B \triangleright K:=\PSp_2(q^2)$. Also by \cite[Prop. 2.6]{BL}, $|A|$ is divisible by $q^3(q^2-1)$, so $A\cap B$ contains a subgroup of order $q$ which must be elementary abelian, and so the cyclic/dihedral condition forces $q=p$.
		
		For convenience, we shall work in $Y = \Sp_4(p)$, where $Y/\la -I \ra = X_0$, and replace $H$ and $K$ by their preimages in $Y$. We have $P_1 = QL$ with $L = \Sp_2(p) \times (p-1)$, and $H = QD$ with $D = D_0 \times C$, where $D_0$ is a subgroup of order $2(p+1)$ in $\Sp_2(p)$ and $C = p-1$. Taking a standard basis $e_1,e_2,f_2,f_1$, with $P_1$ the stabilizer of $\la e_1\ra$, elements of $C$ act as $\hbox{diag}(\a,1,1,\a^{-1})$, while elements of $D_0$ fix $e_1$ and $f_1$. Also $K\cap P_1 = Q_1E$, where $|Q_1|=p^2$ and elements of $E$ act as $\hbox{diag}(\b,\b,\b^{-1},\b^{-1})$.
		
		Let $x \in H\cap K$, and write $x = q_1e = q_2dc$, where $q_1 \in Q_1$, $e \in E$, $q_2 \in Q$, $d \in D_0$ and $c \in C$.
		Considering actions on $W = \la e_2,f_2\ra$, and noting that $c$ acts trivially on $W$, we have $d^W = (ue)^W$, where $u^W$ is lower unitriangular and $e^W = \hbox{diag}(\b,\b^{-1})$. Hence the order of $d$ divides $p-1$, and so it divides $\hbox{gcd}(p-1,\,2(p+1))$. 
		
		If $p \equiv 3 \hbox{ mod }4$, it follows that $d$ has order 1 or 2, hence is $\pm I_2 \in \Sp_2(p)$. Therefore $e = \pm I$ and so $|H \cap K| = 2p$ and we have a factorisation $Y = HK$. Taking images in $X_0$, this gives a factorisation $X_0 = AB$ with $A\cap B = p$; and if we replace $B$ with $\PSp_2(q^2).2$, we obtain a factorisation with $A\cap B = D_{2p}$.
		
		Now assume that $p \equiv 1 \hbox{ mod }4$. In this case $d$ can have order 4, and we get $|H \cap K| = 4p$, so $Y \ne  HK$. There is no factorisation of $X_0$ here, but there is a factorisation of $X_0.2 = \PGSp_4(p)$, as shown in \cite[Prop. 2.6]{BL}.
		
		Thus $X_0$, $A$, $B$ and $A \cap B$ are as in Table \ref{families}, and all assertions in Theorem \ref{classgen} for this case are now proved.
	\end{proof}
	
	\begin{prop}\label{caseI6} Theorem $\ref{classgen}$ holds for $(X_0,H,K)$ as in line $6$ of Table $\ref{9inf}$.
	\end{prop}
	
	\begin{proof} 
		In this case, $X_0 =  \PSU_{2m}(q)$, $A \cap X_0 \le Q.\frac{q^{2m}-1}{(q+1)d}.m < P_m$  with $Q = q^{m^2}$, and $B \triangleright \SU_{2m-1}(q)$. Moreover, $|A|$ is divisible by $q^{2m}(q^{2m}-1)/(q+1)$, by \cite[Prop. 2.7]{BL}.
		
		As in the proof of Proposition \ref{caseI3}, we see that as a $\GL_1(q^{2m})$-module over $\F_q$, the composition factors of $Q$ have dimension $2m$, together with one of dimension $m$ if $m$ is odd. Hence either $|A\cap Q| \le q^m$, or 
		$|A\cap Q| = q^{2m}$, or $|A\cap Q| \ge q^{3m}$. In the first case $q^m$ must divide $|X/X_0|$, which is impossible; and in the last case, $A \cap B$ contains a subgroup of order $q^{m+1}/|X:X_0|_p$, which cannot be cyclic or dihedral.
		
		Thus $|A\cap Q| = q^{2m}$. Moreover, $A\cap B$ has a subgroup of order $q$, so the cyclic/dihedral condition implies that either $q=p$ or $q=4$. Finally, \cite[Prop. 2.7]{BL} demonstrates the existence of a factorisation in this case, namely the factorisation $\PGU_{2m}(q) = AB$ with $A = q^{2m}..\frac{q^{2m}-1}{q+1}$ and $B = N_1$. 
	\end{proof}
	
	\begin{prop}\label{caseI7} Theorem $\ref{classgen}$ holds for $(X_0,H,K)$ as in line $7$ of Table $\ref{9inf}$.
	\end{prop}
	
	\begin{proof} 
		Here $X_0 = \O_{2m+1}(q)$, $A\cap X_0 \le  Q.\frac{q^m-1}{2}.m < P_m$, where $Q = q^{m(m-1)/2}.q^m$ is the unipotent radical of $P_m$, and $B \triangleright \O_{2m}^-(q)$; also $m\ge 3$ and $q$ is odd. By the claim in the proof of \cite[Prop. 2.8]{BL}, for $X$ to factorise as $AB$, it is necessary that $Q \le A$. But then $A\cap B$ has a subgroup of order 
		$q^{m(m-1)/2}$ which cannot be cyclic or dihedral.
	\end{proof}
	
	\begin{prop}\label{caseI8} Theorem $\ref{classgen}$ holds for $(X_0,H,K)$ as in line $8$ of Table $\ref{9inf}$.
	\end{prop}
	
	\begin{proof} 
		In this case, $X_0 = P\O_{2m}^+(q)$, $A\cap X_0 \le Q.\frac{q^m-1}{d}.m < P_{i}$ with $i \in \{m-1,m\}$ and $Q = 
		q^{\frac{1}{2}m(m-1)}$ the unipotent radical of $P_i$, $B \triangleright  \O_{2m-1}(q)$, and $m\ge 5$. Moreover, $|A|$ is divisible by $q^{m}(q^{m}-1)/(2,q-1)$, by \cite[Prop. 2.9]{BL}.
		
		As a $\GL_1(q^{m})$-module over $\F_q$, the composition factors of $Q$ have dimension $m$, together with one of dimension $m/2$ if $m$ is even (see the proof of Proposition \ref{caseI3}). Hence either $|A\cap Q| \le q^{m/2}$, or 
		$|A\cap Q| = q^{m}$, or $|A\cap Q| \ge q^{3m/2}$. In the first case $q^{m/2}$ must divide $|X/X_0|$, which is impossible; and in the last case, $A \cap B$ contains a subgroup of order $q^{\frac{1}{2}m+1}/|X:X_0|_p$, which cannot be cyclic or dihedral.
		
		Thus $|A\cap Q| = q^{m}$. Moreover, $A\cap B$ has a subgroup of order $q$, so the cyclic/dihedral condition implies that either $q=p$ or $q=4$. Finally, \cite[Prop. 2.7]{BL} demonstrates the existence of a factorisation in this case, namely the factorisation $\PSO^+_{2m}(q) = AB$ with $A = q^{m}.\frac{q^{m}-1}{(2,q-1)}$ and $B = N_1$. 
	\end{proof}
	
	\begin{prop}\label{caseI9} Theorem $\ref{classgen}$ holds for $(X_0,H,K)$ as in line $9$ of Table $\ref{9inf}$.
	\end{prop}
	
	\begin{proof} 
		This case is the $m=4$ version of line 8 of Table \ref{9inf}, but allowing a triality automorphism of $X_0$ to be applied to the subgroups $H,K$. Provided $q>2$, the proof of Proposition \ref{caseI8} goes through, and for $q=2$ we use Magma to obtain the conclusion.    \end{proof}
	
	
	\section{The map theorem for the alternating groups}\label{s:anmaps}
	
	In this section we prove Theorem \ref{t:maps-an}. The procedure is to examine the cyclic/dihedral 
	factorisations of alternating and symmetric groups given by Theorem \ref{t:Anfactns}, and for those satisfying the conditions in Hypothesis~\ref{h:map} to determine all corresponding arc-transitive embeddings.  
	For convenience, we restate Theorem \ref{t:maps-an} here:
	
	\begin{thm}\label{p:maps-an} 
		Suppose that Hypothesis $\ref{h:map}$ holds for $X, G, \G$ and involution $g\in G$ with $\soc(X)=A_n$ for some $n\geq5$. Then either 
		\begin{itemize}
			\item[{\rm (i)}]  one of the lines of Table~$\ref{t:anmaps}$ or  $\ref{6711tbl}$ holds; or
			\item[{\rm (ii)}] ${\rm soc}(G) = A_{n-1}$, and there are examples for infinitely many values of $n$.    
		\end{itemize}
	\end{thm}

	We first prove two lemmas about $2$-transitive groups which are relevant for the proof of Theorem~\ref{p:maps-an}.
	
	\begin{lemma}\label{l:aff2t}
		Let $\O$ be a set of size $q=p^f$ for a prime $p$ and integer $f\geq1$ with $q\geq5$,  let $G<\Sym(\O)$ be $2$-transitive of affine type such that $G_\a$ is cyclic or dihedral, for  $\a\in\O$. Then
		
		\begin{enumerate}
			\item[(i)]  $G_\a=C_{q-1}$ is a Singer cycle, and $G=\AGL_1(q)$;
			\item[(ii)] if $q$ is even then $G\leq \Alt(\O)$;
			\item[(iii)] if $q$ is odd then $G\not\leq \Alt(\O)$ and $G\cap \Alt(\O)$ is not $2$-transitive.
		\end{enumerate}
	\end{lemma}
	
	\begin{proof}
		We identify $\O$ with the vector space $\mathbb{F}_p^f$ and choose $\a=0$ so that $H:=G_\a\leq \GL_{f}(p)$ is transitive on $\O\setminus\{\a\}$, and $H$ is cyclic or dihedral.  In particular $p^{f}-1 \mid |H|$, and $H$ is irreducible on $\O$. If $H$ is cyclic then by \cite[Satz II.7.3]{H67}, $H$ is a Singer cycle, and $G=\AGL_1(q)$, as in part (i). 
		
		Now assume that $H$ is dihedral, so $f\geq 2$ and $H=L.2$ with $L$ cyclic. Note that $|L|$ is divisible by $(q-1)/(2,q-1)\geq 2$. 
		We claim that $L$ is irreducible on $\O$. If $(p,f)=(2,6)$ then $Z_{21}\leq L$ is irreducible; while if $(f,p+1)=(2,2^a)$, then $L$ contains $Z_{2(p-1)}$ which is irreducible. If neither of these cases holds, then  $p^{f}-1$ has a primitive prime divisor, $s$ say,  $L$ contains a subgroup of order $s$ which is irreducible on $\O$ (see \cite[Satz II.7.3]{H67}), so the claim is proved. 
		
		By \cite[Satz II.3.10]{H67}, the $L$-action is equivalent to the action of a subgroup of $\GL_1(p^{f})$, that is, $L$ is a subgroup of a Singer cycle $\langle g\rangle$. Thus $L=\langle g^a\rangle$ where $a\mid q-1$, and $a\leq 2$ since $|L|$ is divisible by $(q-1)/(2,q-1)$. The group $H$ is contained in $N_{\GL(f,p)}(L)$, which by \cite[Satz II.7.3(a)]{H67} is the normaliser $\GaL_1(q)$ of the Singer cycle  $\langle g\rangle$. Let $x \in H\setminus L$, so $|x|=2$ as $H$ is a dihedral group, and $x = yz$, where $y \in \langle g\rangle$ (so $y$ centralises  $L$) and $z$ is an automorphism of $\mathbb{F}_{p^f}$.  
		Since $H$ is not cyclic we have $z\ne 1$ and so $f>1$. As $H$ is dihedral, $z$ must be a field automorphism of order $2$, so $f=2b$ say. Hence $(g^a)^x = (g^a)^z = g^{a p^b}$, and since $(g^a)^x=g^{-a}$ (as $H$ is dihedral) it follows that $|g|$ divides $a(p^b+1)$. On the other hand $|g|=p^{2b}-1$ and $a\leq 2$, and hence $p^b\leq 3$. Thus $b=1$ and $p\leq 3$, and as $q=p^f=p^2\geq 5$ it follows that $q=9$ and $a=2$, so $L=\langle g^2\rangle$ and $x=g^jz$ for some $j$. Since $|x|=2$ we have $1=g^jzg^jz = g^{4j}$ and hence $j$ is even (since $|g|=8$). This implies that $H=\langle g^2, z\rangle$, but this group is not transitive on $\O\setminus\{\a\}$, so we have a contradiction. Thus part (i) is proved.
		
		Finally the Singer cycle $H$ contains an odd permutation if and only if $q$ is odd, so part (ii) holds, and if $q$ is odd then $G\cap\Alt(\O)$ has order $q(q-1)/2$, so cannot be $2$-transitive. This completes the proof.
		%
		%
	\end{proof}
	
	\begin{lemma}\label{l:pgl2}
		Let $G=\PGL_2(q)$, for a prime power $q\geq4$, and let $\O$ denote the set of $\binom{q+1}{2}$ unordered pairs from the projective line $\PG_1(q)$, with $G$ acting naturally on $\O$. Then $G\leq \Alt(\O)$. 
	\end{lemma}
	\begin{proof}
		If $q$ is even then $G$ is a nonabelian simple group so $G\leq \Alt(\O)$. Suppose then that $q$ is odd. Then $G\cap \Alt(\O)$ has index at most $2$ in $G$ and hence $\PSL_2(q)\leq\Alt(\O)$. Thus $G\not\leq \Alt(\O)$ if and only if some, and hence all elements in $G\setminus\PSL_2(q)$ induce odd permutations of $\O$.  We write $\PG_1(q)=\{\infty\}\cup \mathbb{F}_q$. Let $\omega\in\mathbb{F}_q^*$ be a primitive element, and consider the element $h\in G$ which acts on $\PG_1(q)$ by fixing $\infty$ and mapping $x\to x\omega$ for $x\in\mathbb{F}_q$. Then $h\in G\setminus\PSL_2(q)$, and $h$ acts on $\PG_1(q)$ with cycle type $1^2\cdot (q-1)^1$. In its action on $\O$, $h$ has one fixed point $\{\infty,0\}$ and all other cycles have length $q-1$ except for a unique cycle of length $(q-1)/2$ consisting of the  pairs $\{\omega^{i},\omega^{i+(q-1)/2}\}$ with $0\leq i<(q-1)/2$.  Thus $h$ on $\O$ has cycle type $1^1\cdot (\frac{q-1}{2})^1\cdot (q-1)^{(q+1)/2}$. If $q\equiv 1\pmod{4}$ then $h$ has $(q+3)/2$ cycles of even length and $(q+3)/2$ is even, so $h\in\Alt(\O)$; if $q\equiv 3\pmod{4}$ then $h$ has $(q+1)/2$ cycles of even length (since here $(q-1)/2$ is odd) and $(q+1)/2$ is even, so again $h\in\Alt(\O)$. 
		It follows that in all cases $G=\PGL_2(q)\leq \Alt(\O)$.
	\end{proof}
	
	\no {\bf Proof of Theorem \ref{p:maps-an}}
	
	\vspace{2mm}
	Let $X, G, g,$ $\G = (\O,E)$ be as in Hypothesis~\ref{h:map}  where $X_0 = A_n$ with $n\geq5$.  Then $G$ does not contain $A_n$, and 
	$X=AB$ where $\{A,B\}=\{X_\a,G\}$. By Theorem~\ref{t:Anfactns}, either $n\in\{6,8,10\}$, or interchanging $A$ and $B$ if necessary we have
	\begin{equation}\label{e:gencaserep1}
		A_{n-k}\unlhd A\leq S_{n-k}\times S_k, \ \mbox{with $B$ $k$-homogeneous on $\Omega$, for some $k\in\{1,\dots,5\}$.}
	\end{equation}
	Since $X = X_\a G$ with $G_\a$ cyclic or dihedral, the possibilities for $(X,G,G_\a)$ were determined in Theorem~\ref{t:Anfactns} and we treat them according to the following subdivision:
	\begin{itemize}
		\item[(a)] \eqref{e:gencaserep1} holds for $X_\a$, and $G$ is as in Thm. \ref{t:Anfactns}(i)-(v);
		\item[(b)] \eqref{e:gencaserep1} holds for $G$, and $X_\a$ is as in Thm. \ref{t:Anfactns}(i)-(v);
		\item[(c)] $n=6,8$ or 10.
	\end{itemize}
	We shall deal with these possibilities case-by-case. For $n\le 10$, we compute using Magma all the possible 
	tuples $(X, X_\a, X_{\a\b}, G, G_\a, g)$ satisfying Hypothesis~\ref{h:map}, and confirm that the result holds for these small cases. We comment on these cases in the proof.   
	
	\subsection{Case (a) with $k=1$ }\label{ak1} Here either $X_\a = S_{n-1}\cap X$, or $(X, X_\a)=(S_n, A_{n-1})$. 
	
	If $X_\a = S_{n-1}\cap X$ then $X$ acts 2-transitively on $\O$ of degree $|\O|=n$, so $\G=K_n$, the complete graph. 
	Being arc-transitive, $G$ is also 2-transitive, and by Hypothesis~\ref{h:map} we have $|G_{\a\b}|\le 2$. If 
	$G_{\a\b} = 1$, then $G$ is sharply 2-transitive on $\O$, and it follows from \cite[Theorem 3.4B]{DM} that $G$ is of affine type so $n$ is a prime-power. And if $G_{\a\b} = C_2$, then a classical result of Ito \cite[Theorem]{Ito} shows that either $(G,\O)$ is affine (in which case $n$ is a prime-power), or $(G,n) = (\PSL_2(5),6)$ or $(\PSL_2(8).3, 28)$. If $G$ is of affine type then, by Lemma~\ref{l:aff2t}, $G=\AGL_1(p^f)$, where $n=p^f$ is a prime power, with $G<A_n$ if and only if $p=2$. 
	In the second case $G<A_6$ so $X=S_6$ or $A_6$, and the third case cannot occur as $G_\a$ is cyclic or dihedral. The infinite family appears in lines 1 and 2 of Table \ref{t:anmaps}, and the case with $n=6$ is in Table~\ref{6711tbl}. 
	
	(Note that for the complete graph $K_n$, the edge-transitive embeddings have been classified completely in \cite[Theorem 1.8]{J21} building on constructions and classifications in \cite{B71, Ja83, JJ85}. In particular, as we have shown, $n$ must be either 6 or a prime-power, and in both cases arc-transitive embeddings exist. If $n=6$ then there is an example with map group $G=\PSL_2(5)\cong A_5$, see \cite[Figure 2]{J21}.)
	
	Now suppose that $(X, X_\a)=(S_n, A_{n-1})$.  Here $|\O| = 2n$, and $X$ has a set $\{\D_1,\D_2\}$ of two  blocks of imprimitivity of size $n$ in $\O$, with $\a\in\D_1$, such that  the subgroup $X^+$ of index $2$ in $X$ fixing both blocks $\D_i$ setwise is $A_n$. Moreover, $X_\a$ is transitive on $\D_1\setminus\{\a\}$, $X_\a$ is the stabiliser of a point $\a'\in\D_2$ and is transitive on $\D_2\setminus\{\a'\}$. The only connected graph $\G = (\O,E)$ on which $X$ acts arc-transitively is $K_{n,n}\setminus nK_2$. As $G$ is vertex transitive it follows that $G\not\leq A_n$ and  the subgroup of $G$ fixing both blocks setwise is $G^+ := G \cap A_n$, of index $2$ in $G$. Arc-transitivity of $G$ implies that $G^+$ induces equivalent 2-transitive actions on the $\D_i$. Thus by Hypothesis~\ref{h:map}, $(G^+,\D_1)$ is a 2-transitive group with the property that, for $\a,\gamma\in\D_1$, $G^+_\a$ is cyclic or dihedral, and $G^+_{\a\gamma}= 1$ or $C_2$.
	
	Since $|G^+_{\a\gamma}| \le 2$, it follows as in the first paragraph of this subsection that either $(G^+,\D_1)$ is affine (in which case $n$ is a prime-power), or $(G^+,n) = (\PSL_2(5),6)$ or $(\PSL_2(8).3, 28)$. 
	The first case is not possible by Lemma~\ref{l:aff2t}, and the last case is not possible, as $\PSL_2(8).3$ cannot be extended by a cyclic group of order 2 contained in $X=S_{28}$. Hence $n=6$ and $(G^+,n) = (\PSL_2(5),6)$. This leads to an example in Table~\ref{6711tbl}; we can check  by a computation in Magma that $G$ contains an  involution $g$ as in Hypothesis~\ref{h:map}(b) such that all of the conditions (i)-(iii) of Proposition \ref{cri} hold, 
	
	
	
	\subsection{Case(a) with $k=2$ } Here $A_{n-2} \le X_\a \le (S_{n-2}\times S_2)\cap X$ with $G$ 2-homogeneous as in Table \ref{tab:B2hom}. 
	
	\subsubsection{}\label{1st} Assume first that $X_\a = (S_{n-2}\times S_2)\cap X$, the stabilizer of a pair in $\{1,\ldots,n\}$. Then $\G$ is either a Johnson graph $J(n,2)$ (where two pairs are joined if and only if their intersection has size 1) or its complement. 
	Suppose first that $\G = J(n,2)$. Then $X_{\a\b} = S_{n-3}\cap X$, and $|X_\a:X_{\a\b}| = 2(n-2)$. Since $X_\a = G_\a X_{\a\b}$, this number must divide $|G_\a|$. On the other hand, from the list of possibilities for $2$-homogeneous subgroups $G$ of $S_n$ in Table \ref{tab:B2hom}, we see that if $G$ is almost simple, then one of the lines of the following table holds, 
	
	\[
	\begin{array}{ccl}
		\hline
		n & \soc(G) & |\soc(G)\cap X_\a| \\
		\hline
		q+1 & \PSL_2(q) & (q-1).(2,q) \\
		q^2+1 & ^2\!B_2(q) & 2(q-1) \\
		q^3+1 & \PSU_3(q) & 2(q^2-1)/(3,q+1) \\
		q^3+1 & ^2\!G_2(q) & 2(q-1) \\
		11 & \PSL_2(11) & 12 \\
		7 & \PSL_2(7) & 8 \\
		\hline
	\end{array}
	\]
	
	\no and when $G$ is affine, here are the possibilities for $|G_\a|$:
	
	\[
	\begin{array}{ccc}
		\hline
		n & G & |G_\a| \\
		\hline
		p^2 & \le \AGL_2(p) & \hbox{divides }p(p-1) \\
		q=p^a & \le \AGaL_1(q) & \hbox{divides }2a \\
		p^2 & \le p^2.(p-1).S_4 & \hbox{divides }\frac{48}{p+1} \\
		(p=5,7,11,23) && \\
		q^2 & \le q^2.(q-1).S_5 & \hbox{divides }\frac{240}{q+1} \\
		(q=11,\ldots,59) && \\
		\hline
	\end{array}
	\]
	
	\no From these tables, we see that the only case where $2(n-2)$ divides $|G_\a|$ is 
	\[
	n = q+1,\;G = \PGL_2(q),\;G_\a = D_{2(q-1)} 
	\hbox{ and }G_{\a\b} = 1.
	\]
	We claim that there is an embedding of $\G = J(n,2)$ in this case, as in  Table~\ref{t:anmaps}. Take $X = S_n$ and $G = \PGL_2(q) < X$. Let $g \in G$ be an involution that fixes a point in its action on $I = \{1,\ldots,n\}$; say $g = (1)\,(2\,3)\cdots $ in this action. Let $\a = \{1,2\}$ and $\b= \{1,3\}$, adjacent vertices of $\G$. Then $X_\a = S_{n-2}\times S_2$, $X_{\a\b} = S_{n-3}$, and $g$ normalises $X_{\a\b}$ but does not normalise $X_\a$. Hence $\la X_\a,g\ra = X$. Also $G_\a=D_{2(q-1)}$ and $G_{\a\b}=1$. Hence conditions (i)-(iii) of Proposition \ref{cri} all hold. Hence there is an embedding in this case, as claimed. Also there is no other subgroup $G$ such that $\PSL_2(q)\leq G\leq \PGaL_2(q)$ and $G_\a$ is cyclic or dihedral of order a multiple of $2(q-1)$ (the valency of $\G$). Thus the only possibility for $G$ is $\PGL(2,q)$. Moreover $G\leq\Alt(\O)$ by Lemma~\ref{l:pgl2} so $X$ may be $S_n$ or $A_n$.
	
	Now suppose $\G = J(n,2)^c$, the complement of the Johnson graph. In this case we have $|X_\a:X_{\a\b}| = \binom{n-2}{2}$, and as above this has to divide $|G_\a|$, which is not possible.
	
	\subsubsection{} 
	Next consider the case where $X_\a = S_{n-2}\cap X$, the stabilizer of an ordered pair in $\{1,\ldots,n\}$. In this case the only connected undirected $X$-arc-transitive graph on $\O$ has vertices $ij$ and $kl$ adjacent if and only if $\{i,j\} \cap \{k,l\} = \emptyset$. This graph has valency $(n-2)(n-3)=|X_\a:X_{\a\b}|$, which must divide $|G_\a|$, but we have seen in the previous case that this is impossible. 
	
	\subsubsection{} The remaining possibilities in this case (case (a) with $k=2$) are that $X = S_n$ and $X_\a$ is one of the following:
	\begin{itemize}
		\item [(i)] $A_{n-2}$,
		\item[(ii)] $A_{n-2}\times S_2$, or 
		\item[(iii)] $(S_{n-2}\times S_2)\cap A_{n-2}$. 
	\end{itemize}
	In all cases, there is a block system $\BB$ for $(X,\O)$ such that the action of $X$ on $\BB$ is the action of $S_n$ on pairs from 
	$\{1,\ldots,n\}$. This is a rank 3 action with subdegrees $1$, $2(n-2)$, $(n-2)(n-3)/2$. For each such pair $\{i,j\}$, label the corresponding block $B_{\{ij\}}$, and take $\a \in B_{\{12\}}$. If there is an edge in $\G$ from $\a$ to a vertex in a block $B_{\{ij\}}$ with $\{1,2\}\cap \{i,j\} = \emptyset$, then the valency of $\G$ is at least $(n-2)(n-3)/2$, which must therefore divide $|G_\a|$; we have seen that this is impossible. Hence there can only be edges from $\a$ to vertices in blocks of the form $B_{\{1i\}}$ or $B_{\{2i\}}$ with $i>2$. 
	
	In case (i), when $X_\a = A_{n-2}$, the block system $\BB$ can be refined to a block system $\BB'$ for which the action of $X$ is the action of $S_n$ on ordered pairs of distinct elements of $\{1,\ldots,n\}$; we label the blocks of $\BB'$ as $B_{ij}$. If $\G$ has an edge between vertices of $B_{12}$ and $B_{1i}$ (where $i>2$), then it is disconnected; and if there is an edge between vertices of $B_{12}$ and $B_{i1}$ (where $i>2$) then $\G$ is directed (and not arc-transitive). A similar argument rules out edges between vertices of $B_{12}$ and $B_{i2}$ or $B_{2i}$ (where $i\ne 1, 2$). Hence this case does not occur.
	
	Finally, consider cases (ii) and (iii). In these cases $X_\a$ acts transitively on the set of blocks of $\BB$ of the form $B_{\{1i\}}$ or $B_{\{2i\}}$, so $\G$ has valency at least $2(n-2)$. As we saw in Section \ref{1st}, this forces $n = q+1$, ${\rm soc}(G) = \PSL_2(q)$, and the stabiliser in $\soc(G)$ of $B_{\{12\}}$ (which is the stabiliser of an unordered pair from $\{1,\dots,n\}$) is $D_{t(q-1)}$ with $t=(2,q)$.
	Note that $\soc(G)$ has trivial centraliser in $S_{q+1}$, and hence $G\leq \Aut(\PSL_2(q))$; also $q>4$ (as otherwise $G$ contains $A_n$).  
	Now the largest dihedral subgroup of $\Aut(\PSL_2(q))$ containing $D_{t(q-1)}$ is $D_{2(q-1)}$ (and lies in $\PGL_2(q)$), and the valency of $\G$, which is at least $2(n-2) = 2(q-1)$, must divide $|G_\a|$. It follows that  $G_\a=D_{2(q-1)}$ and that $\G$ has valency $2(n-2)$, and moreover  $|G|=|V\G|\cdot|G_\a|=q(q+1)\cdot 2(q-1)=|\PGL_2(q)|\cdot2$. Since $D_{2(q-1)} <G$ we must have $\PGL_2(q)\leq G$, and hence $q=q_0^2$ and $G=\PGL_2(q).2$. Now the index $2$ subgroup $G_0:=\PGL_2(q)$ contains the stabiliser $G_\a$, and it follows that $\G$ is bipartite with two bipartite blocks $\D_1,\D_2$ fixed setwise by $G_0$. The setwise stabiliser of $\D_1$ and $\D_2$ in $X$ must be the unique index $2$ subgroup $X_0=A_{q+1}$, and hence $X_0$ contains $G_0$. Also $X_\a<X_0$ implies that $X_0$ is not vertex-transitive, and in particular $X_0$ does not contain $G$, so $G\cap X_0=G_0$. However, if $q=q_0^2$ is even, then the involutory field automorphism $\tau\in G$ has $q_0+1$ fixed points and $q_0(q_0-1)/2$ cycles of length $2$ (which is even since $q>4$) in its action on $\{1,\dots,q+1\}$, and hence $\tau\in X_0$. This implies that $G\leq X_0$ which is a contradiction. Hence $q$ is odd. However in this case we have $G_\a<G_0<A_{q+1}$, while the cyclic subgroup $C_{q-1}$ of $G_\a$ is generated by an odd permutation of $\{1,\dots,q+1\}$, which is a contradiction. 
	
	\subsection{Case(a) with $k=3$ }\label{ak3}
	Here $A_{n-3} \le X_\a \le (S_{n-3}\times S_3)\cap X$ with $G$ 3-homogeneous as in Table \ref{tab:Anfactns}. 
	In the light of our computations with Magma, we may assume that $n\geq 10$ here. The possibilities with $G_\a$ cyclic or dihedral are as follows:
	\[
	\begin{array}{ccc}
		\hline
		n & X_\a &  |G_\a| \hbox{ divides} \\
		\hline
		q+1 & {\rm soc}(X_\a)=\PSL_2(q)\ (q=p^f) & 6f \\
		12 & M_{11} & 12 \\
		8 & AGL_3(2),\,AGL_1(8),\,A\G L_1(8) & 8,1,3 \\
		32 & A\G L_1(32) & 1 \\
		\hline
	\end{array}
	\]
	
	By Hypothesis~\ref{h:map}(c), $X_{\a\b}$ does not contain $A_{n-3}$, and so the factorisation $X_\a = G_\a X_{\a\b}$ implies that $|G_\a|$ is divisible by $|A_{n-3}: X_{\a\b}\cap A_{n-3}|$, the degree of a nontrivial transitive action of $A_{n-3}$. This rules out all but the first line of the above table. For the first line, we have $n-3 = q-2$, so it follows that $5\leq q-2 \le 6f$ (where $q=p^f$), which implies that $q$ is one of the following:
	\[
	32,\,16,\,9,\,8,\,7.
	\]
	For $q=32$, we must have $|G_\a| = 6f=30$ and $G = \PGaL_2(32)$; but then $G_\a = S_3 \times 5$, which is not dihedral. And for $q=16$ or $9$, $A_{q-2}$ has no transitive action of degree dividing $6f$. 
	In the remaining cases $q=8$ or $q=7$, we have $n = 9$ or 8, cases already checked using Magma. 
	

	\subsection{Case(a) with $k=4$ or 5}\label{ak45}
	Here $A_{n-k} \le X_\a \le (S_{n-k}\times S_k)\cap X$ with $n\geq 2k$ and $G$ 3-homogeneous as in Table \ref{tab:Anfactns}. The possibilities with $G_\a$ cyclic or dihedral are as follows:
	\[
	\begin{array}{cccc}
		\hline
		G & k & n & G_\a \\
		\hline
		M_{11} & 4 & 11 & \le D_8 \hbox{ or }S_3 \\
		\PSL_2(8), \,\PSL_2(8).3 & 4 & 9 & \le 2^2 \hbox{ or }3 \\
		\PSL_2(32).5 & 4 & 33 & \le 2^2 \\
		M_{12} & 5 & 12 & \le D_6,\,D_8 \hbox{ or }D_{10} \\
		\hline
	\end{array}
	\]
	As in the previous case, $|G_\a|$ must be divisible by $|A_{n-k}: X_{\a\b}\cap A_{n-k}|$, the degree of a nontrivial transitive action of $A_{n-k}$. This is clearly impossible for all the entries in the above table. 
	
	\subsection{Case(b) with $k=1$}\label{bk1}
	Here $G = A_{n-1}$ or $S_{n-1}$, and $X_\a$ is transitive on $\{1,\ldots,n\}$, as in part (iii) of Theorem \ref{p:maps-an}. It is not feasible to classify all the possibilities in this case, but recall that in Construction~\ref{c:k1eg} we have given an infinite family of  examples via a novel group theoretic construction (see Theorem~\ref{k1eg}). To illustrate that there are many other examples, we list below the examples for $n=8,9$, which were computed using Magma.
	
	\begin{itemize}
		\item[{\rm (i)}] For $(X,X_\a,X_{\a\b},G, G_\a, G_{\a\b})=(S_8,[64],[8],S_7,D_8,1)$, there are several arc-transitive graphs $\G$ of valency $8$ with $630$ vertices such that Hypothesis~\ref{h:map} holds, and the groups $([64],[8])$ can have shape  $(D_8.D_8, C_2^3)$ or $(C_2\wr C_2^2, D_8)$.    
		\item[{\rm (ii)}] For $(X,X_\a,X_{\a\b},G, G_\a, G_{\a\b})=(S_8,2\times S_4,D_8,S_7,D_6,1)$, there are several arc-transitive graphs $\G$ of valency $6$ with $840$ vertices such that Hypothesis~\ref{h:map} holds.
		\item[{\rm (iii)}] For $(X,X_\a,X_{\a\b},G, G_\a, G_{\a\b})=(S_9,3^{1+2}.2,3^2,S_8,D_6,1)$, there is an arc-transitive graph $\G$ of valency $6$ with $56\cdot 5!$ vertices such that Hypothesis~\ref{h:map} holds.
	\end{itemize}

	\subsection{Case(b) with $k=2$}
	Here $A_{n-2} \le G \le (S_{n-2}\times S_2)\cap X$ with $X_\a$ 2-homogeneous as in Table \ref{tab:B2hom}. 
	
	\subsubsection{$X_\a$ almost simple } \label{xas}
	
	Assume that $X_\a$ is almost simple, with socle $T$. From Table \ref{tab:B2hom}, the possibilities are:
	\[
	\begin{array}{ccc}
		\hline
		n & T & G \cap T \le \\
		\hline
		q+1 & \PSL_2(q) & D_{t(q-1)}\,(t=(2,q)) \\
		q^2+1 & ^2\!B_2(q) & D_{2(q-1)} \\
		q^3+1 & \PSU_3(q) & D_{2(q^2-1)} \\
		q^3+1 & ^2\!G_2(q) & D_{2(q-1)} \\
		11 & \PSL_2(11) & D_{12} \\
		7 & \PSL_2(7) & D_8 \\
		\hline
	\end{array}
	\]
	
	\no For the entries in the first four rows of the table (assuming $n\ge 10$), \cite{LPS90} shows that there are no factorisations of $X_\a$ with a factor as in the third column, contradicting the fact that $X_\a = G_\a X_{\a\b}$. 
	
	\medskip\noindent
	\emph{Claim 1: For row $5$ of the table above, the graph $\G$ is as in line $5$ of Table~\ref{6711tbl}. }
	
	Here $T=\PSL_2(11)$ is self-normalising in $S_{11}$ and so $X_\a=T<A_{11}$.  Then since $|X_\a:X_{\a\b}|$ divides $|G_\a|$, which divides $12$, it follows that  $|X_\a:X_{\a\b}| = |G_\a|=12$, and  $G_\a=D_{12}$ by Theorem~\ref{t:Anfactns}. Also  $X_{\a\b}=11.5$,  $G_{\a\b}=1$ and $\G$ has valency $12$. Let $x \in X_{\a,\b}$ of order 11. Then $N_X(X_{\a\b})= N_X(\la x\ra)\leq 11.10$ must contain an involution, and hence we must have $X=S_{11}$ and then $N_{X}(X_{\a\b}) = 11.10$ does indeed contain an involution $g\in X\setminus A_{11}$. Further, $|G|=|X:X_\a|\cdot |G_\a| = \frac{11!}{12.11.5}\cdot 12 = 2\cdot 9!$ and hence $G=S_9\times S_2$.
	Note that since the involution $g\in S_{11}$ stabilises a pair in $\{1,\ldots,11\}$, it can be taken to lie in $G$, and since $g$ normalises $X_{\a\b}$ but does not normalise $X_\a$, $g$ interchanges $\a$ and $\b$. Thus $\G$ is the graph with edge-set $\{\a,\b\}^X$;  both $X$ and $G$ act arc-transitively on $\G$, and $\G$ has $|X:X_\a| = 12 \cdot 7!$ vertices. Finally,  $\la X_\a,g\ra = X$ as $S_{11}$ is the only insoluble transitive group of degree 11 containing an odd permutation, and hence $\G$ is connected. It is the graph in  line $4$ of Table~\ref{6711tbl}.
	
	We note that $\G$ is bipartite (since $X_\a<A_{11}<X$), and that $X$ acts $2$-arc-transitively on $\G$ (since $X_\a$ is $2$-transitive on the $12$ neighbours of $\a$).

	\medskip\noindent
	\emph{Claim 2: For row $6$ of the table above,  the graph $\G$ is as in line $4$ of Table~\ref{6711tbl}. }
	
	Here $T=\PSL_2(7)$ is self-normalising in $S_7$ and so $X_\a=T<A_7$. Then since $|X_\a:X_{\a\b}|$ divides $|G_\a|$, which divides $8$, it follows that  $|X_\a:X_{\a\b}| = |G_\a|=8$,  and $G_\a=D_{8}$ by Theorem~\ref{t:Anfactns}.  Also $X_{\a\b}=7.3$, $G_{\a\b}=1$ and $\G$ has valency $8$.  Let $x \in X_{\a,\b}$ of order 7. Then $N_X(X_{\a\b})= N_X(\la x\ra)\leq 7.6$ must contain an involution, and hence we must have $X=S_{7}$ and then $N_{X}(X_{\a\b}) = 7.6$ does contain an involution $g\in X\setminus A_{7}$. Further, $|G|=|X:X_\a|\cdot |G_\a| = \frac{7!}{8.7.3}\cdot 8 = 2\cdot 5!$ and hence $G=S_5\times S_2$. 
	
	Arguing exactly as in Claim 1, we see that there is an arc-transitive embedding in this case also, with $X_{\a\b} = 7.3$ for an edge $\{\a,\b\}$. The graph $\G$ has 30 vertices and valency 8, and in fact it is isomorphic to the antiflag graph of $\PG_3(2)$; this is a bipartite graph whose vertices are the points and hyperplanes of $\PG_3(2)$, joined if they are non-incident. Its automorphism group is $\PSL_4(2).2 = S_8$, so as $G = S_5 \times 2 \le S_5\times S_3 < S_8$, this is also an embedding that occurs under Case (b) with $k=3$ in Section \ref{bk3} below.
	
	\subsubsection{$X_\a$ affine } 
	
	\no Assume that $X_\a$ is affine. From Table \ref{tab:B2hom}, the possibilities are:
	
	\[
	\begin{array}{ccc}
		\hline
		n & X_\a & |G_\a|\ \hbox{divides }  \\
		\hline
		p^2 & \le \AGL_2(p) & p(p-1) \\
		q=p^a & \le \AGaL_1(q) & 2a \\
		p^2 & \le p^2.(p-1).S_4 & \frac{48}{p+1} \\
		(p=5,7,11,23) && \\
		q^2 & \le q^2.(q-1).S_5 & \frac{240}{q+1} \\
		(q=9,11,\ldots,59) && \\
		\hline
	\end{array}
	\]
	
	\no In all cases we see that the factorisation $X_\a = G_\a X_{\a\b}$ implies that $X_{\a\b}$ contains the full translation subgroup $T$ of order $n$. But then any involution $g$ as in Proposition \ref{cri} normalises $T$, and so $\la X_\a, g \ra \ne X$. So by Proposition \ref{cri}, there are no arc-transitive embeddings in this case.
	
	\subsection{Case (b) with $k=3$ } \label{bk3}
	
	Here $A_{n-3} \le G \le (S_{n-3}\times S_3)\cap X$ with $n\geq6$, and $X_\a$ is 3-homogeneous as in Table \ref{tab:Anfactns}. In similar fashion to `case (a) with $k=3$' in Subsection~\ref{ak3}, the possibilities are:
	\[
	\begin{array}{ccc}
		\hline
		n & X_\a &  |G_\a| \hbox{ divides} \\
		\hline
		q+1 & {\rm soc}(X_\a)=\PSL_2(q)\ (q=p^f) & 6f \\
		12 & M_{11} & 12 \\
		8 & \AGL_3(2),\,\AGL_1(8),\,\AGaL_1(8) & 8,1,3 \\
		32 & \AGaL_1(32) & 1 \\
		\hline
	\end{array}
	\]
	In the first row, $X_{\a\b}$ does not contain $\PSL_2(q)$ by Hypothesis~\ref{h:map}(c), and so  $|X_\a:X_{\a\b}|$ is divisible by the index of a proper subgroup of $\PSL_2(q)$. Since $|X_\a:X_{\a\b}|$ divides $|G_\a|$, which divides $6f$,  we must have $q = 5$ or 8. If $q=5$ then $X_\a = \PGL_2(5) < X = S_6$ and we have $n=6$, $\G=K_6$ as in line 1 of Table~\ref{6711tbl} of Theorem \ref{p:maps-an}. If $q=8$ then $G_\a$ must contain a subgroup $3^2$, so is not cyclic or dihedral.
	
	Next observe that $M_{11}$ has no factorisation with a factor of order 12, so $X_\a \ne M_{11}$. Similarly $X_\a$ is not $\AGL_1(8)$, $\AGaL_1(8)$ or $\AGaL_1(32)$.
	
	Finally, consider the case where $X_\a = \AGL_3(2)$. Here there is an example of an arc-transitive embedding. In fact this occurs with the antiflag graph of $\PG_3(2)$ in line $4$ of Table~\ref{6711tbl}, which we saw in Section \ref{xas}: this graph has automorphism group $X = \PSL_4(2).2 = S_8$ and vertex-stabiliser a parabolic subgroup $X_\a=\AGL_3(2)$. The group $G$ for this embedding is $S_5\times 2$. 
	
	\subsection{Case (b) with $k=4$ or $5$ } \label{bk45}
	
	Here $A_{n-k} \le G \le (S_{n-k}\times S_k)\cap X$, and $X_\a$ is 4-homogeneous as in Table \ref{tab:Anfactns}. However in every case the group $X_\a$ has no factorisation with one of the factors being $G_\a$.
	
	\subsection{Case (c) with $n=6, 8$ or $10$ } \label{cn6810}
	
	We performed  computations with Magma to find all the possibilities for $X, G, \G, g$ satisfying Hypothesis~\ref{h:map} with $\soc(X)=A_n$, where $n\in\{6,8,10\}$. Many of the possibilities occur in families already analysed and identified in Tables~\ref{t:anmaps} and~\ref{6711tbl}, and several that we recorded in \S~\ref{bk1} with $G=A_{n-1}$ or $S_{n-1}$ -- the one case where we do not have a complete classification. There was exactly one new example that did not fit any of these cases. and we present it below.
	
	\subsubsection{$n=6$}\label{cn6} 
	Here is the new set of subgroups:
	\[
	(X,X_\a,G, G_\a, X_{\a\b})=(S_6.2,\, S_5,\, (S_3 \wr S_2).2,\, D_{12},\, F_{20}).
	\] 
	Since $S_5<S_6<S_6.2$ it follows that the graph $\G$ is bipartite. Let $\D_1, \D_2$ be the parts of the bipartition, where $\a\in\D_1, \b\in\D_2$, and let $X^+=S_6$ fixing each $\D_i$ setwise. An involution in $X\setminus S_6$ interchanges $\D_1$ and $\D_2$, and acts as a graph automorphism of $X^+$. Thus the actions of $X^+$ on $\D_1$ and $\D_2$ are not equivalent, so $X_\a$ acts transitively on $\D_2$ and $\G$ is the complete bipartite graph $K_{6,6}$. Since $S:=O_5(X_{\a\b})$ is a Sylow $5$-subgroup of $X$, it follows that $N_X(S)S_6/S_6\cong C_2$, 
	and by \cite{magma}, $N_X(S)\setminus S_6$ contains an involution $g$ such that $g\in G$ and $g$ interchanges $\a$\and $\b$. In addition $\langle X_\a,g\rangle = X$ since $\G$ is connected and so all the conditions of Proposition~\ref{cri} hold and we have an arc-transitive embedding, as in Line 3 of Table~\ref{6711tbl}.
	
	\vspace{4mm}
	This completes the proof of Theorem \ref{p:maps-an}.

	\section{The map theorem for groups of Lie type}\label{s:liemaps}
	
	In this section we prove Theorem \ref{t:classical-maps}. The procedure is to examine each of the cyclic/dihedral factorisations in Theorem~\ref{classgen} to determine which of them correspond to simple connected graphs admitting arc-transitive embeddings. For convenience, we restate the theorem here:
	
	\begin{thm}\label{p:maps-families}
		Suppose that Hypothesis~$\ref{h:map}$ holds for $X, G, \G, g$ with $X$ is  an almost simple group of Lie type (such that $\soc(X)$ is not an alternating group). Then one of the lines of Table~$\ref{class-maps}$ holds.
	\end{thm}

	
	\no {\bf Proof of Theorem \ref{p:maps-families}}
	
	\vspace{2mm}
	Let $X, G, g,$ $\G = (\O,E)$ be as in Hypothesis~\ref{h:map}  where $X_0$ is a simple group of Lie type, not isomorphic to an alternating group. By Theorem \ref{classgen}, the possibilities for $X,G, X_\a$ are as in Tables  \ref{families} (the families) and \ref{excep-psl} - \ref{excep-sp} (the exceptional factorisations). Note that $\{X_\a,G\} = \{A,B\}$ for $A,B$ in the tables.
	

	\subsection{Proof for the tables of exceptional factorisations.}\label{s:maps-exclass}
	
	Here we assume that $X$ and the pair $\{X_\a,G\}$ occurs in one of the lines of Tables~\ref{excep-psl}, \ref{excep-psu}, \ref{excep-orthog} or \ref{excep-sp}. For each possibility $(X, \{X_\a,G\})$ in these tables a computation using Magma~\cite{magma} was performed to identify cases where an involution $g$ exists such that Hypothesis~\ref{h:map} holds. This computation yielded exactly two possibilities for linear groups, and two possibilities for unitary groups, and we consider these cases below separately. 
	
	\medskip
	\noindent
	\emph{Line 1 of Table~\ref{excep-psl} with $X=\PSL_2(7).2$,  $X_\a= S_4$,  $G=7.6$, $G_\a=3$, and $G_{\a\b}=1$.}\ \\ 
	Here $X_0=\PSL_2(7)$ and $X_\a<X_0<X=X_0.2$, and it follows that $\G$ is bipartite.  Let $\D_1,\D_2$ denote the two parts of the bipartition of $V\G$ and assume that $\a\in\D_1$ and $\b\in\D_2$.  Since $X=X_\a G=X_0G$, the involution $g\in G\setminus X_0$ and so $g$  induces an outer automorphism of $X_0$ and interchanges $\D_1$ and $\D_2$, Thus the $X_0$-actions on $\D_1$ and $\D_2$ are not equivalent, and hence $X_\a$ has orbits of lengths $3, 4$ in $\D_2$. As $\G$ has valency $|G_\a:G_{\a\b}|=3$, $\b$ lies in the $X_\a$-orbit in $\D_2$ of length $3$. The most convenient way to identify $\G$ is to observe that $X=\PSL_3(2).2$, so that we may identify $\D_1$ and $\D_2$ with the sets of $7$ points and lines of the projective plane $\PG_2(2)$, respectively, we have then shown that a projective point $\a$ is adjacent to a projective line $\b$ if and only if they are incident in $\PG_2(2)$. Thus $\G$ is the incidence graph of $\PG_2(2)$, as in line $1$ of Table~\ref{class-maps}.
	
	\medskip
	\noindent
	\emph{Line 2 of Table~\ref{excep-psl} with $X=\PSL_2(11).2$,  $X_\a= A_5$,  $G=11.10$, $G_\a=5$, and $G_{\a\b}=1$.}\ \\ 
	As in the previous case, $X_0=\PSL_2(11)$ and $X_\a<X_0<X=X_0.2$, so $\G$ is bipartite, $g\in G\setminus X_0$, $g$ induces an outer automorphism of $X_0$, and   the $X_0$-actions on the parts $\D_1$ and $\D_2$ of the bipartition are not equivalent, and hence $X_\a$ (where $\a\in\D_1$) has orbits of lengths $5, 6$ in $\D_2$, and as $\G$ has valency $|G_\a:G_{\a\b}|=5$, $\b$ lies in the $X_\a$-orbit in $\D_2$ of length $5$.  In this case we identify $\D_1$ and $\D_2$ with the sets of $11$ points and blocks of the  $2-(11,5,2)$ biplane, respectively, we have then shown that a  point $\a$ is adjacent to a block $\b$ if and only if they are form a flag of the biplane. Thus $\G$ is the incidence graph of the biplane, as in line $3$ of Table~\ref{class-maps}.
	
	\medskip
	\noindent
	\emph{Line 3 of Table~\ref{excep-psu} or Line 1 of Table~\ref{excep-sp} with $X=\PSU_4(2).2$,  $X_\a= 3^3.A_4.a$,  $G=2^4.S_5$, $G_\a=D_{12a}$, and $G_{\a\b}=1$, where $a\leq 2$.}\ \\ 
	Here  $X_0=\PSU_4(2)\cong \PSp_4(3)\cong \O_5(3)$ and $X=X_0.2=\SO_5(3)$,  and the two possibilities for $X_\a$ are $3^3.A_4$ and $3^3.S_4$. Suppose first that $X_\a = 3^3.S_4$ (so $a=2$). Here a Magma computation shows that the group $X_0$ is transitive on $\O$, and so (see \cite[p. 26]{Atlas}) $X_0\cap X_\a$ is the unique index $2$ subgroup $3^3.A_4$ of the stabiliser in $X_0$ of a singular $1$-space in $V=\F_3^5$. Thus $X_0\cap X_\a$ is the stabiliser of a non-zero singular vector of $V=\F_3^5$, and so $\O$ can be identified with the set of non-zero singular vectors of $V$. The subgroup $X_\a$ has exactly one orbit of length $24=|X_\a:X_{\a\b}|$ in $\O$, namely the set of singular vectors $\b$ orthogonal to $\a$, so the graph $\G$ is as in line $4$ of Table~\ref{class-maps}. 
	
	Now consider the case where $X_\a = 3^3.A_4$ (that is, $a=1$), of index $160$ n $X$. Here the Magma computations show that the group $X_0$ has two orbits in $\O$, say $\O^+$ and $\O^-$,  and that the $X_0$-action on each $\O^\pm$ is equivalent to its action on the vertex set of the $80$-vertex graph of the previous paragraph. Thus $\G$ is bipartite with $\O^+, \O^-$ the parts of the bipartition, and we identify each $\O^\pm$ with the set of non-zero singular vectors in $V=\mathbb{F}_3^5$. 
	Further, Magma computations also show that $X_\a=(X_0)_\a$ has two orbits of length $12$ in the set of $24$ neighbours of $\a$ in the $80$-vertex graph, and that these suborbits are paired with each other. Hence the two corresponding $X_0$-orbitals $\Delta$ and $\Delta'$ in $\O^\pm\times\O^\pm$ are paired with each other. In this case the graph $\G$ has valency $12=|G_\a:G_{\a\b}|$, and for $\a=v^+\in\O^+$, the set of edges adjacent to $v^+$ must be one of the two orbits $\Delta(v)^-$ or $\Delta'(v)^-$ of $(X_0)_\a$ of size $12$ in $\O^-$ (as the graph is connected). By Lemma~\ref{l:u42}(b), the graphs corresponding to each of these choices are isomorphic to the graph $\widehat{\G}$ of Example~\ref{ex:u42}(c). Thus Line 5 of Table~\ref{class-maps} holds.

	\subsection{Proof for the infinite families of classical group factorisations.}
	
	We now assume that $(X,\{X_\a,G\})$ is as in one of the lines of Table~\ref{families}. 
	First we perform Magma computations as described in Subsection \ref{maggie} to identify any possibilities for $(X, X_\a,G,g)$ satisfying Hypothesis~\ref{h:map} for the following small groups in these families: 
	\begin{equation}\label{listsm}
		\begin{array}{l}
			\PSL_2(q)\ (q\leq 11), \PSL_3(4), \PSL_4(3), \PSU_3(3),  \ \PSU_4(q)\ (q\in\{2,3\}), \\
			\Sp_4(q)\ (q\in\{3,4,8\}),\  \Sp_6(2), \Sp_8(2), \Omega_8^+(2).
		\end{array}
	\end{equation}
	No examples were found apart from those analysed in Subsection~\ref{s:maps-exclass},  and so  we assume that $X_0$ is not one of these groups, and $X_0$ is not an alternating group.    
	We examine each of the lines of Table~\ref{families} separately below.

	\subsubsection{Lines in Table~\ref{families}\label{s:psl-1} with  $X_0=\PSL_n(q)$, $n\geq2$. } Let $V=\F_q^n$ and $q=p^f$ for a prime $p$ and $f\geq1$. From the discussion above, if $n=2$ then $q>11$, and $(n,q)\ne (3,4)$ or $(4,3)$. 
	
	\medskip
	\noindent
	\emph{Line 1a: $X_\a \triangleright q^{n-1}.\SL_{n-1}(q)$, $G\cap X_0\leq \frac{q^n-1}{(q-1)d}.n$, and $G_\a\cap X_0 = r$ where $r\mid n$ and $d=(n,q-1)$.}\ \\
	Here, applying a graph automorphism of $X_0$ if necessary, we may take $X_\a \le P_1 = X_{\la v\ra}$ for some $v\in V$. Suppose first that $X\leq \PGaL_n(q)$. Now $X_{\b}\leq P_1^g$, and if $\la v\ra^g=\la v\ra$ then $\la X_\a,g\ra\leq P_1$, which is a contradiction. Hence $\la v\ra^g=\la w\ra\ne \la v\ra$, and as $X_0$ is $2$-transitive on $1$-spaces, this implies that $(q^n-q)/(q-1)$ divides $|X_\a:X_{\a\b}|=|G_\a:G_{\a\b}|$. However $|G_\a:G_{\a\b}|$ divides $|G_\a|$ which divides $nf$. This implies that $(n,q)=(2,4)$, but in this case $X_0\cong A_5$ which is excluded. 
	
	Now we deal with the case where $X$ contains a graph automorphism so $X=Y.2$ where $Y:=X\cap \PGaL_n(q)$ and $n\geq3$. Since $X_\a\leq P_1<Y$ it follows that $\G$ is bipartite, and as $X=X_\a G=YG$, the subgroup $G$ must also contain a graph automorphism, while $G_\a\leq X_\a<Y$ so $|G_\a|$ divides $nf$. Let $\D_1,\D_2$ denote the two parts of the bipartition of $V\G$, with $\a\in\D_1$ and $\b\in\D_2$. Now $X$ preserves a partition $\BB=\BB_1\cup\BB_2$ of $V\G$, where $\BB_i$ is a partition of $\D_i$ ($i=1,2$), such that the $Y$-actions on $\BB_1, \BB_2$ are equivalent to its actions on $1$-spaces and hyperplanes of $V$, respectively. Since $X_\a\geq P_1'$ it follows that $X_\a$ has exactly two orbits in $\BB_2$, with lengths $m_1:=q^{n-1}$ and $m_2:=(q^{n-1}-1)/(q-1)$. The point $\b$ lies in a block of one of these orbits, and hence $|X_\a:X_{\a\b}|$ is divisible by $m_i$ for some $i$. This implies that $m_i$ divides $|G_\a|$, so  $m_i$ divides $nf$. The only possibilities are $(n,q,m_i)=(3,2,3)$ and $(3,8,9)$. In the former case, $X_0=\PSL_3(2)\cong \PSL_2(7)$, which is excluded. So $(n,q,m_i)=(3,8,9)$, and we identify  the graph $\G$ as follows.  Since $m_i=9=q+1=nf$, it follows that $|G_\a|=9$ and $|V\G|=2(q^n-1)/(q-1)=2\cdot 73$,  so the partition $\BB$ is trivial. Thus we may identify $\D_1, \D_2$ with the points and lines of the projective plane ${\rm PG}_2(8)$ such that the point $\a$ is joined to the set of $9$ lines incident with it. Thus $\G$ is the incidence graph of ${\rm PG}_2(q)$ as  in Line 2  of Table~\ref{class-maps}. 
	The existence of an involution $g$ satisfying Hypothesis \ref{h:map} is computed using Magma, completing the proof for this case.
	
	
	\medskip
	\noindent
	\emph{Line 1b: $X_\a\cap X_0\leq \frac{q^n-1}{(q-1)d}.n$, $G \triangleright q^{n-1}.\SL_{n-1}(q)$, and $G_\a\cap X_0 = r$ where $r\mid n$ and $d=(n,q-1)$.}\ \\
	Let $Y:=X\cap \PGaL_n(q)$ so that $a:=|X:Y|\leq 2$, and if $a=2$ then $n\geq 3$. Note that $G\leq Y$, so $X=X_\a G=X_\a Y$ and hence $a=|X_\a:X_\a\cap Y|$ and $X_\a\leq \frac{q^n-1}{(q-1)d}.nf.a$. 
	Now $|X:P_1|=(q^n-1)/(q-1)$ divides $|X:G|=|X_\a:G_\a|$, and hence $(q^n-1)/(q-1)$ divides $|X_\a|$.
	If both $(nf,p)\ne (6,2)$ and $(n,f,p)\ne (2,1,2^c-1)$,  then $p^{nf}-1$ has a primitive prime divisor $s$ by \cite{ZZ}, and we note that $s\equiv 1\pmod{nf}$ so in particular $s$ does not divide $nfa$. Thus $s$ divides $|X_\a\cap X_0|$, so $X_\a$ has an element $x$ of order $s$ such that $\langle x\rangle$ is normal in $X_\a$ and  irreducible on $V$, and hence $X_\a\leq N_X(\la x\ra )\leq  \frac{q^n-1}{(q-1)d}.nf.a$ (see \cite[Satz II.7.3]{H67}).
	It follows from Hypothesis~\ref{h:map}(c) that $x\not\in X_\b$, and this implies that $s$ divides 
	$|X_\a:X_{\a\b}|=|G_\a:G_{\a\b}|$, which in turn divides $|G_\a|$. Thus $nf<s\leq |G_\a|\leq nf$, which is a contradiction. 
	
	If $(nf,p)= (6,2)$ let $s=21$; and if $(n,f,p)= (2,1,2^c-1)$ let $s=2^{c-2}$, and note that in the latter case $c\geq 5$ and $a=1$ (since we assume that $q>11$ when $n=2$). Then again $X_\a\cap Y$ has an element $x$ of order $s$ which is irreducible on $V$ and  $X_\a\leq N_X(\la x\ra )\leq  \frac{q^n-1}{(q-1)d}.nf.a$. If $s=|x|=21$ and $nf=6$, then $|G_\a|$ is not divisible by $7$, and hence $|X_\a:X_{\a\b}|$ is not divisible by $7$. In this case $\langle x^3\rangle$ is the unique normal subgroup of $X_{\a\b}$ of order $7$, contradicting Hypothesis~\ref{h:map}(c). Hence $(n,f,p)= (2,1,2^c-1)$, $s=2^{c-2}$, and $|G_\a|\leq nf=2$. Thus  $|G_\a|$ is not divisible by $4$, and hence $|X_\a:X_{\a\b}|$ is not divisible by $4$, so $\langle x^2\rangle\unlhd X_{\a\b}$. In this case the group $\frac{q^n-1}{(q-1)d}.nf$ is dihedral of order $q+1=2^c$ and  since $a\geq5$, $\langle x^2\rangle$ is the unique cyclic normal subgroup of $X_{\a\b}$ of order $2^{c-3}$, contradicting Hypothesis~\ref{h:map}(c).

	\medskip
	\noindent
	\emph{Line 2a: $(n,q)=(4,p)$,  $X_\a\cap X_0=\PSp_4(p).a$, $G\cap X_0= p^3.\frac{p^3-1}{d}.e$, and $G_\a\cap X_0 = p\times e$, where  $p$ is odd, $e\mid 3$, $d=(4,p-1)$, $ad=4$,  and $X\geq X_0.2$ if $p\equiv 1\pmod{4}$.}\ \\
	Here $p>3$ (as $\PSL_4(3)$ is excluded by (\ref{listsm}), and $X_{\a\b}$ does not contain $\PSp_4(p)$ by Hypothesis~\ref{h:map}(c). Hence, by \cite[Table 5.2.A]{KL},  $|X_\a:X_{\a\b}|\geq (p^4-1)/(p-1)$ since $p>3$; and $|X_\a:X_{\a\b}|$ divides $|G_\a|\leq (3p)\cdot2(p-1)$, which is a contradiction. 
	
	\medskip
	\noindent
	\emph{Line 2b: $(n,q)=(4,p)$,  $X_\a\cap X_0= p^3.\frac{p^3-1}{d}.e$, $G\cap X_0=\PSp_4(p).a$,  and $G_\a\cap X_0 = p\times e$, where  $p$ is odd, $e\mid 3$, $d=(4,p-1)$, $ad=4$,  and $X\geq X_0.2$ if $p\equiv 1\pmod{4}$.}\ \\
	Again $p>3$ by (\ref{listsm}),  and we have $X_\a=QR$ where $Q=[p^3]=O_p(X_\a)$. Now $|X_\a:X_{\a\b}|=|G_\a:G_{\a\b}|$ divides $|G_\a|$, which in turn divides $2(p-1)pe$, and it follows that $|X_{\a\b}|$ is divisible by $p^2(p^3-1)/(p-1)$. This implies both that $X_{\a\b}\cap Q\ne 1$ and that  $X_{\a\b}$ contains a cyclic subgroup of order $(p^3-1)/(p-1)$ which acts irreducibly on $Q$. Hence $Q\leq X_{\a\b}$, and this contradicts Hypothesis~\ref{h:map}(c).
	
	
	
	\subsubsection{Lines in Table~\ref{families} with $X_0=\PSp_{2m}(q)$} Let $V=\F_q^{2m}$ be the natural module for $X$, where $m\geq2$ and  $q=p^f$ for a prime $p$ and $f\geq1$.  By (\ref{listsm}), we are assuming: if $m=2$ then $q\geq 5$ and $q\ne 8$, and also $(2m,q)\ne (6,2), (8,2)$. 
	
	\medskip
	\noindent
	\emph{Lines 1--3\,a: $q$ even, \  $\O_{2m}^\epsilon(q)\unlhd X_\a$, $\Sp_2(q^m)\unlhd G\leq \Sp_2(q^m).(2mf)$,   and $C_{q^m-\epsilon}\leq G_\a\cap X_0\leq D_{2(q^m-\epsilon)}$, where $\epsilon=\pm$.}\ \\
	Note that the largest dihedral subgroup of $\Aut(\SL_2(q^m))$ containing $C_{q^m-\epsilon}$ is $D_{2(q^m-\epsilon)}$, so $G_\a\leq  D_{2(q^m-\epsilon)}$. Thus $|X_\a:X_{\a\b}|$ divides $|G_\a|$ which divides  $2(q^m-\epsilon)$.
	On the other hand,  by \cite[Theorem 1]{I90}, $|X_\a:X_{\a\b}|$ is divisible by either $q^{m-1}(q^m-\epsilon)$ or $(q^m-\epsilon)(q^{m-1}+\epsilon)$, so one of these quantities must divide $2(q^m-\epsilon)$. This implies that $m=2$ and $q\leq 3$, but these cases are excluded.
	
	\medskip
	\noindent
	\emph{Lines 1--3\,b: $q$ even, \ $\Sp_2(q^m)\unlhd X_\a\leq \Sp_2(q^m).(2mf)$,   $\O_{2m}^\epsilon(q)\unlhd G$,  and $C_{q^m-\epsilon}\leq G_\a\cap X_0\leq D_{2(q^m-\epsilon)}$, where $\epsilon=\pm$.}\ \\
	As noted in the previous case,  $G_\a\leq  D_{2(q^m-\epsilon)}$ and $|X_\a:X_{\a\b}|$ divides  $2(q^m-\epsilon)$. Hence $|X_{\a\b}|$ is divisible by $\frac{1}{2}q^m(q^m+\epsilon)$. If $X_{\a\b}\leq N_{X_\a}(q^m+\epsilon)\leq (q^m+\epsilon).4f$ then $q^m$ divides $8f$, but this is impossible for the parameter values being considered. Thus this inclusion does not hold, and so $\epsilon=-$ and $X_{\a\b}$ is contained in a Borel subgroup $B$ of $X_\a$. 
	
	Now $X_\a=\Sp_{2}(q^m).a$ where $a\mid 2mf$, and $\frac{1}{2}q^m(q^m-1)a = |X_\a:D_{2(q^m+1)}|$ divides $|X_\a:G_\a|$   which divides $|X_{\a\b}|$. It follows that $|B:X_{\a\b}|\leq 2$ and $|D_{2(q^m+1)}:G_\a|\leq 2$. Thus $X_{\a\b}=RM.b$ where $R=[q^m]$, $M=q^m-1$, and $b\leq2$.
	
	\smallskip\noindent
	\emph{Next we analyse $N_{X_0}(X_{\a\b})$.}\\ Now $N:=N_{X_0}(X_{\a\b})\leq  N_X(R)\leq P$ for some parabolic subgroup $P$. Since $M=q^m-1$ acts on the module $V$ as $W\oplus W'$ (a direct sum of two totally isotropic $m$-spaces which are not isomorphic as $M$-modules), the subgroup $P=P_m$ is the stabiliser $X_W$ or $X_{W'}$. Hence $P=QL$ where $Q=[q^{\frac{1}{2}m(m+1)}]$, $L=\GL_m(q)$, and $R\leq Q$; and we can take $M\leq L$. Now $Q$ is elementary abelian and, as an $\mathbb{F}_qL$-module, we have $Q\cong S^2 V_m(q)$.
	
	As in the proof of Proposition \ref{caseI3}, if $m$ is odd then the $\F_q$-composition factors of $Q\downarrow M$ all have dimension $m$, while if $m$ is even then there is one composition factor of dimension $m/2$ and the rest have dimension $m$.
	Hence $R$ is one of the composition factors of dimension $m$ and we have $Q\downarrow M=R\oplus Q_1$. 
	
	\smallskip\noindent
	\emph{Claim 1. $C_Q(M)=1$.}\\ 
	As we have seen above, the eigenvalues of $x$ on $S^2 \overline{V}_m$ are $\l^{2q^i}$ and $\l^{q^i+q^j}$ for $1\le i<j\le m$. None of these is equal to 1, and Claim 1 follows.
	
	
	\smallskip\noindent
	\emph{Claim 2.  $N_Q(RM)=R$.}\\ 
	Let $rq_1\in N_Q(RM)$, where $r\in R$ and $q_1\in Q_1$. Then $q_1\in N_Q(RM)$, so as $R$ and $Q_1$ are both $M$-invariant, $[q_1,M]\leq RM\cap Q_1=1$. Hence $q_1\in C_Q(M)$, and so $q_1=1$ by Claim 1. Thus $N_Q(RM)=R$, proving Claim 2.
	
	\smallskip\noindent
	\emph{Claim 3. $N_{X_0}(X_{\a\b})= R N_L(M)=RM.m\leq N_{X_0}(X_\a)$.}\\ 
	Note that $N_{X_0}(X_\a)=\Sp_2(q^m).m$ and that $N_{\Sp_2(q^m).m}(RM)=RM.m$, so in particular $M.m\leq N_{X_0}(RM)$ and $RM.m\leq N_X(X_\a)$. As noted earlier in the proof, $N_{X_0}(X_{\a\b})=N_{X_0}(RM)\leq P=QL$, so $N_{X_0}(RM)=N_P(RM)$.
	Now $N_P(RM)\leq  N_P(QM)$ (since $QM=Q(RM)$ and $Q\unlhd P$), and $N_P(QM)/Q\cong M.m$ (the normaliser in $L=\GL_m(q)$ of the Singer cycle $M$). 
	Similarly, $N_L(M)=M.m$ and we just saw that this is contained in $N_X(RM)$.  It follows that $N_P(QM) = Q N_L(M)$, and then that $N_P(RM)=N_Q(RM)N_L(M)$. Using Claim 2, we obtain $N_P(RM)=RM.m$, which is contained in $N_{X_0}(X_\a)$. This proves Claim 3.
	
	Now the involution $g$ of Hypothesis~\ref{h:map} lies in $N_G(X_{\a\b})$ and $g$ does not normalise $X_\a$.  However, by Claim 3, $N_{X_0}(X_{\a\b})\leq N_X(X_\a)$, and hence $g\not\in X_0$. Thus $g$ induces an involutory field automorphism of $X$, and so it must also induce an involutory field automorphism of $G$. However $\O_{2m}^-(q)$ has no involutory field automorphisms (see for example \cite[Prop. 2.8.2]{KL}), a final contradiction, completing this case.

	
	

	
	\medskip
	\noindent
	\emph{Lines 4--5\,a: $q$ even, \  $X_\a\cap X_0=\O_{2m}^-(q).a$, $G\cap X_0=[q^m].(q^m-1).m'$,   and $G_\a\cap X_0= a m'$, where $a\leq 2$ and $m'\mid m$.}\ \\
	By Hypothesis~\ref{h:map}(c), $X_{\a\b}$ does not contain $\O_{2m}^-(q)$. Also $|X_\a:X_{\a\b}|$ divides
	$|G_\a|\leq 4mf$, whereas by \cite[Table 5.2.A]{KL}, $|X_\a:X_{\a\b}|$ is much larger than $4mf$.

	\medskip
	\noindent
	\emph{Lines 4--5\,b: $q$ even, \  $X_\a\cap X_0=[q^m].(q^m-1).m'$, $G\cap X_0=\O_{2m}^-(q).a$,   and $G_\a\cap X_0= a m'$, where $a\leq 2$ and $m'\mid m$.}\ \\
	Let $X_\a = RM.m'$ where $R=[q^m]$ and $M=q^m-1$. Now $|X_{\a\b}|$ is divisible by $|X_\a:G_\a|$, which is divisible by $q^m(q^m-1)/2fa$. If $R\cap X_{\a\b}=1$ then this implies that  $q^m$ divides $2fm$, and hence that $m=q=2$, which is excluded. Thus $R\cap X_{\a\b}\ne1$, and as $M$ acts irreducibly on $R$, we must have $R\leq X_{\a\b}$,  contradicting Hypothesis~\ref{h:map}(c). 
	
	\medskip
	\noindent
	\emph{Line 6\,a: $q$ even, \  $X_\a\cap X_0=\O_{2m}^-(q)$, $G\cap X_0=\Sp_2(q^{m/2})\wr S_2$,   
		and $G_\a\cap X_0= D_{2(q^m-1)}$, where $m\equiv 2\pmod{4}$.}\ \\
	As usual, $X_{\a\b}$ does not contain $\O_{2m}^-(q)$ by Hypothesis~\ref{h:map}(c). If $m=2$ (so $q\geq 8$ since $q$ is even), then $X_\a=\PSL_2(q^2).a$    with $a\leq 2$ and $G_\a=D_{2(q^2-1)}$ since there is no larger dihedral subgroup in $X_\a$. Thus
	$|X_{\a\b}|$ is divisible by $|X_\a:G_\a|=q^2(q^2+1)a/2$, but the only such subgroups of $X_\a$ contain $\PSL_2(q^2)$, while $X_{\a\b}$ does not contain $\PSL_2(q^2)$. Thus $m\ne 2$ and so $m\geq 6$. Then  by \cite[Table 5.2.A]{KL}, $|X_\a:X_{\a\b}|\geq (q^m+1)(q^{m-1}-1)/(q-1)$ whereas $|X_\a:X_{\a\b}|$ divides $|G_\a|$, which divides $2(q^m-1)$, giving a contradiction.
	
	\medskip
	\noindent
	\emph{Line 6\,b: $q$ even, \  $X_\a\cap X_0=\Sp_2(q^{m/2})\wr S_2$, $G\cap X_0=\O_{2m}^-(q)$,    
		and $G_\a\cap X_0= D_{2(q^m-1)}$, where $m\equiv 2\pmod{4}$.}\ \\
	Now $X_\a=(\Sp_2(q^{m/2})\wr S_2).a$ for some $a$ dividing $f$, and  the base group of $X_\a\cap X_0$ is $T_1\times T_2$ with each $T_i=\Sp_2(q^{m/2})$ (a nonabelian simple group).
	If $X_{\a\b}$ contains $T_i$, then $G_{\a\b}=G_\a\cap X_{\a\b}\geq D_{2(q^m-1)}\cap T_i\geq (q^{m/2}\pm 1)$, contradicting Hypothesis~\ref{h:map}(a). Thus $X_{\a\b}\cap T_i$ is a proper subgroup of $T_i$ for each $i$. On the other hand, 
	the index $|X_\a:G_\a|=q^m(q^m-1).a$, and this divides $|X_{\a\b}|$. Hence, for each $i$, $|X_{\a\b}\cap T_i|$ is divisible by $q^{m/2}/2$; and for some $i$, say $i=1$,  a primitive prime divisor $r$ of $2^{fm}-1$ (which exists since $(2m,q)\ne (4,8), (6,2)$) divides $|X_{\a\b}\cap T_1|$. The latter condition implies that $X_{\a\b}\cap T_1\leq N_{T_1}(C_r)\leq (q^{m/2}+1).2f$, and hence $q^{m/2}/2$ divides $2f$. This implies that $m=2$ and $q=16$. However, $\Sp_2(16)$ has no proper subgroup of order divisible by $17.8$,  and we have a contradiction.
	
	\medskip
	\noindent
	\emph{Line 7\,a: $q$ even, \  $X_\a\cap X_0=\O_{2m}^-(q)$, $G\cap X_0=\Sp_4(q^{m/4}).2$,   
		and $G_\a\cap X_0= D_{2(q^{m/2}-1)}$, where $m\equiv 4\pmod{8}$ and $f$ is odd.}\ \\
	Arguing as in the case of Line 6.a, $X_{\a\b}$ does not contain $\O_{2m}^-(q)$, and $|X_\a:X_{\a\b}|$ divides $|G_\a|$, which divides $2(q^{m/2}-1).f$; whereas (since $m\geq4$) by   \cite[Table 5.2.A]{KL}, $|X_\a:X_{\a\b}|\geq (q^m+1)(q^{m-1}-1)/(q-1)$, which is a contradiction.
	
	\medskip
	\noindent
	\emph{Line 7\,b: $q$ even, \  $X_\a\cap X_0=\Sp_4(q^{m/4}).2$, $G\cap X_0=\O_{2m}^-(q)$,    
		and $G_\a\cap X_0= D_{2(q^{m/2}-1)}$, where $m\equiv 4\pmod{8}$ and $f$ is odd.}\ \\
	As usual,  $X_{\a\b}$ does not contain $\Sp_4(q^{m/4})'$ and $|X_\a:X_{\a\b}|$ divides $|G_\a|$, which divides $2(q^{m/2}-1).f$, whereas (since $m\geq4$, and $(2m,q)\ne (8,2)$) by   \cite[Table 5.2.A]{KL}, $|X_\a:X_{\a\b}|\geq (q^m-1)/(q^{m/4}-1)$,  giving a contradiction.
	
	\medskip
	\noindent
	\emph{Line 8\,a: $q$ even, \  $X_\a\cap X_0=\O_{2m}^+(q).2$, $G\cap X_0=\Sz(q^{m/2})$,   
		and $G_\a\cap X_0= D_{2(q^{m/2}-1)}$, where $m\equiv 2\pmod{4}$ and $f$ is odd, $f>1$.}\ \\
	If $m\geq 4$, we obtain a contradiction by an exactly similar argument to that for Line 7\,a, so we may assume that $m=2$, and hence that $q>8$ (as $q$ is even). Then $X_\a\cap X_0= O_4^+(q)=\Sp_2(q^{m/2})\wr S_2$ with base group $T_1\times T_2$ (where each $T_i=\Sp_2(q)$), as in Line 6\,b. Arguing as for Line 6\,b,  for some $i$, $X_{\a\b}\cap T_i$ is a proper subgroup of $T_i$ with order divisible by $rq/2$, where $r$ is a primitive prime divisor of $2^{2f}-1$ (note that $r$ exists since $q>8$), and this implies that $q/2$ divides $2f$. This in turn implies that $q=16$, and hence that $r=17$. However, $\Sp_2(16)$ has no proper subgroup of order divisible by $17.8$,  and we have a contradiction. 
	
	\medskip
	\noindent
	\emph{Line 8\,b: $q$ even, \  $X_\a\cap X_0=\, ^2\!B_2(q^{m/2})$, $G\cap X_0=\O_{2m}^+(q).2$,   
		and $G_\a\cap X_0= D_{2(q^{m/2}-1)}$, where $m\equiv 2\pmod{4}$ and $f$ is odd, $f>1$.}\ \\
	In this case  $X_{\a\b}$ does not contain $^2\!B_2(q^{m/2})$ and $|X_\a:X_{\a\b}|$ divides $|G_\a|$, which divides $(q^{m/2}-1).f$. Also, see  \cite[Theorem 7.3.3]{BHRD}, $|X_\a:X_{\a\b}|\geq q^{m}+1$, which is strictly larger than $(q^{m/2}-1).f$,  a contradiction.
	
	\medskip
	\noindent
	\emph{Line 9\,a: $m=2, f=1$, \  $X_\a\cap X_0=\PSp_{2}(p^2).a$, $G\cap X_0=[p^{1+2}](p^2-1)$,   
		and $G_\a\cap X_0= p$ or $D_{2p}$, where  $a\mid 2$, and $X\geq X_0.2$ if $p\equiv 1\pmod{4}$.}\ \\
	Here  $p\geq5$, and by Hypothesis~\ref{h:map}(c), $X_{\a\b}$ does not contain $\Sp_2(p^{2})$. Also $|X_\a:X_{\a\b}|$ divides $|G_\a|$ which divides $4p$, whereas by \cite[Table 5.2.A]{KL}, $|X_\a:X_{\a\b}|\geq p^2+1$, which gives a contradiction.
	
	\medskip
	\noindent
	\emph{Line 9\,b: $m=2, f=1$, \  $X_\a\cap X_0=[p^{1+2}].(p^2-1)$, $G\cap X_0=\PSp_{2}(p^2).a$,   
		and $G_\a\cap X_0= p$ or $D_{2p}$, where  $a\mid 2$, and $X\geq X_0.2$ if $p\equiv 1\pmod{4}$.}\ \\
	Here $p\geq5$, and  $X_\a = QR$ with $Q=[p^{1+2}]$. Also $p^2(p^2-1)/2$ divides $|X_\a:G_\a|$ which divides $|X_{\a\b}|$. Thus $|X_{\a\b}\cap Q|\geq p^2$ and as $(p^2-1)/2$ acts irreducibly on $Q/Q'$ (since $p\geq5$), it follows that $X_{\a\b}$ contains $Q$, which is normal in $X_\a$, and we have a contradiction to Hypothesis~\ref{h:map}(c).
	
	\subsubsection{Line  in Table~\ref{families} with $X_0=\PSU_{2m}(q)$.}\label{sub:mapU} 
	Let $V=\F_{q^2}^{2m}$ be the natural module for $X$, where $m\geq2$ and  $q=p^f$ for a prime $p$ and $f\geq1$. There is only one Line to consider, and it has $q=p$ or $q=4$. Moreover, from our assumptions, if $m=2$ then $q\geq 4$, and in particular $\PSL_{2m-1}(q)$ is a nonabelian simple group. 
	
	\medskip
	\noindent
	\emph{Line 1\,a: $\SU_{2m-1}(q)\leq X_\a\cap X_0\leq N_1$,\ $G\cap X_0 = [q^{2m}].\frac{q^{2m}-1}{(q+ 1)d}.m$, and $G_\a\cap X_0 = [q]$, where $d=(2m,q+1)$ and $q=p$ or $q=4$.}\\
	The index $|X_{\a}:X_{\a\b}|$ divides $|G_\a|$ which divides $qf$. 
	On the other hand, by  Hypothesis~\ref{h:map}(c), $X_{\a\b}$ does not contain $\SU_{2m-1}(q)$, and  by \cite[Table 5.2.A]{KL}, therefore $|X_{\a}:X_{\a\b}|$ is much larger than $qf$.
	
	\medskip
	\noindent
	\emph{Line 1\,b: \ $X_\a\cap X_0 = [q^{2m}].\frac{q^{2m}-1}{(q+ 1)d}.m$, $\SU_{2m-1}(q)\leq G\cap X_0\leq N_1$,\ $G_\a\cap  X_0 = [q]$, where $d=(2m,q+1)$ and $q=p$ or $q=4$.}\\
	Let $Q=[q^{2m}]=O_p(X_\a)$. Since $|G_\a|\leq qf$ and $|X_\a:G_\a|$ divides $|X_{\a\b}|$, it follows that $f\cdot|X_{\a\b}|$ is divisible by $q^{2m-1}(q^{2m}-1)/(q+1)d$. This implies firstly that $Q\cap X_{\a\b}\ne 1$. It also implies that $X_{\a\b}$ acts irreducibly on $Q$, and hence that $Q\leq X_{\a\b}$, which contradicts Hypothesis~\ref{h:map}(c).
	
	\subsubsection{Line  in Table~\ref{families} with  $X_0=\POm_{2m}^+(q)$.} Let $V=\F_{q}^{2m}$ be the natural module for $X$. There is only one line to consider.  Moreover, from our assumptions, if $m=4$ then $q\geq 3$.
	
	\medskip
	\noindent
	\emph{Line 1\,a: $\O_{2m-1}(q)\leq X_\a\cap X_0$,\ $G\cap X_0 \leq [q^{m}].\frac{q^m-1}{d}.m$, $G_\a\cap X_0 = [q]$ or $[2q]$, where  $m\geq4$,  $q=p$ or $q=4$, and $d=(4,q^m-1)$.}\\
	This case gives a contradiction by the same argument as for Line 1a in Subsection~\ref{sub:mapU}, using \cite[Table 5.2.A]{KL} for $\POm_{2m-1}(q)$ with $m\geq4$.
	
	\medskip
	\noindent
	\emph{Line 1\,b: $X_\a\cap X_0 \leq [q^{m}].\frac{q^m-1}{d}.m$, $\O_{2m-1}(q)\leq G\cap X_0$,\ $G_\a\cap X_0 = [q]$ or $[2q]$, where  $m\geq4$,  $q=p$ or $q=4$, and $d=(4,q^m-1)$.}\\
	This case gives a contradiction by the same argument as for Line 1b in Subsection~\ref{sub:mapU}.

	\subsubsection{Lines  in Table~\ref{families} with $X_0=\O_{7}(q)$.}  \label{sub:mapOodd}
	There are two lines to consider, and in both lines  $q=3^f$ for some $f\geq1$.  
	
	\medskip
	\noindent
	\emph{Line 1\,a: $X_\a\cap X_0= \O_6^\epsilon(q)$,\ $G\cap X_0 = \SL_3^{-\epsilon}(q)$, $G_\a\cap X_0 = q^2-1$, with $\epsilon=\pm$.}\\
	The index $|X_{\a}:X_{\a\b}|$ divides $|G_\a|$ which divides $(q^2-1)f$. 
	On the other hand, by  Hypothesis~\ref{h:map}(c), $X_{\a\b}$ does not contain $\O_6^\epsilon(q)$, and  by \cite[Table 5.2.A]{KL}, therefore $|X_{\a}:X_{\a\b}|$ is much larger than $(q^2-1)f$.
	
	\medskip
	\noindent
	\emph{Line 1\,b: $X_\a\cap X_0 = \SL_3^{-\epsilon}(q)$, $G\cap X_0= \O_6^\epsilon(q)$,\ $G_\a\cap X_0 = q^2-1$, with $\epsilon=\pm$.}\\
	If $\epsilon=+$ then this case gives a contradiction by the same argument as for Line 1a, with $ \SL_3^{-\epsilon}(q)$ in place of $\O_6^\epsilon(q)$. If $\epsilon=-$, then
	$|X_{\a\b}|$ is divisible by  $|X_{\a}:G_{\a}|$, which is divisible by $q^3(q^3-1)$. However there is no proper subgroup of $\SL_3(q).2f$ with this property (see \cite[Tables 8.3,8.4]{BHRD}), and hence $X_{\a\b}$ contains $\SL_3(q)$, contradicting Hypothesis~\ref{h:map}(c).
	
	\medskip
	\noindent
	\emph{Line 2\,a: $X_\a\cap X_0= \O_6^+(q).a$,\ $G\cap X_0 = {}^2G_2(q)$, $G_\a\cap X_0 = q-1$ or $D_{2(q-1)}$, where $q=3^f$ with $f$ odd and $a\leq 2$.}\\
	This case gives a contradiction by the same argument as for Line 1a, with $2(q-1)f$ in place of $(q^2-1)f$.
	
	\medskip
	\noindent
	\emph{Line 2\,b: $X_\a\cap X_0 = {}^2G_2(q)$, $G\cap X_0= \O_6^+(q).a$,\ $G_\a\cap X_0 = q-1$ or $D_{2(q-1)}$,  where $q=3^f$ with $f$ odd and $a\leq 2$.}\\
	This case gives a contradiction by the same argument as for Line 1a, with $2(q-1)f$ in place of $(q^2-1)f$, using the fact that each proper subgroup of ${}^2G_2(q)$ has index at least $q^3+1$ for $q>3$, see \cite[Theorem C]{K88},  and at least $9$ for $f=3$ as ${}^2G_2(q)\cong\PSL_2(8)$.

	\subsubsection{Lines  in Table~\ref{families} with $X_0=G_{2}(q)$.} There are two lines to consider, and in both cases $p=3$. 
	
	\medskip
	\noindent
	\emph{Lines 1\,a and 1\,b: $X_\a\cap X_0= \SL_3^\epsilon(q)$,\ $G\cap X_0 = \SL_3^{-\epsilon}(q)$, $G_\a\cap X_0 = q^2-1$, with $\epsilon=\pm$.}\\
	The index $|X_{\a}:X_{\a\b}|$ divides $|G_\a|$ which divides $(q^2-1)f$. 
	On the other hand, by  Hypothesis~\ref{h:map}(c), $X_{\a\b}$ does not contain $ \SL_3^\epsilon(q)$. If $\epsilon=-$ then  by \cite[Table 5.2.A]{KL},  $|X_{\a}:X_{\a\b}|\geq q^3+1$ which is much larger than $(q^2-1)f$. So $\epsilon=+$. In this case the same argument as in Line 1b of Subsection~\ref{sub:mapOodd} gives a contradiction.
	
	\medskip
	\noindent
	\emph{Lines 2\,a and 2\,b:  $\{ X_\a\cap X_0, G\cap X_0\} = \{{}^2G_2(q), \SL_3(q).a\}$, $G_\a\cap X_0 = (q-1)$ or $D_{2(q-1)}$, with $q=3^f$, $f$ odd, and $a\leq 2$.}\\
	By Hypothesis~\ref{h:map}(c), $X_{\a\b}$ does not contain ${}^2G_2(q)$ or $\SL_3(q)$, and by  \cite[Theorem C]{K88} or \cite[Table 5.2.A]{KL},  $|X_{\a}:X_{\a\b}|$ is at least $q^3+1$ or $q^2+q+1$, respectively.  On the other hand $|X_{\a}:X_{\a\b}|$ divides $|G_\a|$ which is at most $2(q-1)$, and we have a contradiction.
	
	\medskip
	This completes the proof of Theorem~\ref{p:maps-families}.

	
	\section{Sporadic groups}\label{s:sporadic}
	
	\vspace{4mm} \noindent {\bf Proof of Theorems \ref{sporadic} and \ref{spormaps}}
	
	All factorisations of the sporadic simple groups and their automorphism groups are given in \cite{Giudici}. The simple sporadic groups that factorise are the Mathieu groups (apart from $M_{22}$), together with $J_2$, $HS$, $Ru$, $Suz$, $Fi_{22}$ and $Co_1$; in addition, there are factorisations of the almost simple groups $M_{22}.2$, $J_2.2$, $HS.2$, $Suz.2$ and $He.2$ that do not intersect the simple group in a factorisation. Given the results of \cite{Giudici}, for the largest groups $Fi_{22}$ and $Co_1$, we can read off from \cite{atlas} that the factorisations are not cyclic/dihedral. For the other groups, it is a routine matter to compute using Magma (as discussed in Subsection \ref{maggie}) the factorisations with cyclic or dihedral intersections, and to determine which of these satisfy Hypothesis \ref{h:map}.
	
	
	\section{Tables of cyclic/dihedral factorisations}\label{tables}

	This section contains the tables for the Factorisation Theorems \ref{t:Anfactns}, \ref{classgen} and \ref{sporadic}.
	
	There is some notation that we use throughout the tables:
	\begin{itemize}
		\item for $r$ a positive integer, $a_r$ and $b_r$ denote divisors of $r$
		\item as is standard group-theoretic notation, a positive integer $n$ denotes a cyclic group of order $n$, and the symbol $[n]$ denotes a (possibly non-cyclic) group of order $n$
		\item in Table \ref{families} we just give the possibilities for $(X_0,A\cap X_0,B\cap X_0)$; whereas in Tables \ref{excep-psl}-\ref{excep-sp}, we give the possibilities $(X,A,B)$ where $X$ is an almost simple group which is minimal subject to having a factorisation $X=AB$ with cyclic or dihedral intersection.
		\item in Table \ref{spor}, if the socle factorises as $X_0 = A_0B_0$, we do not include factorisations of $X = X_0.2$ of the form $X=AB$, where $A\cap X_0 = A_0, B\cap X_0 = B_0$.
	\end{itemize}
	
	\begin{table}[h!]
		\caption{Cyclic/dihedral factorisations of $A_n$ and $S_n$ }\label{tab:Anfactns}
		\begin{tabular}{|l|l|l|l|}
			\hline
			$k$ & $X$ & $A$ & Conditions on $B$ \\
			\hline
			2   &  $S_n$ or $A_n$    & $(S_{n-2}\times S_2)\cap X$ & $B^\Omega$ $2$-homog.,  $B<X$, $B_{\{x,y\}}$ as in Table~\ref{tab:B2hom} \\
			&  $S_n$ or $A_n$    & $(S_{n-2}\times 1)\cap X$ & $B^\Omega$ $2$-trans., $B<X$, $B_{xy}$ as in Table~\ref{tab:B2hom} \\
			&  $S_n$    & $(S_{n-2}\times S_2)\cap A_{n}$ & $B$ contains an odd perm., $(B\cap A_n)^\Omega$ $2$-homog., \\
			&           &                               & $(B\cap A_n)_{\{x,y\}}$ as in Table~\ref{tab:B2hom} \\
			&  $S_n$    & $A_{n-2}\times 1$ &  $B$ contains an odd perm.,  $(B\cap A_n)^\Omega$ $2$-trans.,\\
			&           &           &$(B\cap A_n)_{xy}$ as in Table~\ref{tab:B2hom} \\
			&  $S_n$    & $A_{n-2}\times S_2$ & $(B\cap A_n)^\Omega$ $2$-trans., $(B\cap A_n)_{xy}$ as in Table~\ref{tab:B2hom} \\
			\hline\hline
			3   &  $S_n$ or $A_n$   & $(S_{n-3}\times S)\cap X$ &  $\PSL_2(q)\leq B\leq \hatB :=\PGaL_2(q)$  with $n=q+1$,  \\
			& & & $S\le S_3$; see Remark~\ref{r:Anfactns}(a)\\
			&  $S_{12}$ or $A_{12}$   & $(S_{9}\times 1)\cap X$ &  $B=M_{11}$, $A\cap B=B_{xyz}=D_6$ \\
			&  $S_{8}$ or $A_8$   & $(S_{5}\times S)\cap X$ &  $B=\AGL_3(2)$ and $A\cap B=D_{8}, 2^2$  for $S=S_2, 1$, resp.\\
			&  $S_{8}$ or $A_8$   & $(S_{5}\times S_3)\cap X$ &  $B=\AGL_1(8)$ and $A\cap B=B_{\{x,y, z\}}=1$ \\
			&  $S_8$ or $A_8$   & $(S_{5}\times S)\cap X$ &  $B=\AGaL_1(8)$ and $A\cap B=C_3, 1$  for $S=S_3, S_2$, resp.\\
			&  $S_{32}$ or $A_{32}$   & $(S_{29}\times S_3)\cap X$ &  $B=\AGaL_1(32)$ and $A\cap B=B_{\{x,y, z\}}=1$ \\
			\hline\hline
			4   &  $S_{11}$ or $A_{11}$   & $(S_{7}\times S)\cap X$ &  $B=M_{11}$, $B_{xyzw}=1$, and $A\cap B=S\leq S_4$ cyclic/dihedral \\
			&  $S_9$ or $A_9$   & $(S_{5}\times S)\cap X$ &  $B=\PGL_2(8)$ and $(S,A\cap B)=(S_4,2^2)$  or $(S_3,1)$ \\
			&  $S_9$ or $A_9$   & $(S_{5}\times S)\cap X$ &  $B=\PGaL_2(8)$ and $(S,A\cap B)$ as in Remark~\ref{r:Anfactns}(b)\\
			&  $S_{33}$ or $A_{33}$   & $(S_{29}\times S)\cap X$ &  $B=\PGaL_2(32)$  and $(S,A\cap B)=(S_4,2^2)$  or $(S_3,1)$ \\
			\hline
			\hline
			5   &  $S_{12}$ or $A_{12}$   & $(S_{7}\times S)\cap X$ &  $B=M_{12}$, $B_{xyzwv}=1$, $A\cap B=S\leq S_5$ cyclic/dihedral\\ 
			\hline
		\end{tabular}
	\end{table}
	
	\vspace{1cm}
	
	\begin{table}[h!]
		\caption{$2$-homogeneous groups $B$ with $B_{xy}$ or $B_{\{x,y\}}$ cyclic or dihedral: notation of 
			Lemma \ref{lem:k2-B2hom}}\label{tab:B2hom}
		\begin{tabular}{|c|llll|}
			\hline
			$n$ & $B^\circ$ & $B^\circ_{xy}$ & $B^\circ_{\{x,y\}}$ & $\exists$ odd perm.  \\ \hline
			$q^2+1$ &  $^2\!B_2(q),\,q=2^{2a+1}$ & $q-1$& $D_{2(q-1)}$ & $\times$  \\
			$q^3+1$  &  $^2\!G_2(q),\,q=3^{2a+1}$ & $q-1$& $D_{2(q-1)}$ & $\times$  \\  
			$q^3+1$  &  $\PSU_3(q)$ & $(q^2-1)/(3,q+1)$  &  $D_{2(q^2-1)/u}$ & $\times$  \\  
			$q+1$  &  $\PSL_2(q)$ & $(q-1)/(2,q-1)$  &  $D_{2(q-1)/u}$ & $\checkmark$ iff $q$ odd  \\  
			%
			$p^2$  &  $\ASL_2(p)$ & $p$  &  $D_{2p}$  & $\checkmark$ iff $p\equiv 3\,(4)$ \\
			$q$  &  $B\leq \AGaL_1(q)$ & $B_{xy}\leq \langle \sigma\rangle$  &  $B_{\{x,y\}}\leq  $ & 
			$\checkmark$ iff $q$ odd    \\ 
			&& ($\sigma$ field aut.) & $\langle \sigma\rangle\times 2$ &\\ 
			\hline
			$7$  & $\PSL_2(7)$ & $2^2$ & $D_8$ & $\times$  \\
			$11$ & $\PSL_2(11)$ & $D_{6}$ & $D_{12}$ & $\times$  \\
			$3^4$ &$3^4.2^{1+4}.5$& 1 & 2 & $\times$   \\
			& $3^4.2A_5.2$ & 3 & 6 & $\times$  \\
			$3^6$ & $3^6.\SL_2(13)$ & 3 & 6 & $\times$  \\
			$5^2$ &$5^2.2A_4$& $1$ & $2$ & $\checkmark$  \\
			$7^2$ &$7^2.2S_4$ & 1 & 2 & $\times$  \\
			$11^2$ &$11^2.(2A_4\times 5)$ & 1 & 2 & $\checkmark$  \\
			& $11^2.2A_5$ & 1 & 2 & $\times$  \\
			$19^2$ &$19^2.(2A_5 \times 9)$ & 3 & 6 & $\times$  \\
			$23^2$ &$23^2.(2S_4 \times 11)$ & 1 & 2 & $\times$  \\
			$29^2$ &$29^2.(2A_5 \times 7)$ & 1 & 2 & $\times$  \\
			$59^2$ & $59^2.(2A_5 \times 29)$ & 1 & 2 & $\times$  \\   
			\hline
		\end{tabular}
	\end{table}
	
	
	\begin{table}[h!]
		\caption{Cyclic/dihedral factorisations of groups of Lie type: families}\label{families}
		\[
		\begin{array}{|l|l|l|l|l|}
			\hline
			X_0 & A\cap X_0 & B \cap X_0 & A \cap B \cap X_0 & \hbox{Conditions}  \\
			\hline
			\PSL_n(q), & A \triangleright q^{n-1}.\SL_{n-1}(q) & \le \frac{q^n-1}{(q-1)d}.n & r & r|n,\,d=(n,q-1) \\
			n\ge 2 &&&& \\
			\PSL_4(p) & \PSp_4(p).a & p^3.\frac{p^3-1}{d}.e & p\times e & p\ge 3\hbox{ prime},\,e|3, \\
			& & & & d=(4,p-1),\,ad=4, \\
			& & & & X \ge X_0.2 \hbox{ if } p\equiv 1 \,(4) \\
			\hline
			\Sp_{2m}(q), & \Or_{2m}^\e (q) & \Sp_2(q^m) & D_{2(q^m-\e)} & \e = \pm 1 \\
			q\hbox{ even}, & \O_{2m}^\e(q) & \Sp_2(q^m) & q^m-\e & m \hbox{ odd}  \\
			m\ge 2         & \O_{2m}^-(q) & \Sp_2(q^m).2 & D_{2(q^m+1)} & m \hbox{ even}, \,m_2=2  \\
			& \Or_{2m}^-(q) & q^m.(q^m-1).m' & 2m'& m'|m  \\
			& \O_{2m}^-(q) & q^m.(q^m-1).m' & m'& m'|m\,^*  \\
			& \O_{2m}^-(q) & \Sp_2(q^{m/2}) \wr S_2 & D_{2(q^m-1)} & \frac{m}{2}\hbox{ odd}  \\
			&  & \Sp_4(q^{m/4}).2 & D_{2(q^{m/2}-1)} & \frac{m}{4}\hbox{ odd}, \,q = 2^{2f+1}  \\
			& \Or_{2m}^+(q) & ^2\!B_2(q^{m/2}) & D_{2(q^{m/2}-1)} & \frac{m}{2}\hbox{ odd}  \\
			&&&& \\
			\PSp_4(p) & \PSp_2(p^2).a & p^{1+2}.(p^2-1) & p,\,D_{2p} & a|2,\, X \ge X_0.2 \hbox{ if } p\equiv 1 \,(4) \\
			\hline
			\PSU_{2m}(q),& A \triangleright \SU_{2m-1}(q) & \le q^{2m}.\frac{q^{2m}-1}{(q+1)d}.m & [q] & q=p \hbox{ or }q=4,  \\
			m\ge 2 &&&& d=(2m,q+1) \\
			\hline
			P\O_{2m}^+(q), & A \triangleright \O_{2m-1}(q) & \le q^{m}.\frac{q^m-1}{d}.m & [q],\,[2q] & q=p \hbox{ or }q=4,  \\
			m \ge 4 &&&& d=(4,q^m-1) \\
			\hline
			\O_7(q) & \O_6^\e(q) & SL_3^{-\e}(q) & q^2-1 & q=3^f, \,\e = \pm 1 \\
			& \O_6^+(q).a & ^2\!G_2(q) & q-1,\,D_{2(q-1)} & q=3^{2f+1},\,a|2  \\
			\hline
			G_2(q) & SL_3(q) & SU_3(q) & q^2-1 & q=3^f \\
			& ^2\!G_2(q) & SL_3(q).a & q-1,\,D_{2(q-1)}  & q=3^{2f+1},\,a|2 \\
			\hline
		\end{array}
		\]
	\end{table}

	\begin{table*}[h!]
		$*$  Only know an example if  $(m')_2 = m_2$, where $m_2$ denotes the $2$-part of $m$
	\end{table*}

	\begin{table}[h!]
		\caption{Exceptional cyclic/dihedral factorisations, $X_0 = \PSL_n(q)$}\label{excep-psl}
		\[
		\begin{array}{|l|l|l|l|}
			\hline
			X & A & B & A\cap B   \\
			\hline
			\PSL_2(7) & S_4 & 7,\,7.3 & 1,\,3 \\
			\PSL_2(11) & A_5 & 11,\,11.5 & 1,\,5 \\
			& A_4 & 11.5 & 1 \\
			\PSL_2(16).4 & A_5.4 & D_{34}.4 & 2 \\
			\PSL_2(19) & A_5 & 19.9 & 3 \\
			\PSL_2(23) & S_4 & 23.11 & 1 \\
			\PSL_2(29) & A_5 & 29.7,\,29.14 & 1,\,2 \\
			\PSL_2(59) & A_5 & 59.29 & 1 \\
			\hline
			\PSL_3(3) & 13.3 & 3^2.8 & 1 \\
			\PSL_3(4).2 & A_6.2 & \PSL_3(2).2 & D_6 \\
			\PSL_3(4).S_3 & 7.3.S_3 & 2^4.(3\times D_{10}).2 & 1 \\
			\PSL_3(8).3 & 73.9 & 2^{3+6}.7^2.3 & 1 \\
			\PSL_4(3) & S_5,\,4.A_5,\,4.S_5 & 3^3.\SL_3(3) & 3,\,D_6,\,D_{12} \\
			& 2^4.5.a_4 & 3^3.\SL_3(3) & 2,\,4,\,8 \\
			\PSL_4(3).2 & (4\times A_6).2^2 & 3^3.13.3.2 & 1 \\
			\SL_4(4) & \Sp_4(4) & 2^6.63 & 2^2 \\
			\SL_4(4).2 & (a_5\times \SL_2(16)).4 & \SL_3(4).[2a_3] & 1,\,3,\,5,\,15 \\
			& (5\times \SL_2(16)).4 & 2^6.63.6 & 1 \\
			\SL_5(2) & 31.5 & 2^6.(\SL_3(2) \times S_3) & 1 \\
			\SL_6(2) & G_2(2) & \SL_5(2) & D_6 \\
			\PSL_6(3) & \PSL_2(13) & 3^5.\SL_5(3) & 3 \\
			\hline
		\end{array}
		\]
	\end{table}
	
	\begin{table}[h!]
		\caption{Exceptional cyclic/dihedral factorisations, $X_0 = \PSU_n(q)$. Note: exceptional factorisations of $\PSU_4(2) \cong \PSp_4(3)$ given in Table \ref{excep-sp}}\label{excep-psu}
		\[
		\begin{array}{|l|l|l|l|}
			\hline
			X & A & B & A\cap B \\
			\hline
			\PSU_3(3) & \PSL_2(7) & 3^{1+2}.8 & D_6 \\
			\PSU_3(8).3^2 & 57.9 & 2^{3+6}.63.3 & 1 \\
			\PSU_4(3) & \PSL_3(4) & [2^3.3^4] & 4 \\
			\PSU_4(3).2 & \PSL_3(4).2 & [2.3^4.a_8] & 1,\,2,\,2^2,\,D_8 \\
			&    & [2.3^5.a_2] & 3,\,D_6 \\
			&    & 3^4.[10],\,3^4.[20] & 5,\,D_{10} \\
			\PSU_4(3).[4] & \PSL_2(7).[4] & 3^4.A_6.[4] & D_6 \\
			\SU_4(4).4 & (a\times \SL_2(16).4,\,a\le 3 & (a_5 \times \SU_3(4)).4 & 1,\,2,\,3,\,5,\,15,\,D_{10} \\
			& (D_6 \times \SL_2(16)).4 & (a_5 \times \SU_3(4)).4 & D_6,\,D_{30} \\
			\SU_4(8).3 & 513.3.3 & 2^{12}.\SL_2(64).7.3 & 1 \\
			\hline
		\end{array}
		\]
	\end{table}

	\begin{table}[h!]
		\caption{Exceptional cyclic/dihedral factorisations, $X_0$ orthogonal}\label{excep-orthog}
		\[
		\begin{array}{|l|l|l|l|}
			\hline
			X & A & B & A\cap B \\
			\hline
			\O_7(3) & A_9,\,S_9 & 3^3.\SL_3(3) & D_6,\,D_{12} \\
			& \Sp_6(2) & [3^6].13.a_3 & 3,\,9 \\
			\hline
			\O_8^+(2) & \Sp_6(2) & A_5.[2a_4] & 1,\,2,\,2^2 \\
			&   & A_6.[a_4] & 3,\,D_6,\,D_{12} \\
			&   & [2^4.15.a_4] \hbox{ (sol.)} & 2,\,4,\,8 \\
			&  A_9 & 2^4.A_5.[a_8] & 1,\,2,\,2^2,\,4,\,D_8 \\
			&      & 2^4.A_5.[3a_2] & 3,\,D_6 \\
			&      & [2^6.15.a_4] \hbox{ (sol.)} & 1,\,2,\,4 \\
			& A_8,\,S_8 & \SU_4(2),\,\SU_4(2).2 & 3,\,6,\, D_6,\,D_{12} \\
			P\O_8^+(3).a_2 & \O_7(3).b_2 & 3^4.[40c_4] & 3,\,6,\,12 \\
			\O_8^+(4).2 & \Sp_6(4).2 & \SL_2(16).4,\,2^2.\SL_2(16).4 & 1,\,2^2 \\
			& \O_8^-(2).2 & \O_6^-(4).4 & D_{12} \\
			\O_{12}^+(2) & \Sp_{10}(2) & G_2(2) & D_6 \\
			\hline
		\end{array}
		\]
	\end{table}

	\begin{table}[h!]
		\caption{Exceptional cyclic/dihedral factorisations, $X_0$ symplectic}\label{excep-sp}
		\[
		\begin{array}{|l|l|l|l|}
			\hline
			X & A & B & A\cap B \\
			\hline
			\PSp_4(3) & 2^4.A_5 & [3^3a_{12}] & a_{12},\,D_{a_{12}} \\
			& S_5 & 3^{1+2}.Q_8.a_3 & a_3 \\
			& A_6,\,S_6 & 3^{1+2}.Q_8 & 3,\,D_6 \\
			& 2^4.5.a_2 & 3^{1+2}.2A_4 & 2,\,4 \\
			\PSp_4(3).2 & 2^4.5.4 & [3^3].S_3 & 1 \\
			&         & 3^3.A_4.2 & 2^2 \\
			\Sp_4(4) & \SL_2(16) & 2^6.15 & 2^2 \\
			\Sp_4(4).2 & \SL_2(16).4 & A_5.[2a_2] & a_2 \\
			&    & A_6.[2a_2] & D_{6a_2} \\
			\PSp_4(5) & \PSL_2(25).a_2 & 5^{1+2}.Q_8.3b_2 & 5,\,D_{10},\,D_{20} \\
			\PSp_4(7) & \PSL_2(49).a_2 & 7^{1+2}.[48] & 7,\,D_{14} \\
			\PSp_4(11) & \PSL_2(121).a_2 & 11^{1+2}.10A_4 & 11,\,D_{22} \\
			&    & 11^{1+2}.2A_5 & 11,\,D_{22} \\
			\PSp_4(23) & \PSL_2(23^2).a_2 & 23^{1+2}.22S_4 & 23,\,D_{46} \\
			\PSp_4(29) & \PSL_2(29^2).a_2 & 29^{1+2}.14b_2.A_5 & 29,\,D_{58} \\
			\PSp_4(59) & \PSL_2(59^2).a_2 & 59^{1+2}.(2A_5 \times 29) & 59,\,D_{118} \\
			\Sp_6(2) & S_5.[a_6] & \SU_3(3).b_2 & 1,\,2,\,3,\,D_6 \\
			& A_6.[a_4] & \SU_3(3).b_2 & 3,\,6,\,D_6,\,D_{12} \\
			& \SL_3(2).a_2 & \SU_4(2).b_2 & D_6,\,D_{12} \\
			& 2^4.A_5.[a_4] & \SL_2(8).3 & 1,\,2,\,2^2 \\
			& 2^4.A_5.a_2 & \SU_3(3) & 4,\,8 \\
			& 2^5.A_6 & \SL_2(8) & 2^2 \\
			& A_8.a_2 & [3^3].[8b_2] & D_6,\,D_{12} \\
			\PSp_6(3) & \PSL_2(13) & P_1 & 3 \\
			& \PSL_2(27).3 & 3^{1+4}.2^{1+4}.A_5 & 3 \\
			\PSp_6(3).2 & \PSL_2(27).6 &  3^{1+4}.2^{1+4}.D_{10}.2 & 1 \\
			\Sp_6(4).2 & \Or_6^-(4).2 & \SU_3(3).2 & D_6 \\
			& G_2(4).2 & \SL_2(16).4 & 1 \\
			\Sp_8(2) & \Or_8^+(2) & \PSL_2(17) & D_{18} \\
			& \O_8^-(2) & A_6.2^2 & D_6 \\
			& \Or_8^-(2) & A_5.[2a_4] & 1,\,2,\,2^2 \\
			&     & A_5.[a],\,a=6,12,20 & 3,\,D_6,\,D_{10} \\
			&     & A_6.[a_4] & 3,\,6,\,D_6,\,D_{12} \\
			&     & 2^4.A_5.a_2 & D_8,\,D_{16} \\
			\Sp_{12}(2) & \Or_{12}^-(2) & \SU_3(3).2 & D_6 \\
			\hline
		\end{array}
		\]
	\end{table}
	
	
	\begin{table}[h!]
		\caption{Cyclic/dihedral factorisations of sporadic groups}\label{spor}
		\[
		\begin{array}{|l|l|l|l|}
			\hline
			X & A & B & A\cap B   \\
			\hline
			M_{11} & M_{10} & 11,\,11.5 & 1,\,5 \\
			& M_9.2 & \PSL_2(11),\,11.5 & D_{12},\,1 \\
			&M_9,\,9.8 & \PSL_2(11) & D_6 \\
			\hline
			M_{12} & M_{11} & \hbox{many poss.} & C_i\,(1\le i \le 6), \\
			&   &  & D_{2j}\, (3\le j\le 6) \\
			& \PSL_2(11) \,\hbox{(max.)} & M_{10},\,M_{10}.2 & 5,\,D_{10} \\
			&      & M_9.2,\,M_9,S_3 & 1,\,3 \\
			&      & S_6 & 5 \\
			\hline
			M_{23} & M_{22} & 23,\,23.11 & 1,\,11 \\
			& M_{21} & 23.11 & 1 \\
			& 2^4.A_7 & 23.11 & 1 \\
			\hline
			M_{24} & M_{23} & \hbox{many poss.} & C_i\,(1\le i \le 8), \\
			&   &  & D_{2j}\, (j=2,4,5,6) \\
			& \PSL_2(23) & M_{22},\,M_{22}.2 & 11,\,D_{22} \\
			&    & \PSL_3(4).2,\,\PSL_3(4).S_3 & 1,\,3 \\
			&    & 2^4.A_7,\,2^4.A_8 & 1,\,D_8 \\
			\hline
			J_2 & \PSU_3(3) & A_5\times 5,\,A_5 \times D_{10} & 3,\,6 \\
			J_2.2 & \PSU_3(3).2 & 5^2.[4a_6],\,(5\times A_5).[a_4] & 1,\,2,\,3,\,6,\,12 \\
			\hline
			HS & M_{22} & A_5 \times 5,\,A_5 \times D_{10} & 3,\,6 \\
			HS.2 & M_{22}.2 & 5^2.[4a],\,a=1,2,4,5 & 1,\,2,\,4,\,5 \\
			&     & 5^{1+2}.(2\times 4) & D_{10} \\
			&     & (5\times A_5).[a_4] & 3,\,6,\,12,\,D_6,\,D_{12} \\
			\hline
			He & \Sp_4(4).a_2 & 7^2.\SL_2(7) & 4,\,8 \\
			He.2 & \Sp_4(4).4 & 7^{1+2}.[6a_6] & 1,\,2,\,3,\,D_6 \\
			&    & 7^2.\SL_2(7).a_2 & 8,\,16 \\
			\hline
			Ru & ^2\!F_4(2).2 & \PSL_2(29) & 3 \\
			\hline
			Suz.2 & G_2(4).2 & 3^5.11.2a_5 & 3,\,15 \\
			\hline
		\end{array}
		\]
	\end{table}


\clearpage
	
	\begin{thebibliography}{12}
		
		
		\bibitem{B71}
		N. Biggs, Classification of complete maps on orientable surfaces, {\it Rend. Mat.} (6) {\bf 4} (1971),
		645--655.
		
		\bibitem{magma}
		W.\ Bosma, J.\ Cannon and  C.\ Playoust,
		The {\sc Magma} algebra system I: The user language,
		\emph{J.\ Symbolic Comput.} {\bf 24} (1997), 235--265.
		
		\bibitem{BHRD} J.N. Bray, D.F. Holt and C.M. Roney-Dougal, The maximal subgroups of the low-dimensional finite classical groups, London Mathematical Society Lecture Note Series, vol. 407,
		Cambridge University Press, Cambridge, 2013.
		
		\bibitem{BL}
		T. C. Burness and C. H. Li, 
		On solvable factors of almost simple groups,
		{\it Adv. Math.} {\bf 377} (2021), Paper No. 107499.
		
		\bibitem{CCDKNW}
		D. A. Catalano, M. D. E. Conder, S. F. Du, Y. S. Kwon, R. Nedela and S. Wilson, 
		Classification of regular embeddings of $n$-dimensional cubes, 
		\emph{J. Alg. Comb.} {\bf 33} (2011), 215--238.
		
		
		\bibitem{atlas} J. H. Conway, R. T. Curtis, S. P. Norton, R. A. Parker and R. A. Wilson, {\it Atlas of finite groups}, Oxford University Press, 1985.
		
		
		\bibitem{DM}
		J. D. Dixon and B. Mortimer, {\it Permutation groups}, Graduate Texts in Math. {\bf 163}, Springer, 1996.
		
		\bibitem{GNSS}
		A. Gardiner, R. Nedela, J. \v Sir\'a\v n and M. \v Skoviera, Characterisation of graphs which underlie regular maps on closed surfaces, \emph{J. London Math. Soc.} {\bf 59} (1999), 100--108.
		
		
		\bibitem{Giudici}
		M. Giudici, Factorisations of sporadic simple groups, 
		{\it J. Algebra} {\bf 304} (2006), 311--323.
		
		
		\bibitem{GR}
		C. Godsil and G.F. Royle, \emph{Algebraic Graph Theory}, Springer-Verlag, New York, 2001. 
		
		\bibitem{GW}
		J. E. Graver and M. E. Watkins, Locally finite, planar, edge-transitive graphs, Mem. Amer. Math.
		Soc. {\bf 126} (1997), no. 601.
		
		\bibitem{G83}
		R. Guralnick, Subgroups of prime power index in a simple group, \emph{J. Algebra} {\bf 81} (1983), 
		304--311.
		
		
		\bibitem{HLS}
		C. Hering, M. W. Liebeck and J. Saxl,
		The factorizations of the finite exceptional groups of Lie type, \emph{J. Algebra} {\bf 106} (1987), 517--527.
		
		
		\bibitem{H67}
		B. Huppert, \emph{Endliche Gruppen I}, Springer-Verlag, Berlin, Heidelberg, New York, 1967.
		
		\bibitem{I90}
		N. F. J. Inglis,
		The embedding $O(2m,2^k)\leq \Sp(2m,2^k)$, {\it Arch. Math.} {\bf 54} (1990), 327--330.
		
		\bibitem{Ito}
		N. Ito, 
		On doubly transitive groups of degree $n$ and order $2(n-1)n$. {\it Nagoya Math. J.} {\bf 27} (1966), 409--417.
		
		\bibitem{Ja83}
		L. D. James, Imbeddings of the complete graph, {|it Ars Combin.} {\bf 16} (1983), 57--72.
		
		\bibitem{JJ85}
		L. D. James and G. A. Jones, Regular orientable imbeddings of complete graphs, {\it J. Combin.
			Theory Ser. B} {\bf 39} (1985), 353--367.
		
		\bibitem{J05}
		G. A. Jones, Automorphisms and regular embeddings of merged Johnson graphs, \emph{Eur. J. Combin.}
		{\bf 26} (2005), 417--435.
		
		
		\bibitem{J10}
		G. A. Jones, Regular embeddings of complete bipartite graphs: classification and enumeration. \emph{Proc. Lond. Math. Soc.}
		{\bf 101} (2010), 427--453.
		
		\bibitem{J15} 
		G. A. Jones, Bipartite graph embeddings, Riemann surfaces and Galois groups, \emph{Discrete Math.} {\bf 338} (2015), 1801--1813. 
		
		\bibitem{J21}
		G. A. Jones, Edge-transitive embeddings of complete graphs,
		{\it Art Discrete Appl. Math.} {\bf 4} (2021), no. 3, Paper No. 3.03, 12 pp.
		
		
		
		\bibitem{KLST}
		W. Kimmerle, R. Lyons, R. Sandling, and D. N. Teague, 
		Composition factors from the group ring and Artin's theorem on orders of simple groups,
		\emph{Proc. London Math. Soc.} {\bf 60} (1990),  89--122.
		
		\bibitem{K88}
		P. B. Kleidman. The maximal subgroups of the Chevalley groups $G_2(q)$ with $q$
		odd, the Ree groups ${}^2G_2(q)$, and their automorphism groups, {\it J. Algebra} {\bf 117}
		(1988), 30--71.
		
		\bibitem{KL}
		P. Kleidman and M. Liebeck, {\it The subgroup structure of the finite classical groups},
		London Math. Soc. Lecture Note Series, Vol. 129, Cambridge Univ. Press, 1990.
		
		\bibitem{LiPS21} 
		C.H. Li, C.E. Praeger and S.J. Song,
		Locally finite vertex-rotary maps and coset graphs with finite valency and finite edge multiplicity, 2024, arxiv: 2202.07100.
		
		\bibitem{LiPS24} 
		C.H. Li, C.E. Praeger and S.J. Song, Arc-transitive embeddings of graphs, preprint, 2024.
		
		\bibitem{LWX}
		C.H. Li, L. Wang and B. Xia, The factorizations of finite classical groups, 2024, arXiv: 2402.18373v2.
		
		\bibitem{LX19}
		C.H. Li and B. Xia,  Factorizations of almost simple groups with a factor having many nonsolvable composition factors, {\it J. Algebra} {\bf 528} (2019), 439--473.
		
		\bibitem{LX22}
		C.H. Li and B. Xia, Factorizations of almost simple
		groups with a solvable factor, and Cayley graphs of solvable groups, Mem. Amer. Math. Soc. {\bf 279}, 2022, No. 1375.
		
		\bibitem{lieb} M.W. Liebeck, The affine permutation groups of rank three, {\it Proc. London Math. Soc.}
		{\bf 54} (1987), 477-516.
		
		\bibitem{LPS90}
		M.W.Liebeck, C. E. Praeger, and J.Saxl. The maximal factorizations of the finite simple groups and their automorphism groups,  Mem. Amer. Math. Soc. {\bf 86}, 1990, No. 432.
		
		\bibitem{LPS10}
		M.W.Liebeck, C. E. Praeger, and J.Saxl. Regular subgroups of primitive permutation groups,  Mem. Amer. Math. Soc. {\bf 203}, 2010, No. 952.
		
		
		
		
		\bibitem{S01}
		D. Singerman, Unicellular dessins and a uniqueness theorem for Klein's Riemann surface of genus $3$, \emph{Bull. London Math. Soc.} {\bf 33} (2001), 701--710.
		
		\bibitem{STW}
		J. \v{S}ir\'{a}\v{n}, T. W. Tucker, and M. E. Watkins,
		Realizing finite edge-transitive orientable maps, \emph{J. Graph Theory} {\bf 37} (2001), 1--34.
		
		\bibitem{WW}
		J. Wiegold and A. G. Williamson, The factorisations of the alternating and symmetric groups, {\it Math. Z.} {\bf 175} (1980), 171--179.
		
		\bibitem{ZZ}
		K. Zsigmondy, Zur Theorie der Potenzreste, \emph{Monatshefte f\"ur Mathematik und Physik}  {\bf 3}, (1892), 265--284.
		
	\end{thebibliography}
\end{document}